\documentclass{article}
\usepackage{amssymb}
\usepackage{amsfonts}
\usepackage{url}
\usepackage{amsmath}

\newtheorem{theorem}{Theorem}

\newtheorem{corollary}{Corollary}

\newtheorem{definition}{Definition}

\newtheorem{lemma}{Lemma}

\newtheorem{proposition}{Proposition}
\newtheorem{remark}{Remark}

\newenvironment{proof}[1][Proof]{\textbf{#1.} }{\ \rule{0.5em}{0.5em}}

\begin{document}

\begin{center}
\bigskip {\LARGE Linearized Boltzmann collision operator for a mixture of
monatomic and polyatomic chemically reacting species } \bigskip

{\large Niclas Bernhoff\smallskip }

Department of Mathematics, Karlstad University, 65188 Karlstad, Sweden

niclas.bernhoff@kau.se
\end{center}

\textbf{Abstract:}{\small \ }At higher altitudes, for high temperature
gases, for instance near space shuttles moving at hypersonic speed, not only
mechanical collisions are affecting the gas flow, but also chemical
reactions have an impact on such hypersonic flows. In this work we insert
chemical reactions, in form of dissociations and recombinations
(associations), in an existing model for a mixture of mono- and polyatomic
(non-reacting) species. More general chemical reactions, e.g., bimolecular
ones, can be obtained by instant combinations of associations and
dissociations. The polyatomicity is modelled by a continuous internal energy
variable and the evolution of the gas is described by a Boltzmann equation.

In, e.g., the Chapman-Enskog process - and related half-space problems - the
linearized Boltzmann collision operator plays a central role. Here we extend
some important properties of the linearized Boltzmann collision operator to
the introduced model with chemical reactions. A compactness result, stating
that the linearized operator can be decomposed into a sum of a positive
multiplication operator - the collision frequency - and a compact integral
operator, is proven under reasonable assumptions on the collision kernel.
The strategy is to show that the terms of the integral operator are (at
least) uniform limits of Hilbert-Schmidt integral operators and therefore
also compact operators. Self-adjointness of the linearized operator is a
direct consequences. Moreover, bounds on - including coercivity of - the
collision frequency are obtained for hard sphere like, as well as hard
potentials with cutoff like models. As consequence, Fredholmness, as well
as, the domain, of the linearized collision operator are obtained

\section{Introduction\label{S1}}

At atmospheric reentry of space shuttles, or, higher altitudes flights at
hypersonic speed the vehicles excite the air around them to high
temperatures. At high altitudes the pressure is lower and therefore the air
is excited at lower temperatures than at the sea level \cite{HH-58}. In high
temperature gases, not only mechanical collisions affect the flow, but also
chemical reactions have an impact on such flows \cite{HH-58,Cercignani-00}.
Typical chemical reactions in air at high temperatures are dissociation of
oxygen, $\mathrm{O}_{2}\rightleftarrows 2\mathrm{O}$, and nitrogen, $\mathrm{%
N}_{2}\rightleftarrows 2\mathrm{N,}$ but at higher temperatures even
ionization of oxygen and nitrogen atoms, $\mathrm{O}\rightleftarrows \mathrm{%
O}^{+}+\mathrm{e}^{-}$and $\mathrm{N}\rightleftarrows \mathrm{N}^{+}+\mathrm{%
e}^{-}$, respectively \cite{HH-58}.

At high altitude the gas is rarefied and the Boltzmann equation is used to
describe the evolution of the gas flow, e.g., around a space shuttle in the
upper atmosphere during reentry \cite{BBBD-18,Cercignani-00}. Studies of the
main properties of the linearized Boltzmann collision operator are of great
importance in gaining increased knowledge about related problems, see, e.g., 
\cite{Cercignani-88, Cercignani-00} and references therein.

The linearized collision operator, obtained by considering deviations of an
equilibrium, or Maxwellian, distribution, can in a natural way be written as
a sum of a positive multiplication operator - the collision frequency - and
an integral operator $-K$. Compact properties of the integral operator $K$
(for angular cut-off kernels) are extensively studied for monatomic single
species, see, e.g., \cite{Gr-63, Dr-75, Glassey, Cercignani-88, LS-10}, and
more recently for monatomic multi-component mixtures \cite{BGPS-13,Be-23a}.
Extensions to polyatomic gases, where the polyatomicity is modeled by either
a discrete, or, a continuous internal energy variable for single species 
\cite{Be-23a,Be-23b, BST-24}, or, mixtures \cite{Be-24a,Be-24b} have very
recently been conducted. Compactness results are also recently obtained for
models of polyatomic single gases, with a continuous internal energy
variable, where the molecules undergo resonant collisions (for which
internal energy and kinetic energy, respectively, are conserved) \cite%
{BBS-23}. In this work \footnote{%
This study was inspired by a presentation by Laurent Desvillettes at Current
Trends in Kinetic Theory and Related Models, a conference in memory of
Giampiero Spiga, at Parma University in October 13-14, 2022.}, we extend the
results for mixtures of mono- and polyatomic (non-reacting) species in \cite%
{Be-24a}, where the polyatomicity is modeled by a continuous internal energy
variable, cf., \cite{DMS-05,BBBD-18, ACG-24}, to include chemical reactions,
in form of dissociations and recombinations (associations) \cite%
{Cercignani-00}. More general chemical reactions, e.g., bimolecular ones 
\cite{DMS-05}, can be obtained by instant combinations of associations and
dissociations.

Following the lines of \cite{Be-23a,Be-23b,Be-24a,Be-24b}, motivated by an
approach by Kogan in \cite[Sect. 2.8]{Kogan} for the monatomic single
species case, a probabilistic formulation of the collision operator is
considered as the starting point. With this approach, it is shown that the
integral operator $K$ can be written as a sum of Hilbert-Schmidt integral
operators and operators, which are uniform limits of Hilbert-Schmidt
integral operators (cf. Lemma $\ref{LGD}$ in Section $\ref{PT1}$) - and so
compactness of the integral operator $K$ follows. Self-adjointness of the
operator $K$ and the collision frequency, imply that the linearized
collision operator, as the sum of two self-adjoint operators whereof (at
least) one is bounded, is also self-adjoint.

For models corresponding to hard sphere models, as well as hard potentials
with cut off models, in the monatomic case, bounds on the collision
frequency are obtained. Then the collision frequency is coercive and, hence,
a Fredholm operator. The resultant Fredholmness - vital in the
Chapman-Enskog process - of the linearized operator, is due to that the set
of Fredholm operators is closed under addition with compact operators.
Unlike for hard potential models, the linearized operator is not Fredholm
for soft potential models, even in the monatomic case. The domain of
collision frequency - and, hence, of the linearized collision operator as
well - follows directly by the obtained bounds.

For hard sphere like models the linearized collision operator satisfies all
the properties of the general linear operator in the abstract half-space
problem considered in \cite{Be-23d}, and, hence, the general existence
results in \cite{Be-23d} apply.

The rest of the paper is organized as follows. In Section $\ref{S2}$, the
model considered is presented. The probabilistic formulation of the
collision operators considered and its relations to a more classical
parameterized expression are accounted for in Section $\ref{S2.1}$.\ Some
results for the collision operator in Section $\ref{S2.2}$ and the
linearized collision operator in Section $\ref{S2.3}$ are presented. Section 
$\ref{S3}$ is devoted to the main results of this paper, while the main
proofs are addressed in Sections $\ref{PT1}-\ref{PT2}$; a proof of
compactness of the integral operators $K$ is presented in Section $\ref{PT1}$%
, while a proof of the bounds on the collision frequency appears in Section $%
\ref{PT2}$.

\section{Model\label{S2}}

This section concerns the model considered. Probabilistic formulations of
the collision operators are considered, whose relations to more classical
formulations are accounted for. Properties of the models and corresponding
linearized collision operators are also discussed.

Consider a multicomponent mixture of $s$ species $a_{1},...,a_{s}$, with $%
s_{0}$ monatomic and $s_{1}:=s-s_{0}$ polyatomic species, and masses $%
m_{1},...,m_{s}$, respectively, and introduce the index sets%
\begin{eqnarray*}
\mathcal{I} &=&\left\{ 1,...,s\right\} \text{, }\mathcal{I}_{mono}=\left\{
\alpha ;\text{ }a_{\alpha }\text{ is monatomic}\right\} =\left\{
1,...,s_{0}\right\} \text{, and } \\
\mathcal{I}_{poly} &=&\left\{ \alpha ;\text{ }a_{\alpha }\text{ is polyatomic%
}\right\} =\left\{ s_{0}+1,...,s\right\} \text{.}
\end{eqnarray*}%
Here the polyatomicity is modeled by a continuous internal energy variable $%
I\in $ $\mathbb{R}_{+}$ \cite{BL-75}. Denote by $\varepsilon _{\alpha 0}$
the potential energy due to configuration of species $a_{\alpha }$ for $%
\alpha \in \mathcal{I}$. Moreover, $\delta ^{\left( 1\right) }=...=\delta
^{\left( s_{0}\right) }=2$, while $\delta ^{\left( \alpha \right) }$, with $%
\delta ^{\left( \alpha \right) }\geq 2$, denote the number of internal
degrees of freedom of the species for $\alpha \in \mathcal{I}_{poly}$.

The distribution functions are of the form $f=\left( f_{1},...,f_{s}\right) $%
, where $f_{\alpha }=f_{\alpha }\left( t,\mathbf{x},\mathbf{Z}\right) $,
with 
\begin{equation}
\mathbf{Z=Z}_{\alpha }:=\left\{ 
\begin{array}{l}
\text{ \ \ }\boldsymbol{\xi }\text{ \ \ \ \ \ for }\alpha \in \mathcal{I}%
_{mono}\text{ \ \ } \\ 
\left( \boldsymbol{\xi },I\right) \text{ \ \ for }\alpha \in \mathcal{I}%
_{poly}%
\end{array}%
\right. \text{,}  \label{z1}
\end{equation}%
$\left\{ t,I\right\} \subset \mathbb{R}_{+}$, $\mathbf{x}=\left(
x,y,z\right) \in \mathbb{R}^{3}$, and $\boldsymbol{\xi }=\left( \xi _{x},\xi
_{y},\xi _{z}\right) \in \mathbb{R}^{3}$, is the distribution function for
species $a_{\alpha }$.

Moreover, consider the real Hilbert space 
\begin{equation*}
\mathcal{\mathfrak{h}}:=\left( L^{2}\left( d\boldsymbol{\xi \,}\right)
\right) ^{s_{0}}\times \left( L^{2}\left( d\boldsymbol{\xi \,}dI\right)
\right) ^{s_{1}},
\end{equation*}%
with inner product%
\begin{equation*}
\left( f,g\right) =\sum_{\alpha =1}^{s_{0}}\int_{\mathbb{R}^{3}}f_{\alpha
}g_{\alpha }\,d\boldsymbol{\xi \,}+\sum_{\alpha =s_{0}+1}^{s}\int_{\mathbb{R}%
^{3}\times \mathbb{R}_{+}}f_{\alpha }g_{\alpha }\,d\boldsymbol{\xi \,}dI%
\text{ for }f,g\in \mathcal{\mathfrak{h}}\text{.}
\end{equation*}

The evolution of the distribution functions is (in the absence of external
forces) described by the (vector) Boltzmann equation%
\begin{equation}
\frac{\partial f}{\partial t}+\left( \boldsymbol{\xi }\cdot \nabla _{\mathbf{%
x}}\right) f=Q\left( f\right) \text{, with }Q\left( f\right) =Q_{mech}\left(
f,f\right) +Q_{chem}\left( f\right) \text{,}  \label{BE1}
\end{equation}%
where the (vector) collision operator $Q_{mech}=\left(
Q_{m1},...,Q_{ms}\right) $ is a quadratic bilinear operator, see \cite%
{ACG-24, Be-24a}, that accounts for the change of velocities and internal
energies of particles due to binary collisions (assuming that the gas is
rarefied, such that other collisions are negligible) and the transition
operator $Q_{chem}=\left( Q_{c1},...,Q_{cs}\right) $ accounts for the change
of velocities and internal energies under possible chemical reactions -
processes where two particles associate into one, or, vice versa, one
particle dissociates into two.

A chemical process (dissociation or association/recombination) $\left\{
a_{\gamma },a_{\varsigma }\right\} \longleftrightarrow \left\{ a_{\beta
}\right\} $; one given particle of species $a_{\beta }$ for $\beta \in 
\mathcal{I}_{poly}$ dissociating in two particles of species $a_{\gamma }$
and $a_{\varsigma }$ for $\left( \gamma ,\varsigma \right) \in \mathcal{I}%
^{2}$, or, vice versa, two given particles of species $a_{\gamma }$ and $%
a_{\varsigma }$ for $\left( \gamma ,\varsigma \right) \in \mathcal{I}^{2}$
associating in a particle of species $a_{\beta }$ for $\beta \in \mathcal{I}%
_{poly}$, can be represented by two post-/preprocess elements, each element
consisting of a microscopic velocity and possibly also an internal energy, $%
\mathbf{Z}^{\prime }$ and $\mathbf{Z}_{\ast }^{\prime }$, and one
corresponding post-/pre-process element, $\mathbf{Z}_{\ast }$, consisting of
a microscopic velocity and an internal energy. Denote the set of all indices 
$\left( \beta ,\gamma ,\varsigma \right) \in \mathcal{I}_{poly}\times 
\mathcal{I}^{2}$ of valid considered chemical processes $\left\{ a_{\gamma
},a_{\varsigma }\right\} \longleftrightarrow \left\{ a_{\beta }\right\} $ by 
$\mathcal{C}$. Assume below that $\left( \beta ,\gamma ,\varsigma \right)
\in \mathcal{C}$.

Due to mass, momentum, and total energy conservation, the following
relations have to be satisfied by the elements%
\begin{equation}
m_{\gamma }+m_{\zeta }=m_{\beta }  \label{CIm}
\end{equation}%
\begin{equation}
m_{\gamma }\boldsymbol{\xi }^{\prime }+m_{\zeta }\boldsymbol{\xi }_{\ast
}^{\prime }=m_{\beta }\boldsymbol{\xi }_{\ast }  \label{CIM}
\end{equation}%
\begin{equation}
m_{\gamma }\frac{\left\vert \boldsymbol{\xi }^{\prime }\right\vert ^{2}}{2}%
+I^{\prime }\mathbf{1}_{\gamma \in \mathcal{I}_{poly}}+\varepsilon _{\gamma
0}+m_{\zeta }\frac{\left\vert \boldsymbol{\xi }_{\ast }^{\prime }\right\vert
^{2}}{2}+I_{\ast }^{\prime }\mathbf{1}_{\zeta \in \mathcal{I}%
_{poly}}+\varepsilon _{\zeta 0}=m_{\beta }\frac{\left\vert \boldsymbol{\xi }%
_{\ast }\right\vert ^{2}}{2}+I_{\ast }+\varepsilon _{\beta 0}\text{.}
\label{CIE}
\end{equation}%
It follows that the center of mass velocity is%
\begin{equation*}
\boldsymbol{\xi }_{\ast }=\frac{m_{\gamma }\boldsymbol{\xi }^{\prime
}+m_{\zeta }\boldsymbol{\xi }_{\ast }^{\prime }}{m_{\beta }}=\frac{m_{\gamma
}\boldsymbol{\xi }^{\prime }+m_{\zeta }\boldsymbol{\xi }_{\ast }^{\prime }}{%
m_{\gamma }+m_{\zeta }}=\mathbf{G}_{\gamma \zeta }^{\prime }\text{,}
\end{equation*}%
while%
\begin{eqnarray*}
2\Delta _{\gamma \zeta }^{\beta }\mathcal{E} &=&2\left( \Delta I-\Delta
_{\gamma \zeta }^{\beta }\varepsilon _{0}\right) =m_{\gamma }\left\vert 
\boldsymbol{\xi }^{\prime }\right\vert ^{2}+m_{\zeta }\left\vert \boldsymbol{%
\xi }_{\ast }^{\prime }\right\vert ^{2}-m_{\beta }\left\vert \boldsymbol{\xi 
}_{\ast }\right\vert ^{2} \\
&=&\frac{1}{m_{\gamma }+m_{\zeta }}\left( \left( m_{\gamma }^{2}+m_{\gamma
}m_{\zeta }\right) \left\vert \boldsymbol{\xi }^{\prime }\right\vert
^{2}+\left( m_{\gamma }m_{\zeta }+m_{\zeta }^{2}\right) \left\vert 
\boldsymbol{\xi }_{\ast }^{\prime }\right\vert ^{2}-\left\vert m_{\gamma }%
\boldsymbol{\xi }^{\prime }+m_{\zeta }\boldsymbol{\xi }_{\ast }^{\prime
}\right\vert ^{2}\right) \\
&=&\frac{m_{\gamma }m_{\zeta }}{m_{\gamma }+m_{\zeta }}\left\vert 
\boldsymbol{\xi }^{\prime }-\boldsymbol{\xi }_{\ast }^{\prime }\right\vert
^{2}=\frac{m_{\gamma }m_{\zeta }}{m_{\gamma }+m_{\zeta }}\left\vert \mathbf{g%
}^{\prime }\right\vert ^{2}\text{, with }\mathbf{g}^{\prime }=\boldsymbol{%
\xi }^{\prime }-\boldsymbol{\xi }_{\ast }^{\prime }\text{,} \\
\Delta _{\gamma \zeta }^{\beta }\varepsilon _{0} &=&\varepsilon _{\gamma
0}+\varepsilon _{\zeta 0}-\varepsilon _{\beta 0}>0\text{, and }\Delta
I=I_{\ast }-I^{\prime }\mathbf{1}_{\gamma \in \mathcal{I}_{poly}}-I_{\ast
}^{\prime }\mathbf{1}_{\zeta \in \mathcal{I}_{poly}}\text{.}
\end{eqnarray*}%
The total energy in the center of mass frame is%
\begin{equation*}
E_{\gamma \zeta }^{\prime }:=\dfrac{m_{\gamma }m_{\zeta }}{2\left( m_{\gamma
}+m_{\zeta }\right) }\left\vert \mathbf{g}^{\prime }\right\vert
^{2}+I^{\prime }\mathbf{1}_{\gamma \in \mathcal{I}_{poly}}+I_{\ast }^{\prime
}\mathbf{1}_{\zeta \in \mathcal{I}_{poly}}+\varepsilon _{\gamma
0}+\varepsilon _{\zeta 0}=I_{\ast }+\varepsilon _{\beta 0}=:E_{\beta }\text{.%
}
\end{equation*}%
Furthermore, for any possible chemical process $\left\{ a_{\gamma
},a_{\varsigma }\right\} \longleftrightarrow \left\{ a_{\beta }\right\} $,
the following relation on the numbers of internal degrees of freedom is
assumed:%
\begin{equation}
\delta ^{\left( \gamma \right) }+\delta ^{\left( \zeta \right) }\geq \delta
^{\left( \beta \right) }+1\text{.}  \label{d1}
\end{equation}

On the other hand, a mechanical collision can, given two particles of
species $a_{\alpha }$ and $a_{\beta }$, $\left\{ \alpha ,\beta \right\}
\subset \left\{ 1,...,s\right\} $, respectively, be represented by two
pre-collisional elements $\mathbf{Z}$ and $\mathbf{Z}_{\ast }$, and two
corresponding post-collisional elements $\mathbf{Z}^{\prime }$ and $\mathbf{Z%
}_{\ast }^{\prime }$ \cite{Be-24a}. The notation for pre- and
post-collisional pairs may be interchanged as well. Due to momentum and
total energy conservation, the following relations have to be satisfied by
the elements%
\begin{equation*}
m_{\alpha }\boldsymbol{\xi }+m_{\beta }\boldsymbol{\xi }_{\ast }=m_{\alpha }%
\boldsymbol{\xi }^{\prime }+m_{\beta }\boldsymbol{\xi }_{\ast }^{\prime }
\end{equation*}%
\begin{eqnarray*}
&&m_{\alpha }\left\vert \boldsymbol{\xi }\right\vert ^{2}+m_{\beta
}\left\vert \boldsymbol{\xi }_{\ast }\right\vert ^{2}+I\mathbf{1}_{\alpha
\in \mathcal{I}_{poly}}+I_{\ast }\mathbf{1}_{\beta \in \mathcal{I}_{poly}} \\
&=&m_{\alpha }\left\vert \boldsymbol{\xi }^{\prime }\right\vert
^{2}+m_{\beta }\left\vert \boldsymbol{\xi }_{\ast }^{\prime }\right\vert
^{2}+I^{\prime }\mathbf{1}_{\alpha \in \mathcal{I}_{poly}}+I_{\ast }^{\prime
}\mathbf{1}_{\beta \in \mathcal{I}_{poly}}\text{.}
\end{eqnarray*}%
It follows that the center of mass velocity%
\begin{equation*}
\mathbf{G}_{\alpha \beta }=\frac{m_{\alpha }\boldsymbol{\xi }+m_{\beta }%
\boldsymbol{\xi }_{\ast }}{m_{\alpha }+m_{\beta }}=\frac{m_{\alpha }%
\boldsymbol{\xi }^{\prime }+m_{\beta }\boldsymbol{\xi }_{\ast }^{\prime }}{%
m_{\alpha }+m_{\beta }}=\mathbf{G}_{\alpha \beta }^{\prime }
\end{equation*}%
and the total energy in the center of mass frame%
\begin{eqnarray*}
E_{\alpha \beta }:= &&\dfrac{m_{\alpha }m_{\beta }}{2\left( m_{\alpha
}+m_{\beta }\right) }\left\vert \mathbf{g}\right\vert ^{2}+I\mathbf{1}%
_{\alpha \in \mathcal{I}_{poly}}+I_{\ast }\mathbf{1}_{\beta \in \mathcal{I}%
_{poly}} \\
&=&\dfrac{m_{\alpha }m_{\beta }}{2\left( m_{\alpha }+m_{\beta }\right) }%
\left\vert \mathbf{g}^{\prime }\right\vert ^{2}+I^{\prime }\mathbf{1}%
_{\alpha \in \mathcal{I}_{poly}}+I_{\ast }^{\prime }\mathbf{1}_{\beta \in 
\mathcal{I}_{poly}}=:E_{\alpha \beta }^{\prime }\text{,} \\
&&\text{where }\mathbf{g}=\boldsymbol{\xi }-\boldsymbol{\xi }_{\ast }\text{
and }\mathbf{g}^{\prime }=\boldsymbol{\xi }^{\prime }-\boldsymbol{\xi }%
_{\ast }^{\prime }\text{,}
\end{eqnarray*}%
are conserved.

\subsection{Collision operator\label{S2.1}}

The mechanical collision operator $Q_{mech}=\left( Q_{m1},...,Q_{ms}\right) $
\cite{Be-24a} has components that can be written in the following form 
\begin{eqnarray*}
&&Q_{m\alpha }(f,f)=\sum_{\beta =1}^{s}Q_{\alpha \beta }(f,f)=\sum_{\beta
=1}^{s}\int_{\mathcal{Z}_{\alpha }\times \mathcal{Z}_{\beta }^{2}}W_{\alpha
\beta }\Lambda _{\alpha \beta }(f)\,d\mathbf{Z}_{\ast }d\mathbf{Z}^{\prime }d%
\mathbf{Z}_{\ast }^{\prime }\text{, where } \\
&&\Lambda _{\alpha \beta }(f):=\frac{f_{\alpha }^{\prime }f_{\beta \ast
}^{\prime }}{\varphi _{\alpha }\left( I^{\prime }\right) \varphi _{\beta
}\left( I_{\ast }^{\prime }\right) }-\frac{f_{\alpha }f_{\beta \ast }}{%
\varphi _{\alpha }\left( I\right) \varphi _{\beta }\left( I_{\ast }\right) }%
\text{ and }\mathcal{Z}_{\gamma }=\left\{ 
\begin{array}{l}
\mathbb{R}^{3}\text{ if }\gamma \in \mathcal{I}_{mono} \\ 
\mathbb{R}^{3}\times \mathbb{R}_{+}\text{ if }\gamma \in \mathcal{I}_{poly}%
\end{array}%
\right. \text{.}
\end{eqnarray*}%
Here and below the abbreviations%
\begin{equation*}
f_{\alpha \ast }=f_{\alpha }\left( t,\mathbf{x},\mathbf{Z}_{\ast }\right) 
\text{, }f_{\alpha }^{\prime }=f_{\alpha }\left( t,\mathbf{x},\mathbf{Z}%
^{\prime }\right) \text{, and }f_{\alpha \ast }^{\prime }=f_{\alpha }\left(
t,\mathbf{x},\mathbf{Z}_{\ast }^{\prime }\right) 
\end{equation*}%
where $\mathbf{Z}_{\ast }$, $\mathbf{Z}^{\prime }$, and $\mathbf{Z}_{\ast
}^{\prime }$, are defined as the natural extensions of definition $\left( %
\ref{z1}\right) $, i.e. $\mathbf{Z_{\ast }=Z_{\ast \alpha }}=\left\{ 
\begin{array}{l}
\text{ \ \ }\boldsymbol{\xi }_{\ast }\text{ \ \ \ \ \ \ \ for }\alpha \in 
\mathcal{I}_{mono}\text{ \ \ } \\ 
\left( \boldsymbol{\xi }_{\ast },I_{\ast }\right) \text{ \ \ for }\alpha \in 
\mathcal{I}_{poly}%
\end{array}%
\right. $ etc., are used for $\alpha \in \mathcal{I}$. Moreover, the
degeneracies or renormalization weights $\varphi _{\alpha }=\varphi _{\alpha
}\left( I\right) $, $\alpha \in \mathcal{I}$, with $\varphi _{\alpha }=1$
for $\alpha \in \mathcal{I}_{mono}$, are\ positive functions for $I>0$. A
typical choice of the degeneracies is \cite{DPT-21, ACG-24, Be-24a} 
\begin{equation*}
\varphi _{\alpha }\left( I\right) =I^{\delta ^{\left( \alpha \right) }/2-1}%
\text{, }\alpha \in \mathcal{I}\text{,}
\end{equation*}%
where $\delta ^{\left( 1\right) }=...=\delta ^{\left( s_{0}\right) }=2$,
while $\delta ^{\left( \alpha \right) }$, with $\delta ^{\left( \alpha
\right) }\geq 2$, denote the number of internal degrees of freedom of the
species for $\alpha \in \mathcal{I}_{poly}$. Our main results below will be
proven for this particular choice of degeneracies.

Note that in the literature it is usual to use a slightly different setting 
\cite{BDLP-94,DMS-05,BBBD-18}, where already renormalized distribution
functions are considered, opting to consider a weighted measure - where the
renormalization weights appear as weights - with respect to $I$. However,
this is merely due to a different scaling of the distribution functions
considered.

The transition probabilities $W_{\alpha \beta }$ are of the form \cite%
{Be-24a}%
\begin{eqnarray*}
&&W_{\alpha \beta }=W_{\alpha \beta }(\mathbf{Z},\mathbf{Z}_{\ast
}\left\vert \mathbf{Z}^{\prime },\mathbf{Z}_{\ast }^{\prime }\right. ) \\
&=&\left( m_{\alpha }+m_{\beta }\right) ^{2}m_{\alpha }m_{\beta }\varphi
_{\alpha }\left( I\right) \varphi _{\beta }\left( I_{\ast }\right) \sigma
_{\alpha \beta }\frac{\left\vert \mathbf{g}\right\vert }{\left\vert \mathbf{g%
}^{\prime }\right\vert }\widehat{\mathbf{\delta }}_{1}\widehat{\mathbf{%
\delta }}_{3}\text{,} \\
&&\text{where }\sigma _{\alpha \beta }=\sigma _{\alpha \beta }\left(
\left\vert \mathbf{g}\right\vert ,\cos \theta ,I,I_{\ast },I^{\prime
},I_{\ast }^{\prime }\right) >0\text{ a.e., with} \\
&&\widehat{\mathbf{\delta }}_{1}=\mathbf{\delta }_{1}\left( \frac{1}{2}%
\left( m_{\alpha }\left\vert \boldsymbol{\xi }\right\vert ^{2}+m_{\beta
}\left\vert \boldsymbol{\xi }_{\ast }\right\vert ^{2}-m_{\alpha }\left\vert 
\boldsymbol{\xi }^{\prime }\right\vert ^{2}-m_{\beta }\left\vert \boldsymbol{%
\xi }_{\ast }^{\prime }\right\vert ^{2}\right) -\Delta I\right) \text{,} \\
&&\widehat{\mathbf{\delta }}_{3}=\mathbf{\delta }_{3}\left( m_{\alpha }%
\boldsymbol{\xi }+m_{\beta }\boldsymbol{\xi }_{\ast }-m_{\alpha }\boldsymbol{%
\xi }^{\prime }-m_{\beta }\boldsymbol{\xi }_{\ast }^{\prime }\right) \text{, 
}\cos \theta =\frac{\mathbf{g}\cdot \mathbf{g}^{\prime }}{\left\vert \mathbf{%
g}\right\vert \left\vert \mathbf{g}^{\prime }\right\vert }\text{,} \\
&&\mathbf{g}=\boldsymbol{\xi }-\boldsymbol{\xi }_{\ast }\!\text{, }\mathbf{g}%
^{\prime }=\boldsymbol{\xi }^{\prime }-\boldsymbol{\xi }_{\ast }^{\prime }\!%
\text{, and }\Delta I=\left( I^{\prime }-I\right) \mathbf{1}_{\alpha \in 
\mathcal{I}_{poly}}+\left( I_{\ast }^{\prime }-I_{\ast }\right) \mathbf{1}%
_{\beta \in \mathcal{I}_{poly}}\text{.}
\end{eqnarray*}%
Here $\mathbf{\delta }_{3}$ and $\mathbf{\delta }_{1}$ denote the Dirac's
delta function in $\mathbb{R}^{3}$ and $\mathbb{R}$, respectively - $%
\widehat{\mathbf{\delta }}_{1}$ and $\widehat{\mathbf{\delta }}_{3}$ taking
the conservation of momentum and total energy into account. Moreover, we
have chosen - even if this means not being completely consistent - to not
indicate the dependence of species in the notation, if - like for the energy
gap $\Delta I$ - the dependence is only due to if the species are monatomic
or polyatomic.

Furthermore, it is assumed that for $\left( \alpha ,\beta \right) \in 
\mathcal{I}^{2}$ the scattering cross sections $\sigma _{\alpha \beta }$
satisfy the microreversibility conditions%
\begin{eqnarray*}
&&\varphi _{\alpha }\left( I\right) \varphi _{\beta }\left( I_{\ast }\right)
\left\vert \mathbf{g}\right\vert ^{2}\sigma _{\alpha \beta }\left(
\left\vert \mathbf{g}\right\vert ,\left\vert \cos \theta \right\vert
,I,I_{\ast },I^{\prime },I_{\ast }^{\prime }\right)  \\
&=&\varphi _{\alpha }\left( I^{\prime }\right) \varphi _{\beta }\left(
I_{\ast }^{\prime }\right) \left\vert \mathbf{g}^{\prime }\right\vert
^{2}\sigma _{\alpha \beta }\left( \left\vert \mathbf{g}^{\prime }\right\vert
,\left\vert \cos \theta \right\vert ,I^{\prime },I_{\ast }^{\prime
},I,I_{\ast }\right) \text{.}
\end{eqnarray*}%
Furthermore, to obtain invariance of change of particles in a collision, it
is assumed that the scattering cross sections $\sigma _{\alpha \beta }$ for $%
\left( \alpha ,\beta \right) \in \mathcal{I}^{2}$, satisfy the symmetry
relations%
\begin{equation*}
\sigma _{\alpha \beta }\left( \left\vert \mathbf{g}\right\vert ,\cos \theta
,I,I_{\ast },I^{\prime },I_{\ast }^{\prime }\right) =\sigma _{\beta \alpha
}\left( \left\vert \mathbf{g}\right\vert ,\cos \theta ,I_{\ast },I,I_{\ast
}^{\prime },I^{\prime }\right) \text{,}
\end{equation*}%
while 
\begin{eqnarray*}
\sigma _{\alpha \alpha } &=&\sigma _{\alpha \alpha }\left( \left\vert 
\mathbf{g}\right\vert ,\left\vert \cos \theta \right\vert ,I,I_{\ast
},I^{\prime },I_{\ast }^{\prime }\right) =\sigma _{\alpha \alpha }\left(
\left\vert \mathbf{g}\right\vert ,\left\vert \cos \theta \right\vert
,I_{\ast },I,I^{\prime },I_{\ast }^{\prime }\right)  \\
&=&\sigma _{\alpha \alpha }\left( \left\vert \mathbf{g}\right\vert
,\left\vert \cos \theta \right\vert ,I_{\ast },I,I_{\ast }^{\prime
},I^{\prime }\right) \text{.}
\end{eqnarray*}%
Applying known properties of Dirac's delta function, the transition
probabilities may be transformed to \cite{Be-24a} 
\begin{eqnarray*}
&&W_{\alpha \beta }=W_{\alpha \beta }(\mathbf{Z},\mathbf{Z}_{\ast
}\left\vert \mathbf{Z}^{\prime },\mathbf{Z}_{\ast }^{\prime }\right. ) \\
&=&\varphi _{\alpha }\left( I\right) \varphi _{\beta }\left( I_{\ast
}\right) \sigma _{\alpha \beta }\frac{\left\vert \mathbf{g}\right\vert }{%
\left\vert \mathbf{g}^{\prime }\right\vert ^{2}}\mathbf{1}_{\left\vert 
\mathbf{g}\right\vert ^{2}>2\widetilde{\Delta }_{\alpha \beta }I}\mathbf{%
\delta }_{1}\left( \sqrt{\left\vert \mathbf{g}\right\vert ^{2}-2\widetilde{%
\Delta }_{\alpha \beta }I}-\left\vert \mathbf{g}^{\prime }\right\vert
\right)  \\
&&\times \mathbf{\delta }_{3}\left( \mathbf{G}_{\alpha \beta }-\mathbf{G}%
_{\alpha \beta }^{\prime }\right) \text{, with }\widetilde{\Delta }_{\alpha
\beta }I=\frac{m_{\alpha }+m_{\beta }}{m_{\alpha }m_{\beta }}\Delta I\text{.}
\end{eqnarray*}%
For more details about the mechanical collision operator, including more
familiar formulations, we refer to \cite{Be-24a}.

The chemical process operator $Q_{chem}=\left( Q_{c1},...,Q_{cs}\right) $
has components that can be written in the following form 
\begin{eqnarray*}
Q_{c\alpha }(f) &=&\sum_{\left( \beta ,\gamma ,\varsigma \right) \in 
\mathcal{C}}Q_{\gamma \zeta }^{\beta ,\alpha }(f)=\sum_{\left( \beta ,\gamma
,\varsigma \right) \in \mathcal{C}}\int_{\mathcal{Z}}W_{\gamma \zeta
}^{\beta }\Delta _{\gamma \zeta }^{\beta }\left( \alpha ,\mathbf{Z}\right)
\Lambda _{\gamma \zeta }^{\beta }(f)\,d\mathbf{Z}_{\ast }d\mathbf{Z}^{\prime
}d\mathbf{Z}_{\ast }^{\prime }\text{,} \\
\text{ where }\Lambda _{\gamma \zeta }^{\beta } &=&\frac{f_{\gamma }^{\prime
}f_{\zeta \ast }^{\prime }}{\varphi _{\gamma }\left( I^{\prime }\right)
\varphi _{\zeta }\left( I_{\ast }^{\prime }\right) }-\frac{f_{\beta \ast }}{%
\varphi _{\beta }\left( I_{\ast }\right) }\text{,} \\
\Delta _{\gamma \zeta }^{\beta }\left( \alpha ,\mathbf{Z}\right) &=&\delta
_{\alpha \beta }\mathbf{\delta }_{c}\left( \mathbf{Z}-\mathbf{Z_{\ast }}%
\right) -\delta _{\alpha \gamma }\mathbf{\delta }_{c}\left( \mathbf{Z}-%
\mathbf{Z}^{\prime }\right) -\delta _{\alpha \zeta }\mathbf{\delta }%
_{c}\left( \mathbf{Z}-\mathbf{Z_{\ast }^{\prime }}\right) \text{,} \\
\text{ and }\mathcal{Z} &\mathcal{=}&\mathcal{Z}_{\beta }\mathcal{\times Z}%
_{\gamma }\mathcal{\times Z}_{\zeta }\text{, with }\mathcal{Z}_{\alpha
}=\left\{ 
\begin{array}{l}
\mathbb{R}^{3}\text{ if }\alpha \in \mathcal{I}_{mono} \\ 
\mathbb{R}^{3}\times \mathbb{R}_{+}\text{ if }\alpha \in \mathcal{I}_{poly}%
\end{array}%
\right. \text{.}
\end{eqnarray*}%
Here $\mathbf{\delta }_{c}$ denotes the Dirac's delta function in $\mathbb{R}%
^{3}$ or $\mathbb{R}^{4}$ if $\alpha \in \mathcal{I}_{mono}$ or $\alpha \in 
\mathcal{I}_{poly}$, respectively.

For $\left( \beta ,\gamma ,\varsigma \right) \in \mathcal{C}$ introduce a
positive number - the energy of the transition state of the process - $%
K_{\gamma \zeta }^{\beta }$, such that 
\begin{equation*}
K_{\gamma \zeta }^{\beta }=K_{\zeta \gamma }^{\beta }\geq \varepsilon
_{\gamma 0}+\varepsilon _{\zeta 0}>\varepsilon _{\beta 0}>0\text{.}
\end{equation*}%
That is, for a chemical process $\left\{ a_{\gamma },a_{\varsigma }\right\}
\longleftrightarrow \left\{ a_{\beta }\right\} $ to take place 
\begin{equation}
E_{\gamma \zeta }^{\prime }=E_{\beta }=I_{\ast }+\varepsilon _{\beta 0}\geq
K_{\gamma \zeta }^{\beta }\geq \varepsilon _{\gamma 0}+\varepsilon _{\zeta
0}>0  \label{c1}
\end{equation}%
has to be satisfied. Note that by condition $\left( \ref{c1}\right) $ 
\begin{equation}
\text{ }I_{\ast }=E_{\beta }-\varepsilon _{\beta 0}\geq K_{\gamma \zeta
}^{\beta }-\varepsilon _{\beta 0}\geq \varepsilon _{\gamma 0}+\varepsilon
_{\zeta 0}-\varepsilon _{\beta 0}=\Delta _{\gamma \zeta }^{\beta
}\varepsilon _{0}>0\text{ }  \label{c2}
\end{equation}

The transition probabilities $W_{\gamma \zeta }^{\beta }$ are for $\left(
\beta ,\gamma ,\varsigma \right) \in \mathcal{C}$ of the form%
\begin{eqnarray}
&&W_{\gamma \zeta }^{\beta }=W_{\gamma \zeta }^{\beta }(\mathbf{Z}_{\ast },%
\mathbf{Z}^{\prime },\mathbf{Z}_{\ast }^{\prime })  \notag \\
&=&m_{\beta }^{2}m_{\gamma }m_{\zeta }\varphi _{\beta }\left( I_{\ast
}\right) \sigma _{\beta }^{\gamma \zeta }\widehat{\mathbf{\delta }}_{1}%
\widehat{\mathbf{\delta }}_{3}\mathbf{1}_{E_{\beta }\geq K_{\gamma \zeta
}^{\beta }}  \notag \\
&=&m_{\beta }^{2}m_{\gamma }m_{\zeta }\varphi _{\gamma }\left( I^{\prime
}\right) \varphi _{\zeta }\left( I_{\ast }^{\prime }\right) \sigma _{\gamma
\zeta }^{\beta }\widehat{\mathbf{\delta }}_{1}\widehat{\mathbf{\delta }}_{3}%
\mathbf{1}_{E_{\beta }\geq K_{\gamma \zeta }^{\beta }}\text{,}  \label{tp}
\end{eqnarray}%
where 
\begin{eqnarray*}
&&\sigma _{\beta }^{\gamma \zeta }=\sigma _{\beta }^{\gamma \zeta }\left(
\left\vert \mathbf{g}^{\prime }\right\vert ,\cos \theta ,I^{\prime },I_{\ast
}^{\prime }\right) >0\text{ and }\sigma _{\gamma \zeta }^{\beta }=\sigma
_{\gamma \zeta }^{\beta }\left( \left\vert \mathbf{g}^{\prime }\right\vert
,\cos \theta ,I_{\ast }\right) >0\text{ a.e., } \\
&&\text{with }\widehat{\mathbf{\delta }}_{1}=\mathbf{\delta }_{1}\left( 
\frac{1}{2}\left( m_{\gamma }\frac{\left\vert \boldsymbol{\xi }^{\prime
}\right\vert ^{2}}{2}+m_{\zeta }\frac{\left\vert \boldsymbol{\xi }_{\ast
}^{\prime }\right\vert ^{2}}{2}-m_{\beta }\frac{\left\vert \boldsymbol{\xi }%
_{\ast }\right\vert ^{2}}{2}\right) -\Delta _{\gamma \zeta }^{\beta }%
\mathcal{E}\right) \text{,} \\
&&\widehat{\mathbf{\delta }}_{3}=\mathbf{\delta }_{3}\left( m_{\gamma }%
\boldsymbol{\xi }^{\prime }+m_{\zeta }\boldsymbol{\xi }_{\ast }^{\prime
}-m_{\beta }\boldsymbol{\xi }_{\ast }\right) \text{, }\cos \theta =\frac{%
\boldsymbol{\xi }\cdot \mathbf{g}^{\prime }}{\left\vert \boldsymbol{\xi }%
\right\vert \left\vert \mathbf{g}^{\prime }\right\vert }\text{, }\mathbf{g}%
^{\prime }=\boldsymbol{\xi }^{\prime }-\boldsymbol{\xi }_{\ast }^{\prime }%
\text{,} \\
&&\Delta _{\gamma \zeta }^{\beta }\mathcal{E}=\Delta I-\Delta _{\gamma \zeta
}^{\beta }\varepsilon _{0}\text{, }\Delta _{\gamma \zeta }^{\beta
}\varepsilon _{0}=\varepsilon _{\gamma 0}+\varepsilon _{\zeta 0}-\varepsilon
_{\beta 0}\text{, and } \\
&&\Delta I=I_{\ast }-I^{\prime }\mathbf{1}_{\gamma \in \mathcal{I}%
_{poly}}-I_{\ast }^{\prime }\mathbf{1}_{\zeta \in \mathcal{I}_{poly}}\text{.}
\end{eqnarray*}%
Remind that $\mathbf{\delta }_{3}$ and $\mathbf{\delta }_{1}$ denote the
Dirac's delta function in $\mathbb{R}^{3}$ and $\mathbb{R}$, respectively - $%
\widehat{\mathbf{\delta }}_{1}$ and $\widehat{\mathbf{\delta }}_{3}$ taking
the conservation of momentum and total energy into account, while $m_{\gamma
}+m_{\zeta }=m_{\beta }$. Note that for $\gamma \in \mathcal{I}_{mono}$ the
scattering cross sections $\sigma _{\beta }^{\gamma \zeta }$ are independent
of $I^{\prime }$, while correspondingly, for $\zeta \in \mathcal{I}_{mono}$
the scattering cross sections $\sigma _{\beta }^{\gamma \zeta }$ are
independent of $I_{\ast }^{\prime }$.

Furthermore, it is assumed that the scattering cross sections $\sigma
_{\beta }^{\zeta \gamma }$ and $\sigma _{\gamma \zeta }^{\beta }$ for $%
\left( \beta ,\gamma ,\varsigma \right) \in \mathcal{C}$ satisfy the
microreversibility conditions%
\begin{equation}
\varphi _{\beta }\left( I_{\ast }\right) \sigma _{\beta }^{\gamma \zeta
}\left( \left\vert \mathbf{g}^{\prime }\right\vert ,\left\vert \cos \theta
\right\vert ,I^{\prime },I_{\ast }^{\prime }\right) =\varphi _{\gamma
}\left( I^{\prime }\right) \varphi _{\zeta }\left( I_{\ast }^{\prime
}\right) \sigma _{\gamma \zeta }^{\beta }\left( \left\vert \mathbf{g}%
^{\prime }\right\vert ,\left\vert \cos \theta \right\vert ,I_{\ast }\right) 
\text{.}  \label{mr}
\end{equation}%
Moreover, to obtain invariance of change of particles in a collision, it is
assumed that the scattering cross sections $\sigma _{\beta }^{\zeta \gamma }$
for $\left( \beta ,\gamma ,\varsigma \right) \in \mathcal{C}$ satisfy the
symmetry relations%
\begin{equation}
\sigma _{\beta }^{\gamma \zeta }\left( \left\vert \mathbf{g}^{\prime
}\right\vert ,\left\vert \cos \theta \right\vert ,I^{\prime },I_{\ast
}^{\prime }\right) =\sigma _{\beta }^{\zeta \gamma }\left( \left\vert 
\mathbf{g}^{\prime }\right\vert ,\left\vert \cos \theta \right\vert ,I_{\ast
}^{\prime },I^{\prime }\right) \text{.}  \label{sr}
\end{equation}

Applying known properties of Dirac's delta function, the transition
probabilities may for $\left( \beta ,\gamma ,\varsigma \right) \in \mathcal{C%
}$ be transformed to 
\begin{eqnarray*}
&&W_{\gamma \zeta }^{\beta }=W_{\gamma \zeta }^{\beta }(\mathbf{Z}_{\ast },%
\mathbf{Z}^{\prime },\mathbf{Z}_{\ast }^{\prime }) \\
&=&m_{\beta }^{2}m_{\gamma }m_{\zeta }\varphi _{\beta }\left( I_{\ast
}\right) \sigma _{\beta }^{\gamma \zeta }\mathbf{\delta }_{3}\left( m_{\beta
}\left( \boldsymbol{\xi }_{\ast }-\mathbf{G}_{\gamma \zeta }^{\prime
}\right) \right) \mathbf{\delta }_{1}\left( \frac{m_{\gamma }m_{\zeta }}{%
2m_{\beta }}\left\vert \mathbf{g}^{\prime }\right\vert ^{2}-\Delta _{\gamma
\zeta }^{\beta }\mathcal{E}\right) \mathbf{1}_{E_{\beta }\geq K_{\gamma
\zeta }^{\beta }} \\
&=&\frac{m_{\gamma }m_{\zeta }}{m_{\beta }}\varphi _{\beta }\left( I_{\ast
}\right) \sigma _{\beta }^{\gamma \zeta }\mathbf{\delta }_{3}\left( 
\boldsymbol{\xi }_{\ast }-\mathbf{G}_{\gamma \zeta }^{\prime }\right) 
\mathbf{\delta }_{1}\left( \frac{m_{\gamma }m_{\zeta }}{2m_{\beta }}%
\left\vert \mathbf{g}^{\prime }\right\vert ^{2}-\Delta _{\gamma \zeta
}^{\beta }\mathcal{E}\right) \mathbf{1}_{E_{\beta }\geq K_{\gamma \zeta
}^{\beta }} \\
&=&\varphi _{\beta }\left( I_{\ast }\right) \frac{\sigma _{\beta }^{\gamma
\zeta }}{\left\vert \mathbf{g}^{\prime }\right\vert }\mathbf{1}_{\Delta
_{\gamma \zeta }^{\beta }\mathcal{E}>0}\mathbf{\delta }_{1}\left( \left\vert 
\mathbf{g}^{\prime }\right\vert -\sqrt{\widetilde{\Delta }_{\gamma \zeta
}^{\beta }\mathcal{E}}\right) \mathbf{\delta }_{3}\left( \boldsymbol{\xi }%
_{\ast }-\mathbf{G}_{\gamma \zeta }^{\prime }\right) \mathbf{1}_{E_{\beta
}\geq K_{\gamma \zeta }^{\beta }} \\
&&\text{with }\mathbf{G}_{\gamma \zeta }^{\prime }=\frac{m_{\gamma }%
\boldsymbol{\xi }^{\prime }+m_{\zeta }\boldsymbol{\xi }_{\ast }^{\prime }}{%
m_{\beta }}\text{ and }\widetilde{\Delta }_{\gamma \zeta }^{\beta }\mathcal{E%
}=\frac{2m_{\beta }}{m_{\gamma }m_{\zeta }}\Delta _{\gamma \zeta }^{\beta }%
\mathcal{E}\text{.}
\end{eqnarray*}

Note that for $\left( \beta ,\gamma ,\varsigma \right) \in \mathcal{C}$%
\begin{equation*}
W_{\gamma \zeta }^{\beta }(\mathbf{Z}_{\ast },\mathbf{Z}^{\prime },\mathbf{Z}%
_{\ast }^{\prime })=W_{\zeta \gamma }^{\beta }(\mathbf{Z}_{\ast },\mathbf{Z}%
_{\ast }^{\prime },\mathbf{Z}^{\prime })\text{,}
\end{equation*}%
while%
\begin{equation*}
\mathbf{g}_{\ast }:=\boldsymbol{\xi }_{\ast }-\boldsymbol{\xi }_{\ast
}^{\prime }=\frac{m_{\gamma }}{m_{\beta }}\mathbf{g}^{\prime }=\left( 1-%
\frac{m_{\zeta }}{m_{\beta }}\right) \mathbf{g}^{\prime }\text{ and }%
\widetilde{\mathbf{g}}:=\boldsymbol{\xi }_{\ast }-\boldsymbol{\xi }^{\prime
}=-\frac{m_{\zeta }}{m_{\beta }}\mathbf{g}^{\prime }=\left( \frac{m_{\gamma }%
}{m_{\beta }}-1\right) \mathbf{g}^{\prime }\text{.}
\end{equation*}%
Then, also, for $\left( \beta ,\gamma ,\varsigma \right) \in \mathcal{C}$%
\begin{eqnarray*}
&&W_{\gamma \zeta }^{\beta }=W_{\gamma \zeta }^{\beta }(\mathbf{Z}_{\ast },%
\mathbf{Z}^{\prime },\mathbf{Z}_{\ast }^{\prime }) \\
&=&\varphi _{\gamma }\left( I^{\prime }\right) \varphi _{\zeta }\left(
I_{\ast }^{\prime }\right) \frac{m_{\beta }^{2}m_{\zeta }}{m_{\gamma }^{2}}%
\sigma _{\gamma \zeta }^{\beta }\mathbf{1}_{\Delta _{\gamma \zeta }^{\beta }%
\mathcal{E}>0}\mathbf{\delta }_{1}\left( \frac{m_{\beta }m_{\zeta }}{%
2m_{\gamma }}\left\vert \mathbf{g}_{\ast }\right\vert ^{2}-\Delta _{\gamma
\zeta }^{\beta }\mathcal{E}\right) \\
&&\times \mathbf{\delta }_{3}\left( \boldsymbol{\xi }^{\prime }-\widetilde{%
\mathbf{G}}_{\beta \zeta }\right) \mathbf{1}_{E_{\beta }\geq K_{\gamma \zeta
}^{\beta }} \\
&=&\varphi _{\gamma }\left( I^{\prime }\right) \varphi _{\zeta }\left(
I_{\ast }^{\prime }\right) \frac{m_{\beta }\sigma _{\gamma \zeta }^{\beta }}{%
m_{\gamma }\left\vert \mathbf{g}_{\ast }\right\vert }\mathbf{1}_{\Delta
_{\gamma \zeta }^{\beta }\mathcal{E}>0}\mathbf{\delta }_{1}\left( \sqrt{%
\widetilde{\Delta }_{\beta }^{\gamma \zeta }\mathcal{E}}-\left\vert \mathbf{g%
}_{\ast }\right\vert \right) \mathbf{\delta }_{3}\left( \boldsymbol{\xi }%
^{\prime }-\widetilde{\mathbf{G}}_{\beta \zeta }\right) \mathbf{1}_{E_{\beta
}\geq K_{\gamma \zeta }^{\beta }} \\
&=&\varphi _{\beta }\left( I_{\ast }\right) \sigma _{\beta }^{\gamma \zeta }%
\frac{\left\vert \mathbf{g}^{\prime }\right\vert }{\left\vert \mathbf{g}%
_{\ast }\right\vert ^{2}}\mathbf{\delta }_{1}\left( \sqrt{\widetilde{\Delta }%
_{\beta }^{\gamma \zeta }\mathcal{E}}-\left\vert \mathbf{g}_{\ast
}\right\vert \right) \mathbf{\delta }_{3}\left( \boldsymbol{\xi }^{\prime }-%
\widetilde{\mathbf{G}}_{\beta \zeta }\right) \mathbf{1}_{E_{\beta }\geq
K_{\gamma \zeta }^{\beta }}\text{, with} \\
&&\text{ }\widetilde{\mathbf{G}}_{\beta \zeta }=\frac{m_{\beta }\boldsymbol{%
\xi }_{\ast }-m_{\zeta }\boldsymbol{\xi }_{\ast }^{\prime }}{m_{\gamma }}%
\text{ and }\widetilde{\Delta }_{\beta }^{\gamma \zeta }\mathcal{E}=\frac{%
2m_{\gamma }}{m_{\beta }m_{\zeta }}\Delta _{\gamma \zeta }^{\beta }\mathcal{E%
}\text{.}
\end{eqnarray*}

\begin{remark}
\label{Rem1}Note that for $\left( \beta ,\gamma ,\varsigma \right) \in 
\mathcal{C}$%
\begin{equation*}
\mathbf{\delta }_{1}\left( \frac{m_{\gamma }m_{\zeta }}{2\left( m_{\gamma
}+m_{\zeta }\right) }\left\vert \mathbf{g}^{\prime }\right\vert ^{2}-\Delta
_{\gamma \zeta }^{\beta }\mathcal{E}\right) =\mathbf{\delta }_{1}\left(
E_{\beta }-E_{\gamma \zeta }^{\prime }\right) =\mathbf{\delta }_{1}\left(
I_{\ast }-\widetilde{E}_{\gamma \zeta }^{\prime }\right) ,
\end{equation*}%
where $\widetilde{E}_{\gamma \zeta }^{\prime }=\dfrac{m_{\gamma }m_{\zeta }}{%
2\left( m_{\gamma }+m_{\zeta }\right) }\left\vert \mathbf{g}^{\prime
}\right\vert ^{2}+I^{\prime }\mathbf{1}_{\gamma \in \mathcal{I}%
_{poly}}+I_{\ast }^{\prime }\mathbf{1}_{\zeta \in \mathcal{I}_{poly}}+\Delta
_{\gamma \zeta }^{\beta }\varepsilon _{0}$, with 
\begin{equation*}
\Delta _{\gamma \zeta }^{\beta }\varepsilon _{0}=\varepsilon _{\gamma
0}+\varepsilon _{\zeta 0}-\varepsilon _{\beta 0}>0.
\end{equation*}
\end{remark}

By a series of change of variables:\newline
$\left\{ \boldsymbol{\xi }^{\prime },\boldsymbol{\xi }_{\ast }^{\prime
}\right\} \rightarrow \!\left\{ \mathbf{g}^{\prime }=\boldsymbol{\xi }%
^{\prime }-\boldsymbol{\xi }_{\ast }^{\prime }\text{,}\mathbf{G}_{\gamma
\zeta }^{\prime }=\dfrac{m_{\gamma }\boldsymbol{\xi }^{\prime }+m_{\zeta }%
\boldsymbol{\xi }_{\ast }^{\prime }}{m_{\gamma }+m_{\zeta }}\!\right\} $,
followed by a change to spherical\ coordinates $\left\{ \mathbf{g}^{\prime
}\right\} \rightarrow \left\{ \left\vert \mathbf{g}^{\prime }\right\vert ,%
\boldsymbol{\sigma \,}=\dfrac{\mathbf{g}^{\prime }}{\left\vert \mathbf{g}%
^{\prime }\right\vert }\right\} $, observe that%
\begin{equation*}
d\boldsymbol{\xi }^{\prime }d\boldsymbol{\xi }_{\ast }^{\prime }=d\mathbf{G}%
_{\gamma \zeta }^{\prime }d\mathbf{g}^{\prime }=\left\vert \mathbf{g}%
^{\prime }\right\vert ^{2}d\mathbf{G}_{\gamma \zeta }^{\prime }d\left\vert 
\mathbf{g}^{\prime }\right\vert d\boldsymbol{\sigma \,}\text{,}
\end{equation*}%
then - assuming that $\zeta \in \mathcal{I}_{poly}$ - \newline
$\left\{ \left\vert \mathbf{g}^{\prime }\right\vert ,I_{\ast }^{\prime
}\right\} \rightarrow \left\{ R=\dfrac{m_{\gamma }m_{\zeta }}{2m_{\beta }}%
\dfrac{\left\vert \mathbf{g}^{\prime }\right\vert ^{2}}{\widetilde{E}_{\beta
}},\widetilde{E}_{\gamma \zeta }^{\prime }=\dfrac{m_{\gamma }m_{\zeta }}{%
2m_{\beta }}\left\vert \mathbf{g}^{\prime }\right\vert ^{2}+I^{\prime }%
\mathbf{1}_{\gamma \in \mathcal{I}_{poly}}+I_{\ast }^{\prime }+\Delta
_{\gamma \zeta }^{\beta }\varepsilon _{0}\right\} $, with $\widetilde{E}%
_{\beta }=I_{\ast }-\Delta _{\gamma \zeta }^{\beta }\varepsilon _{0}=\left(
1-\dfrac{\Delta _{\gamma \zeta }^{\beta }\varepsilon _{0}}{I_{\ast }}\right)
I_{\ast }$, follows that 
\begin{equation*}
d\boldsymbol{\xi }^{\prime }d\boldsymbol{\xi }_{\ast }^{\prime }dI_{\ast
}^{\prime }=\sqrt{2}\left( \frac{m_{\beta }}{m_{\gamma }m_{\zeta }}\right)
^{3/2}\widetilde{E}_{\beta }^{3/2}R^{1/2}dRd\boldsymbol{\sigma \,}d\mathbf{G}%
_{\gamma \zeta }^{\prime }d\widetilde{E}_{\gamma \zeta }^{\prime }\text{,}
\end{equation*}%
and, finally, if also $\gamma \in \mathcal{I}_{poly}$, $\left\{ I^{\prime
}\right\} \rightarrow \left\{ r=\dfrac{I^{\prime }}{\left( 1-R\right) 
\widetilde{E}_{\beta }}\right\} $, then%
\begin{equation}
d\boldsymbol{\xi }^{\prime }d\boldsymbol{\xi }_{\ast }^{\prime }dI^{\prime
}dI_{\ast }^{\prime }=\sqrt{2}\left( \frac{m_{\beta }}{m_{\gamma }m_{\zeta }}%
\right) ^{3/2}\widetilde{E}_{\beta }^{5/2}(1-R)R^{1/2}drdRd\boldsymbol{%
\sigma \,}d\mathbf{G}_{\gamma \zeta }^{\prime }d\widetilde{E}_{\gamma \zeta
}^{\prime }\text{.}  \label{df1}
\end{equation}%
Similarly,%
\begin{eqnarray*}
d\boldsymbol{\xi }_{\ast }d\boldsymbol{\xi }_{\ast }^{\prime } &=&d%
\widetilde{\mathbf{G}}_{\beta \zeta }d\mathbf{g}_{\ast }=\left\vert \mathbf{g%
}_{\ast }\right\vert ^{2}d\widetilde{\mathbf{G}}_{\beta \zeta }d\left\vert 
\mathbf{g}_{\ast }\right\vert d\boldsymbol{\sigma \,}\text{,} \\
d\boldsymbol{\xi }_{\ast }d\boldsymbol{\xi }_{\ast }^{\prime }dI_{\ast } &=&%
\sqrt{2}\left( \frac{m_{\gamma }}{m_{\beta }m_{\zeta }}\right) ^{3/2}%
\widetilde{E}_{\beta }^{3/2}R^{1/2}dRd\boldsymbol{\sigma \,}d\widetilde{%
\mathbf{G}}_{\beta \zeta }d\widetilde{E}_{\beta }\text{ if }\left( \gamma
,\zeta \right) \notin \mathcal{I}_{mono}^{2}\text{, and} \\
d\boldsymbol{\xi }_{\ast }d\boldsymbol{\xi }_{\ast }^{\prime }dI_{\ast
}dI_{\ast }^{\prime } &=&\sqrt{2}\left( \frac{m_{\gamma }}{m_{\beta
}m_{\zeta }}\right) ^{3/2}\widetilde{E}_{\beta }^{5/2}(1-R)R^{1/2}drdRd%
\boldsymbol{\sigma \,}d\widetilde{\mathbf{G}}_{\beta \zeta }d\widetilde{E}%
_{\beta }\text{ if }\left\{ \gamma ,\zeta \right\} \subset \mathcal{I}_{poly}%
\text{,}
\end{eqnarray*}%
where $\widetilde{\mathbf{G}}_{\beta \zeta }=\dfrac{m_{\beta }\boldsymbol{%
\xi }_{\ast }-m_{\zeta }\boldsymbol{\xi }_{\ast }^{\prime }}{m_{\gamma }}$,
while%
\begin{eqnarray*}
d\boldsymbol{\xi }_{\ast }d\boldsymbol{\xi }^{\prime } &=&d\widetilde{%
\mathbf{G}}_{\beta \gamma }d\widetilde{\mathbf{g}}=\left\vert \widetilde{%
\mathbf{g}}\right\vert ^{2}d\widetilde{\mathbf{G}}_{\beta \gamma
}d\left\vert \widetilde{\mathbf{g}}\right\vert d\boldsymbol{\sigma }\text{%
\thinspace } \\
d\boldsymbol{\xi }_{\ast }d\boldsymbol{\xi }^{\prime }dI_{\ast } &=&\sqrt{2}%
\left( \frac{m_{\zeta }}{m_{\beta }m_{\gamma }}\right) ^{3/2}\widetilde{E}%
_{\beta }^{3/2}R^{1/2}dRd\boldsymbol{\sigma \,}d\widetilde{\mathbf{G}}%
_{\beta \gamma }d\widetilde{E}_{\beta }\text{ if }\left( \gamma ,\zeta
\right) \notin \mathcal{I}_{mono}^{2}\text{, and} \\
d\boldsymbol{\xi }_{\ast }d\boldsymbol{\xi }^{\prime }dI_{\ast }dI^{\prime }
&=&\sqrt{2}\left( \frac{m_{\zeta }}{m_{\beta }m_{\gamma }}\right) ^{3/2}%
\widetilde{E}_{\beta }^{5/2}(1-R)R^{1/2}drdRd\boldsymbol{\sigma \,}d%
\widetilde{\mathbf{G}}_{\beta \gamma }d\widetilde{E}_{\beta }\text{ if }%
\left( \gamma ,\zeta \right) \in \mathcal{I}_{poly}^{2}\text{,} \\
&&\text{with }d\widetilde{\mathbf{G}}_{\beta \gamma }=\frac{m_{\beta }%
\boldsymbol{\xi }_{\ast }-m_{\gamma }\boldsymbol{\xi }^{\prime }}{m_{\zeta }}%
\text{.}
\end{eqnarray*}%
Then, with $\varphi _{\alpha }\left( I\right) =I^{\delta ^{\left( \alpha
\right) }/2-1}$ for $\alpha \in \mathcal{I}$, for two monatomic
constituents, i.e., with $\left( \gamma ,\zeta \right) \in \mathcal{I}%
_{mono}^{2}$, (mono/mono-case) 
\begin{eqnarray*}
Q_{\gamma \zeta }^{\beta ,\alpha }(f) &=&\int_{\mathbb{R}_{+}\times \mathbb{S%
}^{2}}B_{0\gamma \zeta }^{\beta }\Delta _{0\gamma \zeta }^{\beta }\left(
f_{\gamma }^{\prime }f_{\zeta \ast }^{\prime }-\frac{f_{\beta \ast }}{%
I_{\ast }^{\delta ^{\left( \beta \right) }/2-1}}\right) \,dI_{\ast }d%
\boldsymbol{\sigma }\text{, with } \\
B_{0\gamma \zeta }^{\beta } &=&I_{\ast }^{\delta ^{\left( \beta \right)
}/2-1}\sigma _{\beta }^{\gamma \zeta }\left\vert \mathbf{g}^{\prime
}\right\vert \mathbf{1}_{E_{\beta }\geq K_{\gamma \zeta }^{\beta }}=\sigma
_{\gamma \zeta }^{\beta }\left\vert \mathbf{g}^{\prime }\right\vert \mathbf{1%
}_{E_{\beta }\geq K_{\gamma \zeta }^{\beta }}\text{ and} \\
\Delta _{0\gamma \zeta }^{\beta }\left( \alpha ,\mathbf{Z}\right)  &=&\delta
_{1}\left( E_{\beta }-E_{\gamma \zeta }^{\prime }\right) \mathbf{1}_{\left(
\alpha ,\mathbf{Z}\right) =\left( \beta ,\left( \boldsymbol{\xi }\mathbf{%
_{\ast }},I_{\ast }\right) \right) }-\mathbf{1}_{\left( \alpha ,\mathbf{Z}%
\right) =\left( \gamma ,\boldsymbol{\xi }^{\prime }\right) }-\mathbf{1}%
_{\left( \alpha ,\mathbf{Z}\right) =\left( \zeta ,\boldsymbol{\xi }\mathbf{%
_{\ast }^{\prime }}\right) }\text{,}
\end{eqnarray*}%
or, for one monatomic and one polyatomic constituent, respectively, i.e.,
either with $\left( \gamma ,\zeta \right) \in \mathcal{I}_{mono}\times 
\mathcal{I}_{poly}$ (mono/poly-case),%
\begin{eqnarray*}
Q_{\gamma \zeta }^{\beta ,\alpha }(f) &=&\int_{[0,1]^{2}\mathbb{\times S}%
^{2}}B_{1\gamma \zeta }^{\beta }\Delta _{1\gamma \zeta }^{\beta }\left( 
\frac{f_{\gamma }^{\prime }f_{\zeta \ast }^{\prime }}{\left( I_{\ast
}^{\prime }\right) ^{\delta ^{\left( \zeta \right) }/2-1}}-\frac{f_{\beta
\ast }}{I_{\ast }^{\delta ^{\left( \beta \right) }/2-1}}\right) I_{\ast
}^{\delta ^{\left( \beta \right) }/2-1} \\
&&\times \left( 1-R\right) ^{\delta ^{\left( \zeta \right) }/2-1}\left( 1-%
\dfrac{\Delta _{\gamma \zeta }^{\beta }\varepsilon _{0}}{I_{\ast }}\right)
^{\delta ^{\left( \zeta \right) }/2}drdRd\boldsymbol{\sigma }\text{, }
\end{eqnarray*}%
with%
\begin{eqnarray*}
B_{1\gamma \zeta }^{\beta } &=&\sqrt{\frac{2m_{\beta }}{m_{\gamma }m_{\zeta }%
}}\frac{\widetilde{E}_{\beta }^{1/2}I_{\ast }^{\delta ^{\left( \zeta \right)
}/2}}{\left( I_{\ast }^{\prime }\right) ^{\delta ^{\left( \zeta \right)
}/2-1}}R^{1/2}\sigma _{\beta }^{\gamma \zeta }\mathbf{1}_{\Delta _{\gamma
\zeta }^{\beta }\mathcal{E}>0}\mathbf{1}_{E_{\beta }\geq K_{\gamma \zeta
}^{\beta }} \\
&=&\left\vert \mathbf{g}^{\prime }\right\vert \frac{I_{\ast }^{\delta
^{\left( \zeta \right) }/2}}{\left( I_{\ast }^{\prime }\right) ^{\delta
^{\left( \zeta \right) }/2-1}}\sigma _{\beta }^{\gamma \zeta }\mathbf{1}%
_{\Delta _{\gamma \zeta }^{\beta }\mathcal{E}>0}\mathbf{1}_{E_{\beta }\geq
K_{\gamma \zeta }^{\beta }} \\
&=&\left\vert \mathbf{g}^{\prime }\right\vert I_{\ast }^{\left( \delta
^{\left( \zeta \right) }-\delta ^{\left( \beta \right) }\right) /2+1}\sigma
_{\gamma \zeta }^{\beta }\mathbf{1}_{\Delta _{\gamma \zeta }^{\beta }%
\mathcal{E}>0}\mathbf{1}_{E_{\beta }\geq K_{\gamma \zeta }^{\beta }}\text{
and} \\
\Delta _{1\gamma \zeta }^{\beta }\left( \alpha ,\mathbf{Z}\right)  &=&\delta
_{1}\left( 1-r\right) \mathbf{1}_{\left( \alpha ,\mathbf{Z}\right) =\left(
\beta ,\left( \boldsymbol{\xi }\mathbf{_{\ast }},I_{\ast }\right) \right)
}-\delta _{1}\left( 1-r\right) \mathbf{1}_{\left( \alpha ,\mathbf{Z}\right)
=\left( \gamma ,\boldsymbol{\xi }^{\prime }\right) }-\mathbf{1}_{\left(
\alpha ,\mathbf{Z}\right) =\left( \zeta ,\left( \boldsymbol{\xi }\mathbf{%
_{\ast }^{\prime },}I_{\ast }^{\prime }\right) \right) }\text{,}
\end{eqnarray*}%
or, correspondingly, with $\left( \gamma ,\zeta \right) \in \mathcal{I}%
_{poly}\times \mathcal{I}_{mono}$ (poly/mono-case),%
\begin{eqnarray*}
Q_{\gamma \zeta }^{\beta ,\alpha }(f) &=&\int_{[0,1]^{2}\mathbb{\times S}%
^{2}}B_{1\zeta \gamma }^{\beta }\Delta _{1\zeta \gamma }^{\beta }\left( 
\frac{f_{\gamma }^{\prime }f_{\zeta \ast }^{\prime }}{\left( I^{\prime
}\right) ^{\delta ^{\left( \gamma \right) }/2-1}}-\frac{f_{\beta \ast }}{%
I_{\ast }^{\delta ^{\left( \beta \right) }/2-1}}\right) I_{\ast }^{\delta
^{\left( \beta \right) }/2-1} \\
&&\times \left( 1-R\right) ^{\delta ^{\left( \gamma \right) }/2-1}\left( 1-%
\dfrac{\Delta _{\gamma \zeta }^{\beta }\varepsilon _{0}}{I_{\ast }}\right)
^{\delta ^{\left( \gamma \right) }/2}drdRd\boldsymbol{\sigma }\text{,}
\end{eqnarray*}%
and, finally, for two polyatomic constituents, i.e., with $\left( \gamma
,\zeta \right) \in \mathcal{I}_{poly}^{2}$, (poly/poly-case) 
\begin{eqnarray*}
&&Q_{\gamma \zeta }^{\beta ,\alpha }(f) \\
&=&\int_{[0,1]^{2}\mathbb{\times S}^{2}}B_{2\gamma \zeta }^{\beta }\Delta
_{2\gamma \zeta }^{\beta }\left( \frac{f_{\gamma }^{\prime }f_{\zeta \ast
}^{\prime }}{\left( I^{\prime }\right) ^{\delta ^{\left( \gamma \right)
}/2-1}\left( I_{\ast }^{\prime }\right) ^{\delta ^{\left( \zeta \right)
}/2-1}}-\frac{f_{\beta \ast }}{I_{\ast }^{\delta ^{\left( \beta \right)
}/2-1}}\right) I_{\ast }^{\delta ^{\left( \beta \right) }/2-1} \\
&&\times r^{\delta ^{\left( \gamma \right) }/2-1}\left( 1-r\right) ^{\delta
^{\left( \zeta \right) }/2-1}(1-R)^{\left( \delta ^{\left( \gamma \right)
}+\delta ^{\left( \zeta \right) }\right) /2-1}\left( 1-\dfrac{\Delta
_{\gamma \zeta }^{\beta }\varepsilon _{0}}{I_{\ast }}\right) ^{\left( \delta
^{\left( \gamma \right) }+\delta ^{\left( \zeta \right) }\right) /2}drdRd%
\boldsymbol{\sigma }\text{,}
\end{eqnarray*}%
with%
\begin{eqnarray*}
B_{2\gamma \zeta }^{\beta } &=&\sqrt{\frac{2m_{\beta }}{m_{\gamma }m_{\zeta }%
}}\frac{R^{1/2}\widetilde{E}_{\beta }^{1/2}I_{\ast }^{\left( \delta ^{\left(
\gamma \right) }+\delta ^{\left( \zeta \right) }\right) /2}}{\left(
I^{\prime }\right) ^{\delta ^{\left( \gamma \right) }/2-1}\left( I_{\ast
}^{\prime }\right) ^{\delta ^{\left( \zeta \right) }/2-1}}\sigma _{\beta
}^{\gamma \zeta }\mathbf{1}_{\Delta _{\gamma \zeta }^{\beta }\mathcal{E}>0}%
\mathbf{1}_{E_{\beta }\geq K_{\gamma \zeta }^{\beta }} \\
&=&\left\vert \mathbf{g}^{\prime }\right\vert \frac{I_{\ast }^{\left( \delta
^{\left( \gamma \right) }+\delta ^{\left( \zeta \right) }\right) /2}}{\left(
I^{\prime }\right) ^{\delta ^{\left( \gamma \right) }/2-1}\left( I_{\ast
}^{\prime }\right) ^{\delta ^{\left( \zeta \right) }/2-1}}\sigma _{\beta
}^{\gamma \zeta }\mathbf{1}_{\Delta _{\gamma \zeta }^{\beta }\mathcal{E}>0}%
\mathbf{1}_{E_{\beta }\geq K_{\gamma \zeta }^{\beta }} \\
&=&\left\vert \mathbf{g}^{\prime }\right\vert I_{\ast }^{\left( \delta
^{\left( \gamma \right) }+\delta ^{\left( \zeta \right) }-\delta ^{\left(
\beta \right) }\right) /2+1}\sigma _{\gamma \zeta }^{\beta }\mathbf{1}%
_{\Delta _{\gamma \zeta }^{\beta }\mathcal{E}>0}\mathbf{1}_{E_{\beta }\geq
K_{\gamma \zeta }^{\beta }}\text{, and} \\
\Delta _{2\gamma \zeta }^{\beta }\left( \alpha ,\mathbf{Z}\right)  &=&%
\mathbf{1}_{\left( \alpha ,\mathbf{Z}\right) =\left( \beta ,\left( 
\boldsymbol{\xi }\mathbf{_{\ast }},I_{\ast }\right) \right) }-\mathbf{1}%
_{\left( \alpha ,\mathbf{Z}\right) =\left( \gamma ,\left( \boldsymbol{\xi }%
^{\prime },I^{\prime }\right) \right) }-\mathbf{1}_{\left( \alpha ,\mathbf{Z}%
\right) =\left( \zeta ,\left( \boldsymbol{\xi }\mathbf{_{\ast }^{\prime },}%
I_{\ast }^{\prime }\right) \right) }\text{,}
\end{eqnarray*}%
where%
\begin{equation*}
\left\{ 
\begin{array}{l}
\boldsymbol{\xi }^{\prime }=\boldsymbol{\xi }\mathbf{_{\ast }}+\boldsymbol{%
\sigma }\dfrac{m_{\gamma }}{m_{\beta }}\sqrt{\widetilde{\Delta }_{\gamma
\zeta }^{\beta }\mathcal{E}}\medskip  \\ 
\boldsymbol{\xi }_{\ast }^{\prime }=\boldsymbol{\xi }\mathbf{_{\ast }}-%
\boldsymbol{\sigma }\dfrac{m_{_{\zeta }}}{m_{\beta }}\sqrt{\widetilde{\Delta 
}_{\gamma \zeta }^{\beta }\mathcal{E}}%
\end{array}%
\right. \text{,}
\end{equation*}%
or, correspondingly,%
\begin{equation*}
\left\{ 
\begin{array}{l}
\boldsymbol{\xi }\mathbf{_{\ast }}=\boldsymbol{\xi }_{\ast }^{\prime }+%
\boldsymbol{\sigma }\dfrac{m_{_{\zeta }}}{m_{\beta }}\sqrt{\widetilde{\Delta 
}_{\gamma \zeta }^{\beta }\mathcal{E}}\medskip  \\ 
\boldsymbol{\xi }^{\prime }=\boldsymbol{\xi }_{\ast }^{\prime }+\boldsymbol{%
\sigma }\sqrt{\widetilde{\Delta }_{\gamma \zeta }^{\beta }\mathcal{E}}%
\end{array}%
\right. \text{ and }\left\{ 
\begin{array}{l}
\boldsymbol{\xi }\mathbf{_{\ast }}=\boldsymbol{\xi }^{\prime }-\boldsymbol{%
\sigma }\dfrac{m_{\gamma }}{m_{\beta }}\sqrt{\widetilde{\Delta }_{\gamma
\zeta }^{\beta }\mathcal{E}}\medskip  \\ 
\boldsymbol{\xi }_{\ast }^{\prime }=\boldsymbol{\xi }^{\prime }-\boldsymbol{%
\sigma }\sqrt{\widetilde{\Delta }_{\gamma \zeta }^{\beta }\mathcal{E}}%
\end{array}%
\right. \text{,}
\end{equation*}%
resulting in more explicit forms of the chemical process operators. For the
mono/mono-case, note that%
\begin{equation*}
\frac{m_{\beta }}{m_{\gamma }m_{\zeta }}\left\vert \mathbf{g}^{\prime
}\right\vert \,dI_{\ast }d\boldsymbol{\sigma }=\left\vert \mathbf{g}^{\prime
}\right\vert ^{2}d\left\vert \mathbf{g}^{\prime }\right\vert d\boldsymbol{%
\sigma }=d\mathbf{g}^{\prime }=\left\{ 
\begin{array}{c}
d\boldsymbol{\xi }^{\prime }\text{ for fixed }\boldsymbol{\xi }_{\ast
}^{\prime } \\ 
d\boldsymbol{\xi }_{\ast }^{\prime }\text{ for fixed }\boldsymbol{\xi }%
^{\prime }%
\end{array}%
\right. \text{.}
\end{equation*}

Explicitly, the internal energy gaps are given by $\Delta I=I_{\ast }$ in
the mono/mono-case, $\Delta I=I_{\ast }-I_{\ast }^{\prime }$ in the
mono/poly-case, $\Delta I=I_{\ast }-I^{\prime }$ in the poly/mono-case,
while $\Delta I=I_{\ast }-I^{\prime }-I_{\ast }^{\prime }$ in the
poly/poly-case.

\subsection{Collision invariants and Maxwellian distributions\label{S2.2}}

Denote for $\left( \beta ,\gamma ,\varsigma \right) \in \mathcal{C}$%
\begin{equation*}
dA_{\gamma \zeta }^{\beta }=W_{\gamma \zeta }^{\beta }(\mathbf{Z}_{\ast },%
\mathbf{Z}^{\prime },\mathbf{Z}_{\ast }^{\prime })\,d\mathbf{Z}_{\ast }d%
\mathbf{Z}^{\prime }d\mathbf{Z}_{\ast }^{\prime }
\end{equation*}%
and for $\left( \beta ,\gamma ,\varsigma ,\alpha \right) \in \mathcal{%
C\times I}$%
\begin{equation*}
\widetilde{\mathcal{Z}}=\mathcal{Z\times Z}_{\alpha }=\mathcal{Z}_{\beta }%
\mathcal{\times Z}_{\gamma }\mathcal{\times Z}_{\zeta }\mathcal{\times Z}%
_{\alpha }\text{.}
\end{equation*}

The weak form of the collision operator $Q_{chem}(f)$ reads%
\begin{equation*}
\left( Q_{chem}(f),g\right) =\sum_{\left( \beta ,\gamma ,\varsigma ,\alpha
\right) \in \mathcal{C\times I}}\int_{\widetilde{\mathcal{Z}}}\Delta
_{\gamma \zeta }^{\beta }\left( \alpha \text{,}\mathbf{Z}\right) \Lambda
_{\gamma \zeta }^{\beta }(f)g_{\alpha }\,d\mathbf{Z}dA_{\gamma \zeta
}^{\beta }
\end{equation*}%
for any function $g=\left( g_{1},...,g_{s}\right) $, with $g_{\alpha
}=g_{\alpha }(\boldsymbol{\xi },I)$, such that the integrals are defined for
all $\left( \beta ,\gamma ,\varsigma \right) \in \mathcal{C}$.

We have the following proposition.

\begin{proposition}
\label{P1}Let $g=\left( g_{1},...,g_{s}\right) $, with $g_{\alpha
}=g_{\alpha }(\mathbf{Z})$, be such that for all $\left( \beta ,\gamma
,\varsigma \right) \in \mathcal{C}$%
\begin{equation*}
\sum_{\alpha \in \mathcal{I}}\int_{\widetilde{\mathcal{Z}}}\Delta _{\gamma
\zeta }^{\beta }\left( \alpha \text{,}\mathbf{Z}\right) \Lambda _{\gamma
\zeta }^{\beta }(f)g_{\alpha }\,d\mathbf{Z}dA_{\gamma \zeta }^{\beta }\text{%
, where }\Lambda _{\gamma \zeta }^{\beta }=\frac{f_{\gamma }^{\prime
}f_{\zeta \ast }^{\prime }}{\varphi _{\gamma }\left( I^{\prime }\right)
\varphi _{\zeta }\left( I_{\ast }^{\prime }\right) }-\frac{f_{\beta \ast }}{%
\varphi _{\beta }\left( I_{\ast }\right) }\text{,}
\end{equation*}%
is defined. Then%
\begin{equation*}
\left( Q_{chem}(f),g\right) =\sum_{\left( \beta ,\gamma ,\varsigma \right)
\in \mathcal{C}}\int_{\mathcal{Z}}\Lambda _{\gamma \zeta }^{\beta }(f)\left(
g_{_{\beta \ast }}-g_{\gamma }^{\prime }-g_{\zeta \ast }^{\prime }\right)
\,dA_{\gamma \zeta }^{\beta }.
\end{equation*}
\end{proposition}

\begin{definition}
A function $g=\left( g_{1},...,g_{s}\right) $, with $g_{\alpha }=g_{\alpha }(%
\mathbf{Z})$,$\ $is a chemical process invariant if 
\begin{equation*}
\left( g_{_{\beta \ast }}-g_{\gamma }^{\prime }-g_{\zeta \ast }^{\prime
}\right) \,W_{\gamma \zeta }^{\beta }(\mathbf{Z}_{\ast },\mathbf{Z}^{\prime
},\mathbf{Z}_{\ast }^{\prime })=0\text{ a.e.}
\end{equation*}%
for all $\left( \beta ,\gamma ,\varsigma \right) \in \mathcal{C}$.
\end{definition}

We remind about the following definition \cite{Be-24a}.

\begin{definition}
A function $g=\left( g_{1},...,g_{s}\right) $, with $g_{\alpha }=g_{\alpha }(%
\mathbf{Z})$,$\ $is a mechanical collision invariant if 
\begin{equation*}
\left( g_{\alpha }+g_{_{\beta \ast }}-g_{\alpha }^{\prime }-g_{\beta \ast
}^{\prime }\right) W_{\alpha \beta }(\mathbf{Z},\mathbf{Z}_{\ast }\left\vert 
\mathbf{Z}^{\prime },\mathbf{Z}_{\ast }^{\prime }\right. )=0\text{ a.e.}
\end{equation*}%
for all $\left\{ \alpha ,\beta \right\} \subseteq \left\{ 1,...,s\right\} $.
\end{definition}

Also remind that the set $\left\{ e_{1},...,e_{s},m\xi _{x},m\xi _{y},m\xi
_{z},m\left\vert \boldsymbol{\xi }\right\vert ^{2}+2\mathbb{I}\right\} $,
where \linebreak $m=\left( m_{1},...,m_{s}\right) $, $\mathbb{I}%
=(\varepsilon _{10},...,\varepsilon _{s_{0}0},I+\varepsilon
_{s_{0}+10},...,I+\varepsilon _{s0})$, and $\left\{ e_{1},...,e_{s}\right\} $
is the standard basis of $\mathbb{R}^{s}$, is a basis for the vector space
of mechanical collision invariants \cite{DMS-05, BBBD-18, Be-24a}. Moreover,
we introduce the concept of a common invariant.

\begin{definition}
A function $g=\left( g_{1},...,g_{s}\right) $, with $g_{\alpha }=g_{\alpha }(%
\mathbf{Z})$,$\ $is a collision invariant if it is a chemical process
invariant as well as a mechanical collision invariant.
\end{definition}

Introduce the vector space%
\begin{equation}
\mathcal{U}:=\left\{ u\in \mathbb{R}^{s};\text{ }u\cdot \left( e_{\beta
}-e_{\gamma }-e_{\varsigma }\right) =0\text{ for all }\left( \beta ,\gamma
,\varsigma \right) \in \mathcal{C}\right\} \text{,}  \label{CPI}
\end{equation}%
and denote by 
\begin{equation}
\mathcal{U}_{0}:=\left\{ u_{1},...,u_{\widetilde{s}}\right\} \text{, where }%
\widetilde{s}:=\dim \mathcal{U}\text{,}  \label{CPIb}
\end{equation}%
a basis of $\mathcal{U}$. We have the following proposition.

\begin{proposition}
\label{P2}Let $m=\left( m_{1},...,m_{s}\right) $, $\mathbb{I}=(\varepsilon
_{10},...,\varepsilon _{s_{0}0},I+\varepsilon _{s_{0}+10},...,I+\varepsilon
_{s0})$, and $\left\{ u_{1},...,u_{\widetilde{s}}\right\} $ be a basis of $%
\mathcal{U}$. Then the vector space of collision invariants is generated by 
\begin{equation*}
\left\{ u_{1},...,u_{\widetilde{s}},m\xi _{x},m\xi _{y},m\xi
_{z},m\left\vert \boldsymbol{\xi }\right\vert ^{2}+2\mathbb{I}\right\} .
\end{equation*}
\end{proposition}

For example, if $\left( \beta ,\gamma ,\varsigma \right) \in \mathcal{C}$,
such that the species $a_{\beta }$, $a_{\gamma }$, and $a_{\varsigma }$ do
not take part in any other chemical process, then if $\gamma \neq \varsigma $
one can replace $\left\{ e_{\beta },e_{\gamma },e_{\varsigma }\right\} $ in
the set of generating mechanical collision invariants by $\left\{ e_{\beta
}+e_{\gamma },e_{\beta }+e_{\varsigma }\right\} $ to describe generators for
the set of collision invariants, while if $\gamma =\varsigma $ one can
correspondingly replace $\left\{ e_{\beta },e_{\gamma }\right\} $ in the set
of generating mechanical collision invariants by $\left\{ 2e_{\beta
}+e_{\gamma }\right\} $.

It is known that \cite{Be-24a}%
\begin{equation*}
\mathcal{W}_{mech}\left[ f\right] :=\left( Q_{mech}(f),\log \left( \varphi
^{-1}f\right) \right) \leq 0
\end{equation*}%
where $\varphi =\mathrm{diag}\left( \varphi _{1}\left( I\right) ,...,\varphi
_{s}\left( I\right) \right) $, with equality if and only if%
\begin{equation*}
Q_{mech}(f)\equiv 0\text{.}
\end{equation*}%
Furthermore, the mechanical equilibrium distributions $%
M_{mech}=(M_{m1},...,M_{ms})$ are of the form \cite{Be-24a}%
\begin{equation}
M_{\alpha }=\left\{ 
\begin{array}{l}
\dfrac{n_{\alpha }m_{\alpha }^{3/2}}{\left( 2\pi k_{B}T\right) ^{3/2}}%
e^{-m_{\alpha }\left\vert \boldsymbol{\xi }-\mathbf{u}\right\vert
^{2}/\left( 2k_{B}T\right) }\text{ if }\alpha \in \mathcal{I}_{mono}\text{ \
\ } \\ 
\dfrac{n_{\alpha }\varphi _{\alpha }\left( I\right) m_{\alpha
}^{3/2}e^{-\left( m_{\alpha }\left\vert \boldsymbol{\xi }-\mathbf{u}%
\right\vert ^{2}+2I\right) /\left( 2k_{B}T\right) }}{\left( 2\pi
k_{B}T\right) ^{3/2}q_{\alpha }}\text{ if }\alpha \in \mathcal{I}_{poly}%
\end{array}%
\right. ,  \label{ma1}
\end{equation}%
with the normalization factor $q_{\alpha }=\int_{0}^{\infty }\varphi
_{\alpha }\left( I\right) e^{-I/\left( k_{B}T\right) }\,dI$ for $\alpha \in 
\mathcal{I}_{poly}$, where $n_{\alpha }=\left( M,e_{_{\alpha }}\right) $, $%
\mathbf{u}=\dfrac{1}{\rho }\left( M,m\boldsymbol{\xi }\right) $, and $T=%
\dfrac{1}{3nk_{B}}\left( M,m\left\vert \boldsymbol{\xi }-\mathbf{u}%
\right\vert ^{2}\right) $, with mass vector $m=(m_{1},...,m_{s})$, $%
n=\sum\limits_{\alpha =1}^{s}n_{\alpha }$, and $\rho =\sum\limits_{\alpha
=1}^{s}m_{\alpha }n_{\alpha }$, while $k_{B}$ denote the Boltzmann constant.

For the particular case when $\varphi _{\alpha }\left( I\right) =I^{\delta
^{\left( \alpha \right) }/2-1}$ for all $\alpha \in \mathcal{I}$ 
\begin{equation*}
q_{\alpha }=\left( k_{B}T\right) ^{\delta ^{\left( \alpha \right) }/2}\Gamma
\left( \delta ^{\left( \alpha \right) }/2\right) \text{,}
\end{equation*}%
where $\Gamma =\Gamma (n)$ denote the Gamma function $\Gamma
(n)=\int_{0}^{\infty }x^{n-1}e^{-x}\,dx$.

Correspondingly, define%
\begin{equation*}
\mathcal{W}_{chem}\left[ f\right] :=\left( Q_{chem}(f),\log \left( \varphi
^{-1}f\right) \right) ,
\end{equation*}%
It follows by Proposition $\ref{P1}$ that%
\begin{eqnarray*}
&&\mathcal{W}_{chem}\left[ f\right] \\
&=&-\sum_{\left( \beta ,\gamma ,\varsigma \right) \in \mathcal{C}}\int_{%
\mathcal{Z}}\left( \frac{f_{\gamma }^{\prime }f_{\zeta \ast }^{\prime
}\varphi _{\beta }\left( I_{\ast }\right) }{f_{\beta \ast }\varphi _{\gamma
}\left( I^{\prime }\right) \varphi _{\zeta }\left( I_{\ast }^{\prime
}\right) }-1\right) \log \left( \frac{\varphi _{\beta }\left( I_{\ast
}\right) f_{\gamma }^{\prime }f_{\zeta \ast }^{\prime }}{f_{\beta \ast
}\varphi _{\gamma }\left( I^{\prime }\right) \varphi _{\zeta }\left( I_{\ast
}^{\prime }\right) }\right) \frac{f_{\beta \ast }}{\varphi _{\beta }\left(
I_{\ast }\right) }\,dA_{\gamma \zeta }^{\beta }\text{.}
\end{eqnarray*}%
Since $\left( x-1\right) \mathrm{log}\left( x\right) \geq 0$ for $x>0$, with
equality if and only if $x=1$,%
\begin{equation*}
\mathcal{W}_{chem}\left[ f\right] \leq 0\text{,}
\end{equation*}%
with equality if and only if for all $\left( \beta ,\gamma ,\varsigma
\right) \in \mathcal{C}$ 
\begin{equation}
\Lambda _{\gamma \zeta }^{\beta }(f)W_{\gamma \zeta }^{\beta }\left( \mathbf{%
Z}_{\ast },\mathbf{Z}^{\prime },\mathbf{Z}_{\ast }^{\prime }\right) =0\text{
a.e.,}  \label{m1}
\end{equation}%
or, equivalently, if and only if%
\begin{equation*}
Q_{chem}(f)\equiv 0\text{.}
\end{equation*}

Note that $Q(M)\equiv 0$ if and only if $Q_{chem}(f)\equiv 0$ and $%
Q_{mech}(f)\equiv 0$. Hence, any equilibrium, or, Maxwellian, distribution $%
M=(M_{1},...,M_{s})$, i.e., such that $Q(M)\equiv 0$, has to be of the form $%
\left( \ref{ma1}\right) $, since $Q_{mech}(M)\equiv 0$, but in addition, it
follows by equation $\left( \ref{m1}\right) $, since $Q_{chem}(M)\equiv 0$,
that%
\begin{equation*}
\left( \log \frac{M_{\beta \ast }}{\varphi _{\beta }\left( I_{\ast }\right) }%
-\log \frac{M_{\gamma }^{\prime }}{\varphi _{\gamma }\left( I^{\prime
}\right) }-\log \frac{M_{\zeta \ast }^{\prime }}{\varphi _{\zeta }\left(
I_{\ast }^{\prime }\right) }\right) W_{\gamma \zeta }^{\beta }\left( \mathbf{%
Z}_{\ast },\mathbf{Z}^{\prime },\mathbf{Z}_{\ast }^{\prime }\right) =0\text{
a.e..}
\end{equation*}%
Hence, $\log \left( \varphi ^{-1}M\right) =\left( \log \dfrac{M_{1}}{\varphi
_{1}\left( I\right) },...,\log \dfrac{M_{s}}{\varphi _{s}\left( I\right) }%
\right) $ is a chemical process invariant, and the components of the
Maxwellian distributions $M=(M_{1},...,M_{s})$ are, of the form $\left( \ref%
{ma1}\right) $, with 
\begin{equation}
\frac{n_{\gamma }n_{\varsigma }}{n_{\beta }}=\left( \frac{2\pi
k_{B}Tm_{\beta }}{m_{\gamma }m_{\varsigma }}\right) ^{3/2}e^{-\Delta
_{\gamma \zeta }^{\beta }\varepsilon _{0}/\left( k_{B}T\right) }\dfrac{%
q_{\gamma }q_{\varsigma }}{q_{\beta }}  \label{ma2a}
\end{equation}%
for all $\left( \beta ,\gamma ,\varsigma \right) \in \mathcal{C}$, where%
\begin{equation*}
q_{\alpha }=\left\{ 
\begin{array}{l}
1\text{ if }\alpha \in \mathcal{I}_{mono}\text{ \ \ } \\ 
\int_{0}^{\infty }\varphi _{\alpha }\left( I\right) e^{-I/\left(
k_{B}T\right) }\,dI\text{ if }\alpha \in \mathcal{I}_{poly}%
\end{array}%
\right. \text{.}
\end{equation*}%
In particular, for the particular case when $\varphi _{\alpha }\left(
I\right) =I^{\delta ^{\left( \alpha \right) }/2-1}$ for all $\alpha \in 
\mathcal{I}$%
\begin{eqnarray}
\frac{n_{\gamma }n_{\varsigma }}{n_{\beta }} &=&\left( \frac{2\pi m_{\beta }%
}{m_{\gamma }m_{\varsigma }}\right) ^{3/2}e^{-\Delta _{\gamma \zeta }^{\beta
}\varepsilon _{0}/\left( k_{B}T\right) }\dfrac{\Gamma \left( \delta ^{\left(
\gamma \right) }/2\right) \Gamma \left( \delta ^{\left( \varsigma \right)
}/2\right) }{\Gamma \left( \delta ^{\left( \beta \right) }/2\right) }  \notag
\\
&&\times \left( k_{B}T\right) ^{\left( \delta ^{\left( \gamma \right) }%
\mathbf{1}_{\gamma \in \mathcal{I}_{poly}}+\delta ^{\left( \varsigma \right)
}\mathbf{1}_{\varsigma \in \mathcal{I}_{poly}}+3-\delta ^{\left( \beta
\right) }\right) /2}  \label{ma2}
\end{eqnarray}%
for all $\left( \beta ,\gamma ,\varsigma \right) \in \mathcal{C}$.

Note that, by equation $\left( \ref{m1}\right) $, any Maxwellian
distribution $M=(M_{1},...,M_{s})$ satisfies the relations 
\begin{equation}
\Lambda _{\gamma \zeta }^{\beta }(M)W_{\gamma \zeta }^{\beta }\left( \mathbf{%
Z}_{\ast },\mathbf{Z}^{\prime },\mathbf{Z}_{\ast }^{\prime }\right) =0\text{
a.e.}  \label{M1}
\end{equation}%
for any $\left( \beta ,\gamma ,\varsigma \right) \in \mathcal{C}$.

\begin{remark}
Introducing the $\mathcal{H}$-functional%
\begin{equation*}
\mathcal{H}\left[ f\right] =\left( f,\log \left( \varphi ^{-1}f\right)
\right) \text{,}
\end{equation*}%
an $\mathcal{H}$-theorem can be obtained.
\end{remark}

\subsection{Linearized collision operator\label{S2.3}}

Consider (without loss of generality) a deviation of a non-drifting, i.e.
with $\mathbf{u=0}$ in expression $\left( \ref{ma1}\right) $, Maxwellian
distribution $M=(M_{1},...,M_{s})$ $\left( \ref{ma1}\right) ,\left( \ref%
{ma2a}\right) $ of the form%
\begin{equation}
f=M+\mathcal{M}^{1/2}h\text{, with }\mathcal{M}=\mathrm{diag}\left(
M_{1},...,M_{s}\right) \text{.}  \label{s1}
\end{equation}%
Insertion in the Boltzmann equation $\left( \ref{BE1}\right) $ results in
the system%
\begin{equation}
\frac{\partial h}{\partial t}+\left( \boldsymbol{\xi }\cdot \nabla _{\mathbf{%
x}}\right) h+\mathcal{L}h=S\left( h,h\right) \text{,}  \label{LBE}
\end{equation}%
with the linearized collision operator $\mathcal{L}=\mathcal{L}_{mech}+%
\mathcal{L}_{chem}$, where the components of the linearized collision
operator $\mathcal{L}_{chem}=\left( \mathcal{L}_{c1},...,\mathcal{L}%
_{cs}\right) $ are given by%
\begin{eqnarray}
\mathcal{L}_{c\alpha }h &=&-M_{\alpha }^{-1/2}\sum_{\left( \beta ,\gamma
,\varsigma \right) \in \mathcal{C}}\int_{\mathcal{Z}}W_{\gamma \zeta
}^{\beta }\Delta _{\gamma \zeta }^{\beta }\left( \alpha ,\mathbf{Z}\right) 
\widetilde{\Lambda }_{\gamma \zeta }^{\beta }(f)\,d\mathbf{Z}_{\ast }d%
\mathbf{Z}^{\prime }d\mathbf{Z}_{\ast }^{\prime }\text{, where}  \notag \\
\widetilde{\Lambda }_{\gamma \zeta }^{\beta } &=&\frac{M_{\gamma }^{\prime
}\left( M_{\zeta \ast }^{\prime }\right) ^{1/2}h_{\zeta \ast }^{\prime
}+M_{\zeta \ast }^{\prime }\left( M_{\gamma }^{\prime }\right)
^{1/2}h_{\gamma }^{\prime }}{\varphi _{\gamma }\left( I^{\prime }\right)
\varphi _{\zeta }\left( I_{\ast }^{\prime }\right) }-\frac{M_{\beta \ast
}^{1/2}h_{\beta \ast }}{\varphi _{\beta }\left( I_{\ast }\right) }  \notag \\
&=&\left( \frac{M_{\beta \ast }M_{\gamma }^{\prime }M_{\zeta \ast }^{\prime }%
}{\varphi _{\beta }\left( I_{\ast }\right) \varphi _{\gamma }\left(
I^{\prime }\right) \varphi _{\zeta }\left( I_{\ast }^{\prime }\right) }%
\right) ^{1/2}\left( \frac{h_{\gamma }^{\prime }}{\left( M_{\gamma }^{\prime
}\right) ^{1/2}}+\frac{h_{\zeta \ast }^{\prime }}{\left( M_{\zeta \ast
}^{\prime }\right) ^{1/2}}-\frac{h_{\beta \ast }}{M_{\beta \ast }^{1/2}}%
\right) \text{,}  \notag
\end{eqnarray}%
while $S=S_{mech}+S_{chem}$, where the components of the non-linear operator 
$S_{chem}=\left( S_{c1},...,S_{cs}\right) $ are given by%
\begin{eqnarray}
S_{c\alpha }\left( h\right)  &=&M_{\alpha }^{-1/2}\sum_{\left( \beta ,\gamma
,\varsigma \right) \in \mathcal{C}}\int_{\mathcal{Z}}W_{\gamma \zeta
}^{\beta }\Delta _{\gamma \zeta }^{\beta }\left( \alpha ,\mathbf{Z}\right) 
\widehat{\Lambda }_{\gamma \zeta }^{\beta }(h)\,d\mathbf{Z}_{\ast }d\mathbf{Z%
}^{\prime }d\mathbf{Z}_{\ast }^{\prime }\text{, where}  \notag \\
\widehat{\Lambda }_{\gamma \zeta }^{\beta } &=&\frac{\left( M_{\gamma
}^{\prime }\right) ^{1/2}\left( M_{\zeta \ast }^{\prime }\right) ^{1/2}}{%
\varphi _{\gamma }\left( I^{\prime }\right) \varphi _{\zeta }\left( I_{\ast
}^{\prime }\right) }h_{\gamma }^{\prime }h_{\zeta \ast }^{\prime }=\left( 
\frac{M_{\beta \ast }}{\varphi _{\beta }\left( I_{\ast }\right) \varphi
_{\gamma }\left( I^{\prime }\right) \varphi _{\zeta }\left( I_{\ast
}^{\prime }\right) }\right) ^{1/2}h_{\gamma }^{\prime }h_{\zeta \ast
}^{\prime }\text{.}  \label{nl1}
\end{eqnarray}%
Here, for $\alpha \in \mathcal{I}$%
\begin{equation}
\mathcal{L}_{c\alpha }h=\nu _{c\alpha }h_{\alpha }-K_{c\alpha }\left(
h\right) \text{,}  \label{dec1}
\end{equation}%
\ with%
\begin{eqnarray}
\nu _{c\alpha } &=&\frac{1}{\varphi _{\alpha }\left( I\right) }\sum_{\beta
,\gamma =1}^{s}\int\limits_{\mathcal{Z}_{\beta }\mathcal{\times Z}_{\gamma
}}W_{\beta \gamma }^{\alpha }\left( \mathbf{Z},\mathbf{Z}_{\ast },\mathbf{Z}%
^{\prime }\right) \mathbf{1}_{\left( \alpha ,\beta ,\gamma \right) \in 
\mathcal{C}}+\frac{2M_{\gamma }^{\prime }}{\varphi _{\gamma }\left(
I^{\prime }\right) }W_{\alpha \gamma }^{\beta }\left( \mathbf{Z}_{\ast },%
\mathbf{Z},\mathbf{Z}^{\prime }\right)   \notag \\
&&\times \mathbf{1}_{\left( \beta ,\alpha ,\gamma \right) \in \mathcal{C}}\,d%
\mathbf{Z}_{\ast }d\mathbf{Z}^{\prime }  \label{cf1}
\end{eqnarray}%
and%
\begin{eqnarray}
&&K_{c\alpha }  \notag \\
&=&-\frac{2M_{\alpha }^{1/2}}{\varphi _{\alpha }\left( I\right) }\sum_{\beta
,\gamma =1}^{s}\int\limits_{\mathcal{Z}_{\beta }\mathcal{\times Z}_{\gamma }}%
\frac{M_{\gamma }^{\prime }}{\varphi _{\gamma }\left( I^{\prime }\right) }%
W_{\alpha \gamma }^{\beta }\left( \mathbf{Z}_{\ast },\mathbf{Z},\mathbf{Z}%
^{\prime }\right) \frac{h_{\beta \ast }}{M_{\beta \ast }^{1/2}}\mathbf{1}%
_{\left( \beta ,\alpha ,\gamma \right) \in \mathcal{C}}  \notag \\
&&-\frac{h_{\gamma }^{\prime }}{\left( M_{\gamma }^{\prime }\right) ^{1/2}}%
\left( W_{\beta \gamma }^{\alpha }\left( \mathbf{Z},\mathbf{Z}_{\ast },%
\mathbf{Z}^{\prime }\right) \mathbf{1}_{\left( \alpha ,\beta ,\gamma \right)
\in \mathcal{C}}+\frac{M_{\gamma }^{\prime }}{\varphi _{\gamma }\left(
I^{\prime }\right) }W_{\alpha \gamma }^{\beta }\left( \mathbf{Z}_{\ast },%
\mathbf{Z},\mathbf{Z}^{\prime }\right) \mathbf{1}_{\left( \beta ,\alpha
,\gamma \right) \in \mathcal{C}}\right) \,d\mathbf{Z}_{\ast }d\mathbf{Z}%
^{\prime }  \notag \\
&=&\frac{2M_{\alpha }^{1/2}}{\varphi _{\alpha }\left( I\right) }\sum_{\beta
,\gamma =1}^{s}\int_{\mathcal{Z}_{\beta }\mathcal{\times Z}_{\gamma }}\left(
W_{\beta \gamma }^{\alpha }\left( \mathbf{Z},\mathbf{Z}_{\ast },\mathbf{Z}%
^{\prime }\right) \mathbf{1}_{\left( \alpha ,\beta ,\gamma \right) \in 
\mathcal{C}}+\frac{M_{\beta \ast }}{\varphi _{\beta }\left( I_{\ast }\right) 
}W_{\alpha \beta }^{\gamma }\left( \mathbf{Z}^{\prime },\mathbf{Z},\mathbf{Z}%
_{\ast }\right) \mathbf{1}_{\left( \gamma ,\alpha ,\beta \right) \in 
\mathcal{C}}\right.   \notag \\
&&\left. -\frac{M_{\gamma }^{\prime }}{\varphi _{\gamma }\left( I^{\prime
}\right) }W_{\alpha \gamma }^{\beta }\left( \mathbf{Z}_{\ast },\mathbf{Z},%
\mathbf{Z}^{\prime }\right) \mathbf{1}_{\left( \beta ,\alpha ,\gamma \right)
\in \mathcal{C}}\right) \frac{h_{\beta \ast }}{M_{\beta \ast }^{1/2}}d%
\mathbf{Z}_{\ast }d\mathbf{Z}^{\prime }  \notag \\
&=&2\sum_{\beta ,\gamma =1}^{s}\int_{\mathcal{Z}_{\beta }}k_{\alpha \beta
\gamma }\left( \mathbf{Z},\mathbf{Z}_{\ast }\right) h_{\beta \ast }d\mathbf{Z%
}_{\ast }\text{, where }k_{\alpha \beta \gamma }\left( \mathbf{Z},\mathbf{Z}%
_{\ast }\right)   \notag \\
&=&\int_{\mathcal{Z}_{\gamma }}\left( \frac{W_{\beta \gamma }^{\alpha
}\left( \mathbf{Z},\mathbf{Z}_{\ast },\mathbf{Z}^{\prime }\right) \mathbf{1}%
_{\left( \alpha ,\beta ,\gamma \right) \in \mathcal{C}}+W_{\alpha \beta
}^{\gamma }\left( \mathbf{Z}^{\prime },\mathbf{Z},\mathbf{Z}_{\ast }\right) 
\mathbf{1}_{\left( \gamma ,\alpha ,\beta \right) \in \mathcal{C}}}{\left(
\varphi _{\alpha }\left( I\right) \varphi _{\beta }\left( I_{\ast }\right)
\varphi _{\gamma }\left( I^{\prime }\right) \right) ^{1/2}}\right.   \notag
\\
&&\left. -\frac{W_{\alpha \gamma }^{\beta }\left( \mathbf{Z}_{\ast },\mathbf{%
Z},\mathbf{Z}^{\prime }\right) \mathbf{1}_{\left( \beta ,\alpha ,\gamma
\right) \in \mathcal{C}}}{\left( \varphi _{\alpha }\left( I\right) \varphi
_{\beta }\left( I_{\ast }\right) \varphi _{\gamma }\left( I^{\prime }\right)
\right) ^{1/2}}\left( M_{\gamma }^{\prime }\right) ^{1/2}d\mathbf{Z}^{\prime
}\right) \text{.}  \label{comp1}
\end{eqnarray}

Moreover, the components of the linearized collision operator $\mathcal{L}%
_{mech}=\left( \mathcal{L}_{m1},...,\mathcal{L}_{ms}\right) $ \cite{Be-24a}
are given by%
\begin{eqnarray}
\mathcal{L}_{m\alpha }h &=&-M_{\alpha }^{-1/2}\left( Q_{\alpha }(M,\mathcal{M%
}^{1/2}h)+Q_{\alpha }(\mathcal{M}^{1/2}h,M)\right)   \notag \\
&=&\sum\limits_{\beta =1}^{s}\int_{\mathcal{Z}_{\alpha }\times \mathcal{Z}%
_{\beta }^{2}}W_{\alpha \beta }\left( \frac{h_{\alpha }}{M^{1/2}}+\frac{%
h_{\beta \ast }}{M_{\beta \ast }^{1/2}}-\frac{h_{\alpha }^{\prime }}{\left(
M_{\alpha }^{\prime }\right) ^{1/2}}-\frac{h_{\beta \ast }^{\prime }}{\left(
M_{\beta \ast }^{\prime }\right) ^{1/2}}\right)   \notag \\
&&\times \left( \frac{M_{\beta \ast }M_{\alpha }^{\prime }M_{\beta \ast
}^{\prime }}{\varphi _{\alpha }\left( I\right) \varphi _{\beta }\left(
I_{\ast }\right) \varphi _{\alpha }\left( I^{\prime }\right) \varphi _{\beta
}\left( I_{\ast }^{\prime }\right) }\right) ^{1/2}d\mathbf{Z}_{\ast }d%
\mathbf{Z}^{\prime }d\mathbf{Z}_{\ast }^{\prime }  \notag \\
&=&\nu _{m\alpha }h_{\alpha }-K_{m\alpha }\left( h\right) \text{,}
\label{dec2}
\end{eqnarray}%
\ with%
\begin{eqnarray*}
\nu _{m\alpha } &=&\sum\limits_{\beta =1}^{s}\int_{\mathcal{Z}_{\alpha
}\times \mathcal{Z}_{\beta }^{2}}\frac{M_{\beta \ast }}{\varphi _{\alpha
}\left( I\right) \varphi _{\beta }\left( I_{\ast }\right) }W_{\alpha \beta }d%
\boldsymbol{\xi }_{\ast }d\boldsymbol{\xi }^{\prime }d\boldsymbol{\xi }%
_{\ast }^{\prime }dI_{\ast }dI^{\prime }dI_{\ast }^{\prime }\text{,} \\
K_{m\alpha } &=&\sum\limits_{\beta =1}^{s}\int_{\mathcal{Z}_{\alpha }\times 
\mathcal{Z}_{\beta }^{2}}W_{\alpha \beta }\left( \frac{h_{\alpha }^{\prime }%
}{\left( M_{\alpha }^{\prime }\right) ^{1/2}}+\frac{h_{\beta \ast }^{\prime }%
}{\left( M_{\beta \ast }^{\prime }\right) ^{1/2}}-\frac{h_{\beta \ast }}{%
M_{\beta \ast }^{1/2}}\right)  \\
&&\times \left( \frac{M_{\beta \ast }M_{\alpha }^{\prime }M_{\beta \ast
}^{\prime }}{\varphi _{\alpha }\left( I\right) \varphi _{\beta }\left(
I_{\ast }\right) \varphi _{\alpha }\left( I^{\prime }\right) \varphi _{\beta
}\left( I_{\ast }^{\prime }\right) }\right) ^{1/2}d\mathbf{Z}_{\ast }d%
\mathbf{Z}^{\prime }d\mathbf{Z}_{\ast }^{\prime }\text{,}
\end{eqnarray*}%
while the components of the quadratic term $S_{mech}=\left(
S_{m1},...,S_{ms}\right) $ \cite{Be-24a} are given by%
\begin{equation*}
S_{m\alpha }\left( h,h\right) =M_{\alpha }^{-1/2}Q_{m\alpha }(\mathcal{M}%
^{1/2}h,\mathcal{M}^{1/2}h)\text{.}
\end{equation*}%
for $\alpha \in \mathcal{I}$. The multiplication operator $\Lambda $ defined
by 
\begin{equation*}
\Lambda (f)=\nu f\text{, where }\nu =\mathrm{diag}\left( \nu _{1},...,\nu
_{s}\right) \text{, with }\nu _{\alpha }=\nu _{m\alpha }+\nu _{c\alpha }%
\text{,}
\end{equation*}%
is a closed, densely defined, self-adjoint operator on $\mathcal{\mathfrak{h}%
}$. It is Fredholm as well if and only if $\Lambda $ is coercive.

We remind the following properties of $\mathcal{L}_{mech}$ \cite{Be-24a}.

\begin{proposition}
\label{P4}The linearized collision operator $\mathcal{L}_{mech}$ is
symmetric and nonnegative,%
\begin{equation*}
\left( \mathcal{L}_{mech}h,g\right) =\left( h,\mathcal{L}_{mech}g\right) 
\text{ and }\left( \mathcal{L}_{mech}h,h\right) \geq 0\text{,}
\end{equation*}%
and\ the kernel of $\mathcal{L}_{mech}$, $\ker \mathcal{L}_{mech}$, is
generated by%
\begin{equation*}
\left\{ \mathcal{M}_{1}^{1/2}e_{1},...,\mathcal{M}^{1/2}e_{s},\mathcal{M}%
^{1/2}m\xi _{x},\mathcal{M}^{1/2}m\xi _{y},\mathcal{M}^{1/2}m\xi _{z},%
\mathcal{M}^{1/2}\left( m\left\vert \boldsymbol{\xi }\right\vert ^{2}+2%
\mathbb{I}\right) \right\} \text{,}
\end{equation*}%
where $\left\{ e_{1},...,e_{s}\right\} $ is the standard basis of $\mathbb{R}%
^{s}$, $m=\left( m_{1},...,m_{s}\right) $, $\mathcal{M}=\mathrm{diag}\left(
M_{1},...,M_{s}\right) $, and $\mathbb{I}=(\varepsilon _{10},...,\varepsilon
_{s_{0}0},I+\varepsilon _{s_{0}+10},...,I+\varepsilon _{s0})$.
\end{proposition}

Denote for $\left( \beta ,\gamma ,\varsigma \right) \in \mathcal{C}$%
\begin{equation*}
d\widetilde{A}_{\gamma \zeta }^{\beta }=\left( \frac{M_{\beta \ast
}M_{\gamma }^{\prime }M_{\zeta \ast }^{\prime }}{\varphi _{\beta }\left(
I_{\ast }\right) \varphi _{\gamma }\left( I^{\prime }\right) \varphi _{\zeta
}\left( I_{\ast }^{\prime }\right) }\right) ^{1/2}dA_{\gamma \zeta }^{\beta }%
\text{.}
\end{equation*}%
The weak form of the linearized operator $\mathcal{L}_{chem}$ reads%
\begin{equation*}
\left( \mathcal{L}_{chem}h,g\right) =\sum_{\left( \beta ,\gamma ,\varsigma
,\alpha \right) \in \mathcal{C\times I}}\int_{\widetilde{\mathcal{Z}}}\Delta
_{\gamma \zeta }^{\beta }\left( \alpha \text{,}\mathbf{Z}\right) \left( 
\frac{h_{\gamma }^{\prime }}{\left( M_{\gamma }^{\prime }\right) ^{1/2}}+%
\frac{h_{\zeta \ast }^{\prime }}{\left( M_{\zeta \ast }^{\prime }\right)
^{1/2}}-\frac{h_{\beta \ast }}{M_{\beta \ast }^{1/2}}\right) \frac{g_{\alpha
}}{M_{\alpha }^{1/2}}\,d\mathbf{Z}d\widetilde{A}_{\gamma \zeta }^{\beta }
\end{equation*}%
for any function $g=\left( g_{1},...,g_{s}\right) $, with $g_{\alpha
}=g_{\alpha }(\boldsymbol{\xi },I)$, such that the integrals are defined for
all $\left( \beta ,\gamma ,\varsigma \right) \in \mathcal{C}$.

We have the following lemma.

\begin{lemma}
\label{L2}Let $g=\left( g_{1},...,g_{s}\right) $, with $g_{\alpha
}=g_{\alpha }(\boldsymbol{\xi },I)$, be such that 
\begin{equation*}
\sum_{\alpha \in \mathcal{I}}\int_{\widetilde{\mathcal{Z}}}\Delta _{\gamma
\zeta }^{\beta }\left( \alpha \text{,}\mathbf{Z}\right) \left( \frac{%
h_{\gamma }^{\prime }}{\left( M_{\gamma }^{\prime }\right) ^{1/2}}+\frac{%
h_{\zeta \ast }^{\prime }}{\left( M_{\zeta \ast }^{\prime }\right) ^{1/2}}-%
\frac{h_{\beta \ast }}{M_{\beta \ast }^{1/2}}\right) \frac{g_{\alpha }}{%
M_{\alpha }^{1/2}}\,d\mathbf{Z}d\widetilde{A}_{\gamma \zeta }^{\beta }
\end{equation*}%
is defined for all $\left( \beta ,\gamma ,\varsigma \right) \in \mathcal{C}$%
. Then%
\begin{eqnarray*}
\left( \mathcal{L}_{chem}h,g\right) &=&\sum_{\left( \beta ,\gamma ,\varsigma
\right) \in \mathcal{C}}\int\limits_{\mathcal{Z}}\left( \frac{h_{\gamma
}^{\prime }}{\left( M_{\gamma }^{\prime }\right) ^{1/2}}+\frac{h_{\zeta \ast
}^{\prime }}{\left( M_{\zeta \ast }^{\prime }\right) ^{1/2}}-\frac{h_{\beta
\ast }}{M_{\beta \ast }^{1/2}}\right) \\
&&\times \left( \frac{g_{\gamma }^{\prime }}{\left( M_{\gamma }^{\prime
}\right) ^{1/2}}+\frac{g_{\zeta \ast }^{\prime }}{\left( M_{\zeta \ast
}^{\prime }\right) ^{1/2}}-\frac{g_{\beta \ast }}{M_{\beta \ast }^{1/2}}%
\right) \,d\widetilde{A}_{\gamma \zeta }^{\beta }.
\end{eqnarray*}
\end{lemma}

We conclude in the following proposition.

\begin{proposition}
\label{P3}The linearized collision operator $\mathcal{L}$ is symmetric and
nonnegative,%
\begin{equation*}
\left( \mathcal{L}h,g\right) =\left( h,\mathcal{L}g\right) \text{ and }%
\left( \mathcal{L}h,h\right) \geq 0\text{,}
\end{equation*}%
and\ the kernel of $\mathcal{L}$, $\ker \mathcal{L}$, is generated by%
\begin{equation*}
\left\{ \mathcal{M}^{1/2}u_{1},...,\mathcal{M}^{1/2}u_{\widetilde{s}},%
\mathcal{M}^{1/2}m\xi _{x},\mathcal{M}^{1/2}m\xi _{y},\mathcal{M}^{1/2}m\xi
_{z},\mathcal{M}^{1/2}\left( m\left\vert \boldsymbol{\xi }\right\vert ^{2}+2%
\mathbb{I}\right) \right\} \text{,}
\end{equation*}%
where $m=\left( m_{1},...,m_{s}\right) $, $\mathbb{I}=(\varepsilon
_{10},...,\varepsilon _{s_{0}0},I+\varepsilon _{s_{0}+10},...,I+\varepsilon
_{s0})$, $\mathcal{M}=\mathrm{diag}\left( M_{1},...,M_{s}\right) $, and $%
\mathcal{U}_{0}:=\left\{ u_{1},...,u_{\widetilde{s}}\right\} $ denotes a
basis $\left( \ref{CPIb}\right) $ of the set $\left( \ref{CPI}\right) $.
\end{proposition}

\begin{proof}
By Lemma $\ref{L2}$ and Proposition $\ref{P4}$, $\left( \mathcal{L}%
h,g\right) =\left( h,\mathcal{L}g\right) $ is immediate, and%
\begin{equation*}
\left( \mathcal{L}_{chem}h,h\right) =\sum_{\left( \beta ,\gamma ,\varsigma
\right) \in \mathcal{C}}\int_{\mathcal{Z}}\left( \frac{h_{\gamma }^{\prime }%
}{\left( M_{\gamma }^{\prime }\right) ^{1/2}}+\frac{h_{\zeta \ast }^{\prime }%
}{\left( M_{\zeta \ast }^{\prime }\right) ^{1/2}}-\frac{h_{\beta \ast }}{%
M_{\beta \ast }^{1/2}}\right) ^{2}\,d\widetilde{A}_{\gamma \zeta }^{\beta
}\geq 0.
\end{equation*}%
Furthermore, $h\in \ker \mathcal{L}$ if and only if $\left( \mathcal{L}%
h,h\right) =0$, which will be fulfilled\ if and only if $\left( \mathcal{L}%
_{chem}h,h\right) =\left( \mathcal{L}_{mech}h,h\right) =0$. However, $\left( 
\mathcal{L}_{chem}h,h\right) =0$ if and only if 
\begin{equation*}
\left( \frac{h_{\gamma }^{\prime }}{\left( M_{\gamma }^{\prime }\right)
^{1/2}}+\frac{h_{\zeta \ast }^{\prime }}{\left( M_{\zeta \ast }^{\prime
}\right) ^{1/2}}-\frac{h_{\beta \ast }}{M_{\beta \ast }^{1/2}}\right)
\,W_{\gamma \zeta }^{\beta }\left( \mathbf{Z},\mathbf{Z}^{\prime },\mathbf{Z}%
_{\ast }^{\prime }\right) =0\text{ a.e.,}
\end{equation*}%
for all $\left( \beta ,\gamma ,\varsigma \right) \in \mathcal{C}$, i.e., if
and only if $\mathcal{M}^{-1/2}h$ is a chemical process invariant. The last
part of the lemma now follows by Proposition $\ref{P2}$ and $\ref{P4}.$
\end{proof}

\begin{remark}
Note also that it trivially follows that the quadratic term is orthogonal to
the kernel of $\mathcal{L}$, i.e., $S\left( h\right) \in \left( \ker 
\mathcal{L}\right) ^{\perp _{\mathcal{\mathfrak{h}}}}$.
\end{remark}

\section{Main Results\label{S3}}

This section is devoted to the main results, concerning compact properties
in Theorem \ref{Thm1} and bounds of collision frequencies in Theorem \ref%
{Thm2}. Below we consider the particular case $\varphi _{\alpha }\left(
I\right) =I^{\delta ^{\left( \alpha \right) }/2-1}$ for $\alpha \in \mathcal{%
I}$.

Assume that for some positive number $\tau $, such that $0<\tau <1$, there
is for all $\left( \alpha ,\beta \right) \in \mathcal{I}^{2}$ a bound 
\begin{eqnarray}
&&0\leq \sigma _{\alpha \beta }\left( \left\vert \mathbf{g}\right\vert ,\cos
\theta ,I,I_{\ast },I^{\prime },I_{\ast }^{\prime }\right) \leq C\frac{\Psi
_{\alpha \beta }+\left( \Psi _{\alpha \beta }\right) ^{\tau /2}}{\left\vert 
\mathbf{g}\right\vert ^{2}}\Upsilon _{\alpha \beta }\text{, where}  \notag \\
&&\Upsilon _{\alpha \beta }=\frac{\varphi _{\alpha }\left( I^{\prime
}\right) \varphi _{\beta }\left( I_{\ast }^{\prime }\right) }{\widetilde{%
\mathcal{E}}_{\alpha \beta }^{\delta ^{\left( \alpha \right) }/2}\left( 
\widetilde{\mathcal{E}}_{\alpha \beta }^{\ast }\right) ^{\delta ^{\left(
\beta \right) }/2}}\text{ and }\Psi _{\alpha \beta }=\left\vert \mathbf{g}%
\right\vert \sqrt{\left\vert \mathbf{g}\right\vert ^{2}-2\widetilde{\Delta }%
_{\alpha \beta }I}\text{,}  \label{est1a}
\end{eqnarray}%
for $\left\vert \mathbf{g}\right\vert ^{2}>2\widetilde{\Delta }_{\alpha
\beta }I$, with $\widetilde{\Delta }_{\alpha \beta }I=\dfrac{m_{\alpha
}+m_{\beta }}{m_{\alpha }m_{\beta }}\Delta I$, on the scattering cross
sections.

Here and below $\widetilde{\mathcal{E}}_{\alpha \beta }=1$ if $\alpha \in 
\mathcal{I}_{mono}$ and $\widetilde{\mathcal{E}}_{\alpha \beta }^{\ast }=1$
if $\beta \in \mathcal{I}_{mono}$, while, otherwise, 
\begin{eqnarray*}
\widetilde{\mathcal{E}}_{\alpha \beta }=\widetilde{\mathcal{E}}_{\alpha
\beta }^{\ast } &=&\frac{m_{\alpha }m_{\beta }}{2\left( m_{\alpha }+m_{\beta
}\right) }\left\vert \mathbf{g}\right\vert ^{2}+I\mathbf{1}_{\alpha \in 
\mathcal{I}_{poly}}+I_{\ast }\mathbf{1}_{\beta \in \mathcal{I}_{poly}} \\
&=&\frac{m_{\alpha }m_{\beta }}{2\left( m_{\alpha }+m_{\beta }\right) }%
\left\vert \mathbf{g}^{\prime }\right\vert ^{2}+I^{\prime }\mathbf{1}%
_{\alpha \in \mathcal{I}_{poly}}+I_{\ast }^{\prime }\mathbf{1}_{\beta \in 
\mathcal{I}_{poly}}\text{.}
\end{eqnarray*}

We remind the following compactness result for $K_{mech}$ in \cite{Be-24a}.

\begin{theorem}
\label{Thm1a}Assume that for all $\left( \alpha ,\beta \right) \in \mathcal{I%
}^{2}$ the scattering cross sections $\sigma _{\alpha \beta }$ satisfy the
bound $\left( \ref{est1a}\right) $ for some positive number $\tau $, such
that $0<\tau <1$. \newline
Then the operator $K_{mech}=\left( K_{m1},...,K_{ms}\right) $ is a
self-adjoint compact operator on  $\mathcal{\mathfrak{h}}=\left( L^{2}\left(
d\boldsymbol{\xi \,}\right) \right) ^{s_{0}}\times \left( L^{2}\left( d%
\boldsymbol{\xi \,}dI\right) \right) ^{s_{1}}$.
\end{theorem}

Assume that for some positive number $\chi $, such that $0<\chi <1$, there
is for all $\left( \beta ,\gamma ,\varsigma \right) \in \mathcal{C}$ a bound%
\begin{equation}
0\leq \sigma _{\gamma \zeta }^{\beta }\left( \left\vert \mathbf{g}^{\prime
}\right\vert ,\left\vert \cos \theta \right\vert ,I_{\ast }\right) \leq
C\left( 1+\left\vert \mathbf{g}^{\prime }\right\vert ^{\chi -1}\right) \frac{%
\varphi _{\beta }\left( I_{\ast }\right) }{\mathcal{E}_{\gamma \zeta
}^{\delta ^{\left( \gamma \right) }/2}\left( \mathcal{E}_{\gamma \zeta
}^{\ast }\right) ^{\delta ^{\left( \zeta \right) }/2}}  \label{est1}
\end{equation}%
for $E_{\beta }\geq K_{\gamma \zeta }^{\beta }$, on the scattering cross
sections, or, equivalently, the bound 
\begin{equation*}
0\leq B_{i\gamma \zeta }^{\beta }\left( \left\vert \mathbf{g}^{\prime
}\right\vert ,\cos \theta ,I_{\ast },I^{\prime },I_{\ast }^{\prime }\right)
\leq C\left\vert \mathbf{g}^{\prime }\right\vert \left( 1+\left\vert \mathbf{%
g}^{\prime }\right\vert ^{\chi -1}\right)
\end{equation*}%
for $E_{\beta }\geq K_{\gamma \zeta }^{\beta }$, on the collision kernels.
Here and below $\mathcal{E}_{\gamma \zeta }=1$ if $\gamma \in \mathcal{I}%
_{mono}$ and $\mathcal{E}_{\gamma \zeta }^{\ast }=1$ if $\zeta \in \mathcal{I%
}_{mono}$, while, otherwise,%
\begin{equation*}
\mathcal{E}_{\gamma \zeta }=\mathcal{E}_{\gamma \zeta }^{\ast }=\frac{%
m_{\gamma }m_{\zeta }}{2m_{\beta }}\left\vert \mathbf{g}^{\prime
}\right\vert ^{2}+I^{\prime }\mathbf{1}_{\gamma \in \mathcal{I}%
_{poly}}+I_{\ast }^{\prime }\mathbf{1}_{\zeta \in \mathcal{I}_{poly}}+\Delta
_{\gamma \zeta }^{\beta }\varepsilon _{0}=I_{\ast }.
\end{equation*}

Then the following result may be obtained.

\begin{lemma}
\label{L5}Assume that for all $\left( \beta ,\gamma ,\varsigma \right) \in 
\mathcal{C}$ the scattering cross sections $\sigma _{\gamma \zeta }^{\beta }$
satisfy the bound $\left( \ref{est1}\right) $ for some positive number $\chi 
$, such that $0<\chi <1$. \newline
Then the operator $K_{chem}=\left( K_{c1},...,K_{cs}\right) $, with the
components $K_{c\alpha }$ given by $\left( \ref{dec1}\right) $ is a
self-adjoint compact operator on $\mathcal{\mathfrak{h}}=\left( L^{2}\left( d%
\boldsymbol{\xi \,}\right) \right) ^{s_{0}}\times \left( L^{2}\left( d%
\boldsymbol{\xi \,}dI\right) \right) ^{s_{1}}$.
\end{lemma}

The proof of Lemma \ref{L5} will be addressed in Section $\ref{PT1}$. The
following theorem follows by Theorem $\ref{Thm1a}$ and Lemma $\ref{L5}$.

\begin{theorem}
\label{Thm1}Assume that for all $\left( \beta ,\gamma ,\varsigma \right) \in 
\mathcal{C}$ the scattering cross sections $\sigma _{\gamma \zeta }^{\beta }$
satisfy the bound $\left( \ref{est1}\right) $ for some positive number $\chi 
$, such that $0<\chi <1$, and that for all $\left( \alpha ,\beta \right) \in 
\mathcal{I}_{poly}^{2}$ the scattering cross sections $\sigma _{\alpha \beta
}$ satisfy the bound $\left( \ref{est1a}\right) $ for some positive number $%
\tau $, such that $0<\tau <1$. Then the operator 
\begin{equation*}
K=K_{mech}+K_{chem}
\end{equation*}
given by $\left( \ref{dec1}\right)$  is a self-adjoint compact operator on $%
\mathcal{\mathfrak{h}}=\left( L^{2}\left( d\boldsymbol{\xi \,}\right)
\right) ^{s_{0}}\times \left( L^{2}\left( d\boldsymbol{\xi \,}dI\right)
\right) ^{s_{1}}$.
\end{theorem}

By Theorem \ref{Thm1}, the linear operator $\mathcal{L}=\Lambda -K$ is
closed as the sum of a closed and a bounded operator, and densely defined,
since the domains of the linear operators $\mathcal{L}$ and $\Lambda $ are
equal; $D(\mathcal{L})=D(\Lambda )$. Furthermore, it is a self-adjoint
operator, since the set of self-adjoint operators is closed under addition
of bounded self-adjoint operators, see Theorem 4.3 of Chapter V in \cite%
{Kato}.

\begin{corollary}
\label{Cor1}The linearized collision operator $\mathcal{L}$, with scattering
cross sections satisfying $\left( \ref{est1}\right) $, is a closed, densely
defined, self-adjoint operator on $\mathcal{\mathfrak{h}}$.
\end{corollary}

Now consider, for some nonnegative number $\eta $ less than $1$, $0\leq \eta
<1,$ - cf. hard sphere models for $\eta =0$ - the scattering cross sections%
\begin{equation}
\sigma _{\gamma \zeta }^{\beta }\left( \left\vert \mathbf{g}^{\prime
}\right\vert ,\left\vert \cos \theta \right\vert ,I_{\ast }\right)
=C_{\gamma \zeta }^{\beta }\frac{1}{E_{\beta }^{\eta /2}}\frac{\varphi
_{\beta }\left( I_{\ast }\right) }{\mathcal{E}_{\gamma \zeta }^{\delta
^{\left( \gamma \right) }/2}\left( \mathcal{E}_{\gamma \zeta }^{\ast
}\right) ^{\delta ^{\left( \zeta \right) }/2}}\text{, if }E_{\beta }\geq
K_{\gamma \zeta }^{\beta }\text{,}  \label{e1}
\end{equation}%
for some positive constants $C_{\gamma \zeta }^{\beta }>0$ for $\left( \beta
,\gamma ,\varsigma \right) \in \mathcal{C}$, and%
\begin{equation}
\sigma _{\alpha \beta }=C_{\alpha \beta }\dfrac{\sqrt{\left\vert \mathbf{g}%
\right\vert ^{2}-2\widetilde{\Delta }_{\alpha \beta }I}}{\left\vert \mathbf{g%
}\right\vert E_{\alpha \beta }^{\eta /2}}\frac{\varphi _{\alpha }\left(
I^{\prime }\right) \varphi _{\beta }\left( I_{\ast }^{\prime }\right) }{%
\mathcal{E}_{\alpha \beta }^{\delta ^{\left( \alpha \right) }/2}\left( 
\mathcal{E}_{\alpha \beta }^{\ast }\right) ^{\delta ^{\left( \beta \right)
}/2}}\text{, if }\left\vert \mathbf{g}\right\vert ^{2}>2\widetilde{\Delta }%
_{\alpha \beta }I\text{,}  \label{e1a}
\end{equation}%
with $\widetilde{\Delta }_{\alpha \beta }I=\dfrac{m_{\alpha }+m_{\beta }}{%
m_{\alpha }m_{\beta }}\Delta I$, for some positive constants $C_{\alpha
\beta }>0$ for $\left (\alpha ,\beta \right ) \in \mathcal{I} $.

In fact, it would be enough with the bounds%
\begin{equation}
C_{-}\frac{1}{E_{\beta }^{\eta /2}}\frac{\varphi _{\beta }\left( I_{\ast
}\right) }{\mathcal{E}_{\gamma \zeta }^{\delta ^{\left( \gamma \right)
}/2}\left( \mathcal{E}_{\gamma \zeta }^{\ast }\right) ^{\delta ^{\left(
\zeta \right) }/2}}\leq \sigma _{\gamma \zeta }^{\beta }\leq C_{+}\frac{1}{%
E_{\beta }^{\eta /2}}\frac{\varphi _{\beta }\left( I_{\ast }\right) }{%
\mathcal{E}_{\gamma \zeta }^{\delta ^{\left( \gamma \right) }/2}\left( 
\mathcal{E}_{\gamma \zeta }^{\ast }\right) ^{\delta ^{\left( \zeta \right)
}/2}}\text{,}  \label{ie1}
\end{equation}%
if $E_{\beta }\geq K_{\gamma \zeta }^{\beta }$, and 
\begin{eqnarray}
&&C_{-}\dfrac{\sqrt{\left\vert \mathbf{g}\right\vert ^{2}-2\widetilde{\Delta 
}_{\alpha \beta }I}}{\left\vert \mathbf{g}\right\vert E_{\alpha \beta
}^{\eta /2}}\Upsilon _{\alpha \beta }\leq \sigma _{\alpha \beta }\leq C_{+}%
\dfrac{\sqrt{\left\vert \mathbf{g}\right\vert ^{2}-2\widetilde{\Delta }%
_{\alpha \beta }I}}{\left\vert \mathbf{g}\right\vert E_{\alpha \beta }^{\eta
/2}}\Upsilon _{\alpha \beta }\text{ if }\left\vert \mathbf{g}\right\vert
^{2}>2\widetilde{\Delta }_{\alpha \beta }I\text{,}  \notag \\
&&\text{ with }\Upsilon _{\alpha \beta }=\frac{\varphi _{\alpha }\left(
I^{\prime }\right) \varphi _{\beta }\left( I_{\ast }^{\prime }\right) }{%
\mathcal{E}_{\alpha \beta }^{\delta ^{\left( \alpha \right) }/2}\left( 
\mathcal{E}_{\alpha \beta }^{\ast }\right) ^{\delta ^{\left( \beta \right)
}/2}}\text{ and }\widetilde{\Delta }_{\alpha \beta }I=\frac{m_{\alpha
}+m_{\beta }}{m_{\alpha }m_{\beta }}\Delta I\text{,}  \label{ie1a}
\end{eqnarray}%
for some nonnegative number $\eta $, such that $0\leq \eta <1$, and some
positive constants $C_{\pm }>0$, on the scattering cross sections - cf. hard
potential with cut-off models.

\begin{theorem}
\label{Thm2} The linearized collision operator $\mathcal{L}$, with
scattering cross sections $\left( \ref{e1}\right) $ and $\left( \ref{e1a}%
\right) $ (or $\left( \ref{ie1}\right) $ and $\left( \ref{ie1a}\right) $),
can be split into a positive multiplication operator $\Lambda $, where $%
\Lambda \left( f\right) =\nu f$, with $\nu =\nu _{mech}(\left\vert 
\boldsymbol{\xi }\right\vert )+\nu _{chem}(\left\vert \boldsymbol{\xi }%
\right\vert )$, minus a compact operator $K$ on $\mathcal{\mathfrak{h}}$,
such that there exist positive numbers $\nu _{-}$ and $\nu _{+}$, with
\linebreak $0<\nu _{-}<\nu _{+}$, such that for any $\alpha \in \mathcal{I} $
\begin{equation}
\nu _{-}\left( 1+\left\vert \boldsymbol{\xi }\right\vert +\mathbf{1}_{\alpha
\in \mathcal{I}_{poly}}\sqrt{I}\right) ^{1-\eta }\leq \nu _{\alpha }\leq \nu
_{+}\left( 1+\left\vert \boldsymbol{\xi }\right\vert +\mathbf{1}_{\alpha \in 
\mathcal{I}_{poly}}\sqrt{I}\right) ^{1-\eta }\text{.}  \label{ine1}
\end{equation}
\end{theorem}

The decomposition follows by decomposition $\left( \ref{dec1}\right) $,$%
\left( \ref{dec2}\right) $ and Theorem $\ref{Thm1}$, while the bounds $%
\left( \ref{ine1}\right) $ on the collision frequency will be proven in
Section $\ref{PT2}$.

By Theorem \ref{Thm2} the multiplication operator $\Lambda $ is coercive,
and thus it is a Fredholm operator. Furthermore, the set of Fredholm
operators is closed under addition of compact operators, see Theorem 5.26 of
Chapter IV in \cite{Kato} and its proof, so, by Theorem \ref{Thm2}, $%
\mathcal{L}$ is a Fredholm operator.

\begin{corollary}
\label{Cor2}The linearized collision operator $\mathcal{L}$, with scattering
cross sections $\left( \ref{e1}\right) $ and $\left( \ref{e1a}\right) $ (or $%
\left( \ref{ie1}\right) $ and $\left( \ref{ie1a}\right) $), is a Fredholm
operator with domain%
\begin{equation*}
D(\mathcal{L})=\left( L^{2}\left( \left( 1+\left\vert \boldsymbol{\xi }%
\right\vert \right) ^{1-\eta }d\boldsymbol{\xi \,}\right) \right)
^{s_{0}}\times \left( L^{2}\left( \left( 1+\left\vert \boldsymbol{\xi }%
\right\vert +\sqrt{I}\right) ^{1-\eta }d\boldsymbol{\xi \,}dI\right) \right)
^{s_{1}}\text{.}
\end{equation*}
\end{corollary}

For hard sphere like models we obtain the following result.

\begin{corollary}
\label{Cor3}For the linearized collision operator $\mathcal{L}$, with
scattering cross sections $\left( \ref{e1}\right) $ and $\left( \ref{e1a}%
\right) $ (or $\left( \ref{ie1}\right) $ and $\left( \ref{ie1a}\right) $),
where $\eta =0$, there exists a positive number $\lambda $, $0<\lambda <1$,
such that 
\begin{equation*}
\left( h,\mathcal{L}h\right) \geq \lambda \left( h,\nu (\left\vert 
\boldsymbol{\xi }\right\vert ,I)h\right) \geq \lambda \nu _{-}\left(
h,\left( 1+\left\vert \boldsymbol{\xi }\right\vert \right) h\right)
\end{equation*}%
for all $h\in D\left( \mathcal{L}\right) \cap \mathrm{Im}\mathcal{L}$.
\end{corollary}

The proof is the same as in the non-reactive case, cf., e.g., \cite%
{Be-23a,Be-23b,Be-24a,Be-24b}.

\begin{remark}
By Proposition $\ref{P3}$ and Corollary $\ref{Cor1}-\ref{Cor3}$ the
linearized operator $\mathcal{L}$ fulfills the properties assumed on the
linear operators in \cite{Be-23d}, and hence, the results for half-space
problems with general boundary conditions  therein can be applied to hard
sphere like models.
\end{remark}

\section{Compactness \label{PT1}}

This section concerns the proof of Lemma \ref{L5}. Note that in Section \ref%
{S2.3} the kernels are written in such a way that $\mathbf{Z}_{\ast }$
always will be an argument of the distribution function. Either $\mathbf{Z}$
and $\mathbf{Z}_{\ast }$ are both arguments in the dissociated term of the
collision operator, or, either of $\mathbf{Z}$ and $\mathbf{Z}_{\ast }$ is
the argument in the associated term. In the former case, the kernel will be
shown to be Hilbert-Schmidt , while, in the latter case, the terms will be
shown to be uniform limits of Hilbert-Schmidt integral operators, i.e.,
approximately Hilbert-Schmidt in the sense of Lemma \ref{LGD}.

To obtain the compactness properties we will apply the following result.
Denote, for any (nonzero) natural number $N$,%
\begin{align}
\mathfrak{h}_{N}& :=\left\{ (\mathbf{Z},\mathbf{Z}_{\ast })\in \mathbb{%
Y\times Y}_{\ast }:\left\vert \boldsymbol{\xi }-\boldsymbol{\xi }_{\ast
}\right\vert \geq \frac{1}{N}\text{; }\left\vert \boldsymbol{\xi }%
\right\vert \leq N\right\} \text{, and}  \notag \\
b^{(N)}& =b^{(N)}(\mathbf{Z},\mathbf{Z}_{\ast }):=b(\mathbf{Z},\mathbf{Z}%
_{\ast })\mathbf{1}_{\mathfrak{h}_{N}}\text{.}  \label{rd}
\end{align}%
Here, either $\mathbf{Z}=\boldsymbol{\xi }$ and $\mathbb{Y=R}^{3}$, or, $%
\mathbf{Z}=\left( \boldsymbol{\xi },I\right) $ and $\mathbb{Y=R}^{3}\times 
\mathbb{R}_{+}$, and correspondingly, either $\mathbf{Z}_{\ast }=\boldsymbol{%
\xi }_{\ast }$ and $\mathbb{Y_{\ast }=R}^{3}$, or, $\mathbf{Z}_{\ast
}=\left( \boldsymbol{\xi }_{\ast },I_{\ast }\right) $ and $\mathbb{Y_{\ast
}=R}^{3}\times \mathbb{R}_{+}$. Then we have the following lemma, cf.
Glassey \cite[Lemma 3.5.1]{Glassey} and Drange \cite{Dr-75}.

\begin{lemma}
\label{LGD} Assume that $Tf\left( \mathbf{Z}\right) =\int_{\mathbb{Y}_{\ast
}}b(\mathbf{Z},\mathbf{Z}_{\ast })f\left( \mathbf{Z}_{\ast }\right) \,d%
\mathbf{Z}_{\ast }$, with $b(\mathbf{Z},\mathbf{Z}_{\ast })\geq 0$. Then $T$
is compact on $L^{2}\left( d\mathbf{Z}\right) $ if

(i) $\int_{\mathbb{Y}}b(\mathbf{Z},\mathbf{Z}_{\ast })\,d\mathbf{Z}$ is
bounded in $\mathbf{Z}_{\ast }$;

(ii) $b^{(N)}\in L^{2}\left( d\mathbf{Z}\boldsymbol{\,}d\mathbf{Z}_{\ast
}\right) $ for any (nonzero) natural number $N$;

(iii) $\underset{\mathbf{Z}\in \mathbb{Y}}{\sup }\int_{\mathbb{Y}_{\ast }}b(%
\mathbf{Z},\mathbf{Z}_{\ast })-b^{(N)}(\mathbf{Z},\mathbf{Z}_{\ast })\,d%
\mathbf{Z}_{\ast }\rightarrow 0$ as $N\rightarrow \infty $.
\end{lemma}

Then $T$ \cite[Lemma 3.5.1]{Glassey} is the uniform limit of Hilbert-Schmidt
integral operators and we say that the kernel $b(\mathbf{Z},\mathbf{Z}_{\ast
})$ is approximately Hilbert-Schmidt, while $T$ is an approximately
Hilbert-Schmidt integral operator. The reader is referred to Glassey \cite[%
Lemma 3.5.1]{Glassey} for a proof of Lemma \ref{LGD}.

Now we turn to the proof of Theorem \ref{Thm1}. \newline
Note that throughout the proof $C$ will denote a generic positive constant.
Moreover, remind that $\varphi _{\alpha }\left( I\right) =I^{\delta ^{\left(
\alpha \right) }/2-1}$ for $\alpha \in \mathcal{I}$ below.

\begin{proof}
For $\alpha \in \mathcal{I}$ write expression $\left( \ref{comp1}\right) $ as%
\begin{equation*}
K_{c\alpha }=2\sum_{\beta ,\gamma =1}^{s}\int_{\mathcal{Z}_{\beta
}}k_{\alpha \beta \gamma }\left( \mathbf{Z},\mathbf{Z}_{\ast }\right)
h_{\beta \ast }d\mathbf{Z}_{\ast }\text{, where }k_{\alpha \beta \gamma
}\left( \mathbf{Z},\mathbf{Z}_{\ast }\right) =k_{\beta \gamma }^{1\alpha
}+k_{\alpha \beta }^{2\gamma }-k_{\alpha \gamma }^{3\beta }\text{,}
\end{equation*}%
with%
\begin{eqnarray}
k_{\beta \gamma }^{1\alpha }\left( \mathbf{Z},\mathbf{Z}_{\ast }\right) 
&=&\int_{\mathcal{Z}_{\gamma }}W_{\beta \gamma }^{\alpha }\left( \mathbf{Z},%
\mathbf{Z}_{\ast },\mathbf{Z}^{\prime }\right) \frac{M_{\alpha }^{1/2}}{%
\varphi _{\alpha }\left( I\right) M_{\beta \ast }^{1/2}}d\mathbf{Z}^{\prime }
\notag \\
&=&\int_{\mathcal{Z}_{\gamma }}W_{\beta \gamma }^{\alpha }\left( \mathbf{Z},%
\mathbf{Z}_{\ast },\mathbf{Z}^{\prime }\right) \left( \frac{M_{\alpha
}M_{\gamma }^{\prime }}{\left( \varphi _{\alpha }\left( I\right) \right)
^{3}\varphi _{\beta }\left( I_{\ast }\right) \varphi _{\gamma }\left(
I^{\prime }\right) M_{\beta \ast }}\right) ^{1/4}d\mathbf{Z}^{\prime }\text{,%
}  \notag \\
k_{\alpha \beta }^{2\gamma }\left( \mathbf{Z},\mathbf{Z}_{\ast }\right) 
&=&\int_{\mathcal{Z}_{\gamma }}W_{\alpha \beta }^{\gamma }\left( \mathbf{Z}%
^{\prime },\mathbf{Z},\mathbf{Z}_{\ast }\right) \frac{M_{\alpha
}^{1/2}M_{\beta \ast }^{1/2}}{\varphi _{\alpha }\left( I\right) \varphi
_{\beta }\left( I_{\ast }\right) }d\mathbf{Z}^{\prime }\text{, and}  \notag
\\
k_{\alpha \gamma }^{3\beta }\left( \mathbf{Z},\mathbf{Z}_{\ast }\right) 
&=&\int_{\mathcal{Z}_{\gamma }}W_{\alpha \gamma }^{\beta }\left( \mathbf{Z}%
_{\ast },\mathbf{Z},\mathbf{Z}^{\prime }\right) \frac{M_{\beta \ast }^{1/2}}{%
\varphi _{\beta }\left( I_{\ast }\right) M_{\alpha }^{1/2}}d\mathbf{Z}%
^{\prime }=k_{\alpha \gamma }^{1\beta }\left( \mathbf{Z}_{\ast },\mathbf{Z}%
\right) \text{.}  \label{ck1}
\end{eqnarray}

\textbf{I. Compactness of }$K_{\beta \gamma }^{1\alpha }=\int_{\mathcal{Z}%
_{\beta }}k_{\beta \gamma }^{1\alpha }\left( \mathbf{Z},\mathbf{Z}_{\ast
}\right) h_{\beta \ast }d\mathbf{Z}_{\ast }$ for $\left( \alpha ,\beta
,\gamma \right) \in \mathcal{C}$.

Under assumption $\left( \ref{est1}\right) $ the\ following bound for the
transition probabilities $W_{\beta \gamma }^{\alpha }=W_{\beta \gamma
}^{\alpha }\left( \mathbf{Z},\mathbf{Z}_{\ast },\mathbf{Z}^{\prime }\right) $%
, see also Remark $\ref{Rem1}$, may be obtained 
\begin{eqnarray*}
&&W_{\beta \gamma }^{\alpha }\left( \mathbf{Z},\mathbf{Z}_{\ast },\mathbf{Z}%
^{\prime }\right)  \\
&=&\frac{m_{\alpha }^{2}m_{\beta }}{\left( m_{\alpha }-m_{\beta }\right) ^{2}%
}\varphi _{\beta }\left( I_{\ast }\right) \varphi _{\gamma }\left( I^{\prime
}\right) \sigma _{\beta \gamma }^{\alpha }\mathbf{\delta }_{1}\left( I-%
\widetilde{E}_{\beta \gamma }^{\prime }\right) \mathbf{\delta }_{3}\left( 
\boldsymbol{\xi }^{\prime }-\mathbf{g}_{\alpha \beta }\right) \mathbf{1}%
_{E_{\alpha }\geq K_{\beta \gamma }^{\alpha }} \\
&\leq &C\frac{\varphi _{\alpha }\left( I\right) \varphi _{\beta }\left(
I_{\ast }\right) \varphi _{\gamma }\left( I^{\prime }\right) }{\mathcal{E}%
_{\beta \gamma }^{\delta ^{\left( \beta \right) }/2}\left( \mathcal{E}%
_{\beta \gamma }^{\ast }\right) ^{\delta ^{\left( \gamma \right) }/2}}\left(
1+\left\vert \boldsymbol{\xi }^{\prime }-\boldsymbol{\xi }_{\ast
}\right\vert ^{\chi -1}\right) \mathbf{\delta }_{1}\left( I-\widetilde{E}%
_{\beta \gamma }^{\prime }\right) \mathbf{\delta }_{3}\left( \boldsymbol{\xi 
}^{\prime }-\mathbf{g}_{\alpha \beta }\right) \mathbf{1}_{E_{\alpha }\geq
K_{\beta \gamma }^{\alpha }} \\
&\leq &C\frac{\varphi _{\alpha }\left( I\right) \varphi _{\beta }\left(
I_{\ast }\right) }{\mathcal{E}_{\beta \gamma }^{\delta ^{\left( \beta
\right) }/2}\left( \mathcal{E}_{\beta \gamma }^{\ast }\right) ^{\delta
^{\left( \gamma \right) }/2}}\left( \frac{\varphi _{\gamma }\left( I^{\prime
}\right) }{M_{\gamma }^{\prime }}\right) ^{1/4}\left( 1+\left\vert \mathbf{g}%
\right\vert ^{\chi -1}\right) \delta _{1}\left( I-\widetilde{E}_{\beta
\gamma }^{\prime }\right)  \\
&&\times \delta _{3}\left( \boldsymbol{\xi }^{\prime }-\mathbf{g}_{\alpha
\beta }\right) \mathbf{1}_{E_{\alpha }\geq K_{\beta \gamma }^{\alpha }}\text{%
, where }\mathbf{g}_{\alpha \beta }=\frac{m_{\alpha }\boldsymbol{\xi }%
-m_{\beta }\boldsymbol{\xi }_{\ast }}{m_{\alpha }-m_{\beta }}\text{ and }%
\mathbf{g}=\boldsymbol{\xi }-\boldsymbol{\xi }_{\ast }\text{.}
\end{eqnarray*}%
It follows immediately, by expression $\left( \ref{est1}\right) $ for $%
k_{\beta \gamma }^{1\alpha }=k_{\beta \gamma }^{1\alpha }\left( \mathbf{Z},%
\mathbf{Z}_{\ast }\right) $, that 
\begin{equation*}
k_{\beta \gamma }^{1\alpha }\leq C\left( \frac{\varphi _{\alpha }\left(
I\right) \varphi _{\beta }\left( I_{\ast }\right) }{\mathcal{E}_{\beta
\gamma }^{\delta ^{\left( \beta \right) }}\left( \mathcal{E}_{\beta \gamma
}^{\ast }\right) ^{\delta ^{\left( \gamma \right) }}}\right) ^{1/2}\left(
1+\left\vert \mathbf{g}\right\vert ^{\chi -1}\right) \left( \frac{M_{\alpha
}\varphi _{\beta }\left( I_{\ast }\right) }{M_{\beta \ast }\varphi _{\alpha
}\left( I\right) }\right) ^{1/4}\mathbf{1}_{E_{\alpha }\geq K_{\beta \gamma
}^{\alpha }}\text{.}
\end{equation*}

Firstly, assume that $\gamma \in \mathcal{I}_{poly}$. By relation $\left( %
\ref{d1}\right) $, since $\delta ^{\left( \gamma \right) }\geq 2$,%
\begin{equation*}
\frac{\delta ^{\left( \alpha \right) }}{2}-1\leq \frac{\delta ^{\left( \beta
\right) }}{2}+\frac{\delta ^{\left( \gamma \right) }}{2}-\frac{3}{2}\leq 
\frac{\delta ^{\left( \beta \right) }}{2}+\delta ^{\left( \gamma \right) }-%
\frac{5}{2}\text{.}
\end{equation*}%
Then, since $I\geq \Delta _{\beta \gamma }^{\alpha }\varepsilon
_{0}=\varepsilon _{\beta 0}+\varepsilon _{\gamma 0}-\varepsilon _{\alpha 0}>0
$,%
\begin{eqnarray*}
k_{\beta \gamma }^{1\alpha } &\leq &C\left( \frac{\varphi _{\beta }\left(
I_{\ast }\right) }{\mathcal{E}_{\beta \gamma }^{\delta ^{\left( \beta
\right) }}}\frac{\varphi _{\alpha }\left( I\right) }{I^{\delta ^{\left(
\gamma \right) }}}\right) ^{1/2}\left( 1+\left\vert \mathbf{g}\right\vert
^{\chi -1}\right) \left( \frac{M_{\alpha }\varphi _{\beta }\left( I_{\ast
}\right) }{M_{\beta \ast }\varphi _{\alpha }\left( I\right) }\right) ^{1/4}%
\mathbf{1}_{E_{\alpha }\geq K_{\beta \gamma }^{\alpha }} \\
&\leq &\frac{C}{I^{3/4}}\left( 1+\left\vert \mathbf{g}\right\vert ^{\chi
-1}\right) \left( \frac{M_{\alpha }\varphi _{\beta }\left( I_{\ast }\right) 
}{M_{\beta \ast }\varphi _{\alpha }\left( I\right) }\right) ^{1/4}\mathbf{1}%
_{E_{\alpha }\geq K_{\beta \gamma }^{\alpha }}\text{, with }\frac{C}{I^{3/4}}%
\leq C\text{.}
\end{eqnarray*}%
Note that%
\begin{eqnarray*}
&&m_{\alpha }\frac{\left\vert \boldsymbol{\xi }\right\vert ^{2}}{2}-m_{\beta
}\frac{\left\vert \boldsymbol{\xi }_{\ast }\right\vert ^{2}}{2}+I-I_{\ast }%
\mathbf{1}_{\beta \in \mathcal{I}_{poly}} \\
&=&\frac{m_{\alpha }-m_{\beta }}{2}\left\vert \mathbf{g}_{\alpha \beta
}\right\vert ^{2}+I-I_{\ast }\mathbf{1}_{\beta \in \mathcal{I}_{poly}}-\frac{%
m_{\alpha }m_{\beta }}{2\left( m_{\alpha }-m_{\beta }\right) }\left\vert 
\mathbf{g}\right\vert ^{2} \\
&\geq &\frac{m_{\alpha }-m_{\beta }}{2}\left\vert \mathbf{g}_{\alpha \beta
}\right\vert ^{2}\text{ }
\end{eqnarray*}%
Then%
\begin{eqnarray*}
&&\int_{\mathbb{R}^{3}\times \mathbb{R}_{+}}k_{\beta \gamma }^{1\alpha
}\left( \mathbf{Z},\mathbf{Z}_{\ast }\right) d\mathbf{Z}=\int_{\mathbb{R}%
^{3}\times \mathbb{R}_{+}}k_{\beta \gamma }^{1\alpha }\left( \mathbf{Z},%
\mathbf{Z}_{\ast }\right) d\boldsymbol{\xi }dI \\
&\leq &C\int_{\mathbb{R}^{3}}\left( 1+\left\vert \mathbf{g}\right\vert
^{\chi -1}\right) \exp \left( -\frac{m_{\alpha }-m_{\beta }}{8k_{B}T}%
\left\vert \mathbf{g}_{\alpha \beta }\right\vert ^{2}\right) d\boldsymbol{%
\xi }\int_{0}^{\infty }\exp \left( -\frac{\widetilde{I}}{4k_{B}T}\right) d%
\widetilde{I} \\
&\leq &C\int_{\mathbb{R}^{3}}\left( 1+\left\vert \mathbf{g}\right\vert
^{\chi -1}\right) \exp \left( -\frac{m_{\alpha }-m_{\beta }}{8k_{B}T}%
\left\vert \mathbf{g}\right\vert ^{2}\right) \mathbf{1}_{\left\vert \mathbf{g%
}\right\vert \leq \left\vert \mathbf{g}_{\alpha \beta }\right\vert }d\mathbf{%
g} \\
&&+C\int_{\mathbb{R}^{3}}\left( 1+\left\vert \mathbf{g}_{\alpha \beta
}\right\vert ^{\chi -1}\right) \exp \left( -\frac{m_{\alpha }-m_{\beta }}{%
8k_{B}T}\left\vert \mathbf{g}_{\alpha \beta }\right\vert ^{2}\right) \mathbf{%
1}_{\left\vert \mathbf{g}_{\alpha \beta }\right\vert \leq \left\vert \mathbf{%
g}\right\vert }d\mathbf{g}_{\alpha \beta } \\
&\leq &C\int_{0}^{\infty }\left( R^{2}+R^{\chi +1}\right) \exp \left( -\frac{%
m_{\alpha }-m_{\beta }}{8k_{B}T}R^{2}\right) dR=C
\end{eqnarray*}%
is bounded (in $\mathbf{Z}_{\ast }$).

Furthermore, it follows, for $\mathfrak{h}_{N}$ given by notation $\left( %
\ref{rd}\right) $,\ that%
\begin{eqnarray}
&&\left( k_{\beta \gamma }^{1\alpha }\left( \mathbf{Z},\mathbf{Z}_{\ast
}\right) \right) ^{2}\mathbf{1}_{\mathfrak{h}_{N}}  \notag \\
&\leq &C\frac{\varphi _{\beta }\left( I_{\ast }\right) }{\mathcal{E}_{\beta
\gamma }^{\delta ^{\left( \beta \right) }}}\frac{\varphi _{\alpha }\left(
I\right) }{I^{\delta ^{\left( \gamma \right) }}}\left( 1+\left\vert \mathbf{g%
}\right\vert ^{2\chi -2}\right) \left( \frac{M_{\alpha }\varphi _{\beta
}\left( I_{\ast }\right) }{M_{\beta \ast }\varphi _{\alpha }\left( I\right) }%
\right) ^{1/2}\mathbf{1}_{E_{\alpha }\geq K_{\beta \gamma }^{\alpha }}%
\mathbf{1}_{\mathfrak{h}_{N}}  \notag \\
&\leq &C\frac{\varphi _{\beta }\left( I_{\ast }\right) }{\mathcal{E}_{\beta
\gamma }^{\delta ^{\left( \beta \right) }}}\frac{\varphi _{\alpha }\left(
I\right) }{I^{\delta ^{\left( \gamma \right) }}}\left( 1+N^{2-2\chi }\right)
\exp \left( -\frac{m_{\alpha }-m_{\beta }}{4k_{B}T}\left\vert \mathbf{g}%
_{\alpha \beta }\right\vert ^{2}\right) \mathbf{1}_{E_{\alpha }\geq K_{\beta
\gamma }^{\alpha }}\mathbf{1}_{\mathfrak{h}_{N}}  \notag \\
&\leq &CN^{2}\frac{\varphi _{\beta }\left( I_{\ast }\right) }{\mathcal{E}%
_{\beta \gamma }^{\delta ^{\left( \beta \right) }/2}}\frac{1}{I^{3/2}}\exp
\left( -\frac{m_{\alpha }-m_{\beta }}{4k_{B}T}\left\vert \mathbf{g}_{\alpha
\beta }\right\vert ^{2}\right) \mathbf{1}_{E_{\alpha }\geq K_{\beta \gamma
}^{\alpha }}\mathbf{1}_{\mathfrak{h}_{N}}\text{.}  \label{ie5}
\end{eqnarray}%
Hence, $k_{\beta \gamma }^{1\alpha }\left( \mathbf{Z},\mathbf{Z}_{\ast
}\right) \mathbf{1}_{\mathfrak{h}_{N}}\in L^{2}\left( d\mathbf{Z}\boldsymbol{%
\,}d\mathbf{Z}_{\ast }\right) $ for any (nonzero) natural number $N$, since%
\begin{eqnarray*}
&&\int_{\mathcal{Z}_{\alpha }\times \mathcal{Z}_{\beta }}\left( k_{\beta
\gamma }^{1\alpha }\left( \mathbf{Z},\mathbf{Z}_{\ast }\right) \right) ^{2}%
\mathbf{1}_{\mathfrak{h}_{N}}d\mathbf{Z}d\mathbf{Z}_{\ast } \\
&\leq &CN^{2}\int\limits_{\mathbb{R}^{3}}\exp \left( -\frac{m_{\alpha
}-m_{\beta }}{4k_{B}T}\left\vert \mathbf{g}_{\alpha \beta }\right\vert
^{2}\right) d\mathbf{g}_{\alpha \beta }\int\limits_{\left\vert \boldsymbol{%
\xi }\right\vert \leq N}d\boldsymbol{\xi }\int\limits_{0}^{\mathcal{E}%
_{\beta \gamma }}\frac{I_{\ast }^{\delta ^{\left( \beta \right) }/2-1}}{%
\mathcal{E}_{\beta \gamma }^{\delta ^{\left( \beta \right) }/2}}dI_{\ast
}\int\limits_{\Delta _{\beta \gamma }^{\alpha }\varepsilon _{0}}^{\mathcal{%
\infty }}\frac{1}{I^{3/2}}dI \\
&=&CN^{5}\text{.}
\end{eqnarray*}%
Note that%
\begin{equation*}
\frac{\left\vert m_{\alpha }\boldsymbol{\xi }-m_{\beta }\boldsymbol{\xi }%
_{\ast }\right\vert ^{2}}{2\left( m_{\alpha }-m_{\beta }\right) }\geq
m_{\beta }\left\vert \boldsymbol{\xi }\right\vert \left\vert \mathbf{g}%
\right\vert \left( 1+\cos \phi \right) \text{, where }\cos \phi =\frac{%
\boldsymbol{\xi }}{\left\vert \boldsymbol{\xi }\right\vert }\cdot \frac{%
\mathbf{g}}{\left\vert \mathbf{g}\right\vert }\text{,}
\end{equation*}%
since for any $\left( \mathbf{x},\mathbf{y}\right) \in \left( \mathbb{R}%
^{3}\right) ^{2}$%
\begin{eqnarray*}
\left\vert \mathbf{x}\pm \mathbf{y}\right\vert ^{2} &=&\left\vert \mathbf{x}%
\right\vert ^{2}\pm 2\left\vert \mathbf{x}\right\vert \left\vert \mathbf{y}%
\right\vert \cos \theta +\left\vert \mathbf{y}\right\vert ^{2} \\
&=&\left\vert \left\vert \mathbf{x}\right\vert -\left\vert \mathbf{y}%
\right\vert \right\vert ^{2}+2\left\vert \mathbf{x}\right\vert \left\vert 
\mathbf{y}\right\vert \left( 1\pm \cos \theta \right)  \\
&\geq &2\left\vert \mathbf{x}\right\vert \left\vert \mathbf{y}\right\vert
\left( 1\pm \cos \theta \right) \text{, where }\cos \theta =\frac{\mathbf{x}%
}{\left\vert \mathbf{x}\right\vert }\cdot \frac{\mathbf{y}}{\left\vert 
\mathbf{y}\right\vert }\text{.}
\end{eqnarray*}%
Then 
\begin{eqnarray*}
k_{\beta \gamma }^{1\alpha }\left( \mathbf{Z},\mathbf{Z}_{\ast }\right) 
&\leq &C\frac{\left( \varphi _{\beta }\left( I_{\ast }\right) \right) ^{1/2}%
}{\mathcal{E}_{\beta \gamma }^{\delta ^{\left( \beta \right) }/4+1/2}}\left(
1+\left\vert \mathbf{g}\right\vert ^{\chi -1}\right) \exp \left( -\dfrac{%
m_{\beta }}{8k_{B}T}\left\vert \boldsymbol{\xi }\right\vert \left\vert 
\mathbf{g}\right\vert \left( 1+\cos \phi \right) \right)  \\
&&\times \exp \left( -\frac{m_{\alpha }-m_{\beta }}{16k_{B}T}\left\vert 
\mathbf{g}_{\alpha \beta }\right\vert ^{2}\right) .
\end{eqnarray*}%
It follows that,%
\begin{eqnarray*}
&&\sup \int_{\mathcal{Z}_{\beta }}k_{\beta \gamma }^{1\alpha }\left( \mathbf{%
Z},\mathbf{Z}_{\ast }\right) -k_{\beta \gamma }^{1\alpha }\left( \mathbf{Z},%
\mathbf{Z}_{\ast }\right) \mathbf{1}_{\mathfrak{h}_{N}}d\mathbf{Z}_{\ast } \\
&\leq &C\int_{0}^{\mathcal{E}_{\beta \gamma }}\frac{I_{\ast }^{\delta
^{\left( \beta \right) }/4-1/2}}{\mathcal{E}_{\beta \gamma }^{\delta
^{\left( \beta \right) }/4+1/2}}dI_{\ast }\sup \left( \int_{\left\vert 
\mathbf{g}\right\vert \leq \frac{1}{N}}1+\left\vert \mathbf{g}\right\vert
^{\chi -1}d\mathbf{g}+\int_{\left\vert \mathbf{g}\right\vert \geq \frac{1}{N}%
\text{, }\left\vert \boldsymbol{\xi }\right\vert \geq N\geq 1}\left(
1+\left\vert \mathbf{g}\right\vert ^{\chi -1}\right) \right.  \\
&&\left. \exp \left( -\frac{m_{\alpha }-m_{\beta }}{16k_{B}T}\left\vert 
\mathbf{g}_{\alpha \beta }\right\vert ^{2}\right) \exp \left( -\dfrac{%
m_{\beta }}{8k_{B}T}\left\vert \boldsymbol{\xi }\right\vert \left\vert 
\mathbf{g}\right\vert \left( 1+\cos \phi \right) \right) d\boldsymbol{\xi }%
_{\ast }\right)  \\
&\leq &C\left( \int_{0}^{1/N}R^{\chi +1}dR+\frac{1}{N}\int_{\mathbb{R}%
^{3}}\left( \left\vert \mathbf{g}\right\vert ^{-1}+\left\vert \mathbf{g}%
\right\vert ^{\chi -2}\right) \exp \left( -\frac{m_{\alpha }-m_{\beta }}{%
16k_{B}T}\left\vert \mathbf{g}_{\alpha \beta }\right\vert ^{2}\right)
\right.  \\
&&\times \left. \left\vert \boldsymbol{\xi }\right\vert \left\vert \mathbf{g}%
\right\vert \exp \left( -\dfrac{m_{\beta }}{8k_{B}T}\left\vert \boldsymbol{%
\xi }\right\vert \left\vert \mathbf{g}\right\vert \left( 1+\cos \phi \right)
\right) d\boldsymbol{\xi }_{\ast }\right)  \\
&\leq &\frac{C}{N^{2+\chi }}+\frac{C}{N}\int_{0}^{\pi }\left\vert 
\boldsymbol{\xi }\right\vert \left\vert \mathbf{g}\right\vert \exp \left( -%
\dfrac{m_{\beta }}{8k_{B}T}\left\vert \boldsymbol{\xi }\right\vert
\left\vert \mathbf{g}\right\vert \left( 1+\cos \phi \right) \right) \sin
\phi d\phi  \\
&&\times \left( \int_{\mathbb{R}^{3}}\left( \left\vert \mathbf{g}\right\vert
+\left\vert \mathbf{g}\right\vert ^{\chi }\right) \exp \left( -\frac{%
m_{\alpha }-m_{\beta }}{16k_{B}T}\left\vert \mathbf{g}\right\vert
^{2}\right) d\left\vert \mathbf{g}\right\vert \right.  \\
&&\left. +\int_{\mathbb{R}^{3}}\left( \left\vert \mathbf{g}_{\alpha \beta
}\right\vert +\left\vert \mathbf{g}_{\alpha \beta }\right\vert ^{\chi
}\right) \exp \left( -\frac{m_{\alpha }-m_{\beta }}{16k_{B}T}\left\vert 
\mathbf{g}_{\alpha \beta }\right\vert ^{2}\right) d\left\vert \mathbf{g}%
_{\alpha \beta }\right\vert \right)  \\
&\leq &C\left( \frac{1}{N^{2}}+\frac{1}{N}\right) \rightarrow 0\text{ as }%
N\rightarrow \infty .
\end{eqnarray*}%
By Lemma $\ref{LGD}$, $K_{\beta \gamma }^{1\alpha }=\int_{\mathcal{Z}_{\beta
}}k_{\beta \gamma }^{1\alpha }\left( \mathbf{Z},\mathbf{Z}_{\ast }\right)
h_{\beta \ast }d\mathbf{Z}_{\ast }$ is compact on $L^{2}\left( d\mathbf{Z}%
\right) $ for any $\left( \alpha ,\beta ,\gamma \right) \in \mathcal{C}$
such that $\gamma \in \mathcal{I}_{poly}$.

Now, assume that $\left( \beta ,\gamma \right) \in \mathcal{I}_{mono}^{2}$.
Then%
\begin{eqnarray*}
&&k_{\beta \gamma }^{1\alpha }\left( \mathbf{Z},\mathbf{Z}_{\ast }\right)  \\
&\leq &CI^{1/4}\left( 1+\left\vert \mathbf{g}\right\vert ^{\chi -1}\right)
\left( \frac{M_{\alpha }}{M_{\beta \ast }\varphi _{\alpha }\left( I\right) }%
\right) ^{1/2}\mathbf{\delta }_{1}\left( I-\widetilde{E}_{\beta \gamma
}^{\prime }\right) \mathbf{1}_{E_{\alpha }\geq K_{\beta \gamma }^{\alpha }}
\\
&\leq &C\left( 1+\left\vert \mathbf{g}\right\vert ^{1/2}\right) \left(
1+\left\vert \mathbf{g}\right\vert ^{\chi -1}\right) \left( \frac{M_{\alpha }%
}{M_{\beta \ast }\varphi _{\alpha }\left( I\right) }\right) ^{1/2}\mathbf{%
\delta }_{1}\left( I-\widetilde{E}_{\beta \gamma }^{\prime }\right) \mathbf{1%
}_{E_{\alpha }\geq K_{\beta \gamma }^{\alpha }} \\
&\leq &C\left( \left\vert \mathbf{g}\right\vert ^{1/2}+\left\vert \mathbf{g}%
\right\vert ^{\chi -1}\right) \exp \left( -\frac{m_{\alpha }-m_{\beta }}{%
4k_{B}T}\left\vert \mathbf{g}_{\alpha \beta }\right\vert ^{2}\right) \mathbf{%
\delta }_{1}\left( I-\widetilde{E}_{\beta \gamma }^{\prime }\right) \mathbf{1%
}_{E_{\alpha }\geq K_{\beta \gamma }^{\alpha }}\text{,}
\end{eqnarray*}%
implies that%
\begin{eqnarray*}
&&\int_{\mathbb{R}^{3}\times \mathbb{R}_{+}}k_{\beta \gamma }^{1\alpha
}\left( \mathbf{Z},\mathbf{Z}_{\ast }\right) d\mathbf{Z}=\int_{\mathbb{R}%
^{3}\times \mathbb{R}_{+}}k_{\beta \gamma }^{1\alpha }\left( \mathbf{Z},%
\mathbf{Z}_{\ast }\right) d\boldsymbol{\xi }dI \\
&\leq &C\int_{\mathbb{R}^{3}}\left( \left\vert \mathbf{g}_{\alpha \beta
}\right\vert ^{1/2}+\left\vert \boldsymbol{\xi }\right\vert
^{1/2}+\left\vert \mathbf{g}\right\vert ^{\chi -1}\right) \exp \left( -\frac{%
m_{\alpha }-m_{\beta }}{4k_{B}T}\left\vert \mathbf{g}_{\alpha \beta
}\right\vert ^{2}\right) d\boldsymbol{\xi }\int_{0}^{\infty }\mathbf{\delta }%
_{1}\left( I-\widetilde{E}_{\beta \gamma }^{\prime }\right) dI \\
&\leq &C\int_{\mathbb{R}^{3}}\left( 1+\left\vert \mathbf{g}_{\alpha \beta
}\right\vert ^{1/2}+\left\vert \mathbf{g}_{\alpha \beta }\right\vert ^{\chi
-1}\right) \exp \left( -\frac{m_{\alpha }-m_{\beta }}{4k_{B}T}\left\vert 
\mathbf{g}_{\alpha \beta }\right\vert ^{2}\right) d\mathbf{g}_{\alpha \beta }
\\
&&+C\int_{\mathbb{R}^{3}}\left\vert \mathbf{g}\right\vert ^{\chi -1}\exp
\left( -\frac{m_{\alpha }-m_{\beta }}{4k_{B}T}\left\vert \mathbf{g}%
\right\vert ^{2}\right) d\mathbf{g} \\
&&+C\int_{\mathbb{R}^{3}}\left\vert \boldsymbol{\xi }\right\vert \left\vert 
\mathbf{g}\right\vert \exp \left( -\dfrac{m_{\beta }}{4k_{B}T}\left\vert 
\boldsymbol{\xi }\right\vert \left\vert \mathbf{g}\right\vert \left( 1+\cos
\phi \right) \right) \frac{1}{\left\vert \mathbf{g}\right\vert }\exp \left( -%
\frac{m_{\alpha }-m_{\beta }}{8k_{B}T}\left\vert \mathbf{g}_{\alpha \beta
}\right\vert ^{2}\right) d\boldsymbol{\xi } \\
&\leq &C\int_{0}^{\infty }\left( R^{5/2}+R^{2}+R^{\chi +1}\right) \exp
\left( -\frac{m_{\alpha }-m_{\beta }}{4k_{B}T}R^{2}\right) dR \\
&&+\int_{0}^{\pi }\left\vert \boldsymbol{\xi }\right\vert \left\vert \mathbf{%
g}\right\vert \exp \left( -\dfrac{m_{\beta }}{4k_{B}T}\left\vert \boldsymbol{%
\xi }\right\vert \left\vert \mathbf{g}\right\vert \left( 1+\cos \phi \right)
\right) \sin \phi d\phi \left( \int_{\mathbb{R}^{3}}\left\vert \mathbf{g}%
\right\vert \exp \left( -\frac{m_{\alpha }-m_{\beta }}{8k_{B}T}\left\vert 
\mathbf{g}\right\vert ^{2}\right) d\left\vert \mathbf{g}\right\vert \right. 
\\
&&\left. +\int_{\mathbb{R}^{3}}\left\vert \mathbf{g}_{\alpha \beta
}\right\vert \exp \left( -\frac{m_{\alpha }-m_{\beta }}{16k_{B}T}\left\vert 
\mathbf{g}_{\alpha \beta }\right\vert ^{2}\right) d\left\vert \mathbf{g}%
_{\alpha \beta }\right\vert \right) =C\text{.}
\end{eqnarray*}%
Furthermore, it follows, for $\mathfrak{h}_{N}$ given by notation $\left( %
\ref{rd}\right) $,\ that%
\begin{eqnarray*}
\left( k_{\beta \gamma }^{1\alpha }\left( \mathbf{Z},\mathbf{Z}_{\ast
}\right) \right) ^{2}\mathbf{1}_{\mathfrak{h}_{N}} &\leq &CI^{1/2}\left(
1+\left\vert \mathbf{g}\right\vert ^{2\chi -2}\right) \frac{M_{\alpha }}{%
M_{\beta \ast }\varphi _{\alpha }\left( I\right) }\mathbf{\delta }_{1}\left(
I-\widetilde{E}_{\beta \gamma }^{\prime }\right) \mathbf{1}_{\mathfrak{h}%
_{N}} \\
&\leq &C(1+N)\left( 1+N^{2-2\chi }\right) \frac{\left\vert \mathbf{g}%
\right\vert M_{\alpha }}{M_{\beta \ast }\varphi _{\alpha }\left( I\right) }%
\mathbf{\delta }_{1}\left( I-\widetilde{E}_{\beta \gamma }^{\prime }\right) 
\mathbf{1}_{\mathfrak{h}_{N}} \\
&\leq &CN^{3}\left( \left\vert \boldsymbol{\xi }\right\vert +\left\vert 
\mathbf{g}_{\alpha \beta }\right\vert \right) \exp \left( -\frac{m_{\alpha
}-m_{\beta }}{2k_{B}T}\left\vert \mathbf{g}_{\alpha \beta }\right\vert
^{2}\right) \mathbf{\delta }_{1}\left( I-\widetilde{E}_{\beta \gamma
}^{\prime }\right) \mathbf{1}_{\mathfrak{h}_{N}} \\
&\leq &CN^{4}\left( 1+\left\vert \mathbf{g}_{\alpha \beta }\right\vert
\right) \exp \left( -\frac{m_{\alpha }-m_{\beta }}{2k_{B}T}\left\vert 
\mathbf{g}_{\alpha \beta }\right\vert ^{2}\right) \mathbf{\delta }_{1}\left(
I-\widetilde{E}_{\beta \gamma }^{\prime }\right) \mathbf{1}_{\mathfrak{h}%
_{N}}\text{.}
\end{eqnarray*}%
Hence, $k_{\beta \gamma }^{1\alpha }\left( \mathbf{Z},\mathbf{Z}_{\ast
}\right) \mathbf{1}_{\mathfrak{h}_{N}}\in L^{2}\left( d\mathbf{Z}\boldsymbol{%
\,}d\mathbf{Z}_{\ast }\right) $ for any (nonzero) natural number $N$, since%
\begin{eqnarray*}
&&\int_{\left( \mathbb{R}^{3}\right) ^{2}\times \mathbb{R}_{+}}\left(
k_{\beta \gamma }^{1\alpha }\left( \mathbf{Z},\mathbf{Z}_{\ast }\right)
\right) ^{2}\mathbf{1}_{\mathfrak{h}_{N}}d\boldsymbol{\xi }d\boldsymbol{\xi }%
_{\ast }dI \\
&\leq &CN^{4}\int_{\left\vert \boldsymbol{\xi }\right\vert \leq N}d%
\boldsymbol{\xi }\int_{\mathbb{R}^{3}}\left( 1+\left\vert \mathbf{g}_{\alpha
\beta }\right\vert \right) \exp \left( -\frac{m_{\alpha }-m_{\beta }}{2k_{B}T%
}\left\vert \mathbf{g}_{\alpha \beta }\right\vert ^{2}\right) d\mathbf{g}%
_{\alpha \beta }\int_{0}^{\infty }\mathbf{\delta }_{1}\left( I-\widetilde{E}%
_{\beta \gamma }^{\prime }\right) dI \\
&\leq &CN^{7}\text{.}
\end{eqnarray*}%
Furthermore, 
\begin{eqnarray*}
k_{\beta \gamma }^{1\alpha }\left( \mathbf{Z},\mathbf{Z}_{\ast }\right) 
&\leq &C\left( \left\vert \mathbf{g}\right\vert ^{1/2}+\left\vert \mathbf{g}%
\right\vert ^{\chi -1}\right) \exp \left( -\frac{m_{\alpha }-m_{\beta }}{%
4k_{B}T}\left\vert \mathbf{g}_{\alpha \beta }\right\vert ^{2}\right)  \\
&&\times \exp \left( -\dfrac{m_{\beta }}{4k_{B}T}\left\vert \boldsymbol{\xi }%
\right\vert \left\vert \mathbf{g}\right\vert \left( 1+\cos \phi \right)
\right) \mathbf{\delta }_{1}\left( I-\widetilde{E}_{\beta \gamma }^{\prime
}\right) . \\
&\leq &C\left( \left\vert \boldsymbol{\xi }\right\vert ^{1/2}+\left\vert 
\mathbf{g}_{\alpha \beta }\right\vert ^{1/2}+\left\vert \mathbf{g}%
\right\vert ^{\chi -1}\right) \exp \left( -\frac{m_{\alpha }-m_{\beta }}{%
4k_{B}T}\left\vert \mathbf{g}_{\alpha \beta }\right\vert ^{2}\right)  \\
&&\times \exp \left( -\dfrac{m_{\beta }}{4k_{B}T}\left\vert \boldsymbol{\xi }%
\right\vert \left\vert \mathbf{g}\right\vert \left( 1+\cos \phi \right)
\right) \mathbf{\delta }_{1}\left( I-\widetilde{E}_{\beta \gamma }^{\prime
}\right)  \\
&\leq &C\left( \left\vert \boldsymbol{\xi }\right\vert ^{1/2}+1+\left\vert 
\mathbf{g}\right\vert ^{\chi -1}\right) \exp \left( -\dfrac{m_{\beta }}{%
4k_{B}T}\left\vert \boldsymbol{\xi }\right\vert \left\vert \mathbf{g}%
\right\vert \left( 1+\cos \phi \right) \right) \mathbf{\delta }_{1}\left( I-%
\widetilde{E}_{\beta \gamma }^{\prime }\right) \text{.}
\end{eqnarray*}%
It follows that,%
\begin{eqnarray*}
&&\sup \int_{\mathbb{R}^{3}}k_{\beta \gamma }^{1\alpha }\left( \mathbf{Z},%
\mathbf{Z}_{\ast }\right) -k_{\beta \gamma }^{1\alpha }\left( \mathbf{Z},%
\mathbf{Z}_{\ast }\right) \mathbf{1}_{\mathfrak{h}_{N}}d\boldsymbol{\xi } \\
&\leq &C\sup \left( \int_{\left\vert \mathbf{g}\right\vert \leq \frac{1}{N}%
}\left( \left\vert \mathbf{g}\right\vert ^{1/2}+\left\vert \mathbf{g}%
\right\vert ^{\chi -1}\right) \mathbf{\delta }_{1}\left( I-\widetilde{E}%
_{\beta \gamma }^{\prime }\right) d\mathbf{g}+\int_{\left\vert \mathbf{g}%
\right\vert \geq \frac{1}{N}\text{, }\left\vert \boldsymbol{\xi }\right\vert
\geq N\geq 1}\mathbf{\delta }_{1}\left( I-\widetilde{E}_{\beta \gamma
}^{\prime }\right) \right.  \\
&&\left. \left( \left\vert \boldsymbol{\xi }\right\vert ^{1/2}+1+\left\vert 
\mathbf{g}\right\vert ^{\chi -1}\right) \exp \left( -\dfrac{m_{\beta }}{%
4k_{B}T}\left\vert \boldsymbol{\xi }\right\vert \left\vert \mathbf{g}%
\right\vert \left( 1+\cos \phi \right) \right) d\mathbf{g}\right)  \\
&\leq &C\sup \left( \int_{\left\vert \mathbf{g}\right\vert \leq \frac{1}{N}%
}\left( \left\vert \mathbf{g}\right\vert ^{3/2}+\left\vert \mathbf{g}%
\right\vert ^{\chi }\right) \mathbf{\delta }_{1}\left( I-\widetilde{E}%
_{\beta \gamma }^{\prime }\right) d\left\vert \mathbf{g}\right\vert
^{2}+\left( \frac{1}{N^{1/2}}+\frac{1}{N}+\frac{1}{N^{\chi }}\right) \right. 
\\
&&\left. \times \int_{0}^{\infty }\int_{0}^{\pi }\left\vert \boldsymbol{\xi }%
\right\vert \left\vert \mathbf{g}\right\vert \exp \left( -\dfrac{m_{\beta }}{%
4k_{B}T}\left\vert \boldsymbol{\xi }\right\vert \left\vert \mathbf{g}%
\right\vert \left( 1+\cos \phi \right) \right) \sin \phi d\phi \delta
_{1}\left( I-\widetilde{E}_{\beta \gamma }^{\prime }\right) d\left\vert 
\mathbf{g}\right\vert ^{2}\right)  \\
&\leq &C\left( \frac{1}{N^{3/2}}+\frac{1}{N^{\chi }}+\frac{1}{N^{1/2}}+\frac{%
1}{N^{\chi }}\right) \int_{0}^{\infty }\mathbf{\delta }_{1}\left( I-%
\widetilde{E}_{\beta \gamma }^{\prime }\right) d\left\vert \mathbf{g}%
\right\vert ^{2} \\
&\leq &C\left( \frac{1}{N^{1/2}}+\frac{1}{N^{\chi }}\right) \rightarrow 0%
\text{ as }N\rightarrow \infty .
\end{eqnarray*}%
Now by Lemma $\ref{LGD}$, $K_{\beta \gamma }^{1\alpha }=\sum_{\beta ,\gamma
=1}^{s}\int_{\mathcal{Z}_{\beta }}k_{\beta \gamma }^{1\alpha }\left( \mathbf{%
Z},\mathbf{Z}_{\ast }\right) h_{\beta \ast }d\mathbf{Z}_{\ast }$ is compact
on $L^{2}\left( d\mathbf{Z}\right) $ for any $\left( \alpha ,\beta ,\gamma
\right) \in \mathcal{C}$ such that $\left( \beta ,\gamma \right) \in 
\mathcal{I}_{mono}^{2}$.

Finally, assume that $\left( \gamma ,\beta \right) \in \mathcal{I}%
_{mono}\times \mathcal{I}_{poly}$. Then%
\begin{eqnarray*}
&&k_{\beta \gamma }^{1\alpha }\left( \mathbf{Z},\mathbf{Z}_{\ast }\right)  \\
&\leq &C\frac{\left( \varphi _{\beta }\left( I_{\ast }\right) \right) ^{1/2}%
}{\mathcal{E}_{\beta \gamma }^{\delta ^{\left( \beta \right) }/4}}\left(
1+\left\vert \mathbf{g}\right\vert ^{\chi -1}\right) \exp \left( -\frac{%
m_{\alpha }-m_{\beta }}{4k_{B}T}\left\vert \mathbf{g}_{\alpha \beta
}\right\vert ^{2}\right) \mathbf{\delta }_{1}\left( I-\widetilde{E}_{\beta
\gamma }^{\prime }\right) \mathbf{1}_{E_{\alpha }\geq K_{\beta \gamma
}^{\alpha }}\text{,}
\end{eqnarray*}%
implies that%
\begin{eqnarray*}
&&\int_{\mathbb{R}^{3}\times \mathbb{R}_{+}}k_{\beta \gamma }^{1\alpha
}\left( \mathbf{Z},\mathbf{Z}_{\ast }\right) d\mathbf{Z}=\int_{\mathbb{R}%
^{3}\times \mathbb{R}_{+}}k_{\beta \gamma }^{1\alpha }\left( \mathbf{Z},%
\mathbf{Z}_{\ast }\right) d\boldsymbol{\xi }dI \\
&\leq &C\int_{\mathbb{R}^{3}}\left( 1+\left\vert \mathbf{g}\right\vert
^{\chi -1}\right) \exp \left( -\frac{m_{\alpha }-m_{\beta }}{4k_{B}T}%
\left\vert \mathbf{g}_{\alpha \beta }\right\vert ^{2}\right) d\boldsymbol{%
\xi }\int_{0}^{\infty }\mathbf{\delta }_{1}\left( I-\widetilde{E}_{\beta
\gamma }^{\prime }\right) dI \\
&\leq &C\int_{\mathbb{R}^{3}}\left( 1+\left\vert \mathbf{g}\right\vert
^{\chi -1}\right) \exp \left( -\frac{m_{\alpha }-m_{\beta }}{4k_{B}T}%
\left\vert \mathbf{g}\right\vert ^{2}\right) d\mathbf{g} \\
&&+C\int_{\mathbb{R}^{3}}\left( 1+\left\vert \mathbf{g}_{\alpha \beta
}\right\vert ^{\chi -1}\right) \exp \left( -\frac{m_{\alpha }-m_{\beta }}{%
4k_{B}T}\left\vert \mathbf{g}_{\alpha \beta }\right\vert ^{2}\right) d%
\mathbf{g}_{\alpha \beta } \\
&\leq &C\int_{0}^{\infty }\left( R^{2}+R^{\chi +1}\right) e^{-R^{2}}dR=C%
\text{.}
\end{eqnarray*}%
It follows, for $\mathfrak{h}_{N}$ given by notation $\left( \ref{rd}\right) 
$,\ that 
\begin{eqnarray*}
&&\left( k_{\beta \gamma }^{1\alpha }\left( \mathbf{Z},\mathbf{Z}_{\ast
}\right) \right) ^{2}\mathbf{1}_{\mathfrak{h}_{N}} \\
&\leq &C\frac{\varphi _{\beta }\left( I_{\ast }\right) }{\mathcal{E}_{\beta
\gamma }^{\delta ^{\left( \beta \right) }/2}}\left( 1+N^{2-2\chi }\right)
\exp \left( -\frac{m_{\alpha }-m_{\beta }}{2k_{B}T}\left\vert \mathbf{g}%
_{\alpha \beta }\right\vert ^{2}\right) \mathbf{\delta }_{1}\left( I-%
\widetilde{E}_{\beta \gamma }^{\prime }\right) \mathbf{1}_{\mathfrak{h}_{N}}.
\end{eqnarray*}%
Hence, $k_{\beta \gamma }^{1\alpha }\left( \mathbf{Z},\mathbf{Z}_{\ast
}\right) \mathbf{1}_{\mathfrak{h}_{N}}\in L^{2}\left( d\mathbf{Z}\boldsymbol{%
\,}d\mathbf{Z}_{\ast }\right) $ for any (nonzero) natural number $N$, since%
\begin{eqnarray*}
&&\int_{\mathcal{Z}_{\alpha }\times \mathcal{Z}_{\beta }}\left( k_{\beta
\gamma }^{1\alpha }\left( \mathbf{Z},\mathbf{Z}_{\ast }\right) \right) ^{2}%
\mathbf{1}_{\mathfrak{h}_{N}}d\mathbf{Z}d\mathbf{Z}_{\ast } \\
&\leq &CN^{2}\int\limits_{\mathbb{R}^{3}}\exp \left( -\frac{m_{\alpha
}-m_{\beta }}{2k_{B}T}\left\vert \mathbf{g}_{\alpha \beta }\right\vert
^{2}\right) d\mathbf{g}_{\alpha \beta }\int\limits_{\left\vert \boldsymbol{%
\xi }\right\vert \leq N}d\boldsymbol{\xi }\int\limits_{0}^{\mathcal{E}%
_{\beta \gamma }}\frac{I_{\ast }^{\delta ^{\left( \beta \right) }/2-1}}{%
\mathcal{E}_{\beta \gamma }^{\delta ^{\left( \beta \right) }/2}}dI_{\ast
}\int\limits_{0}^{\mathcal{\infty }}\mathbf{\delta }_{1}\left( I-\widetilde{E%
}_{\beta \gamma }^{\prime }\right) dI \\
&=&CN^{5}\text{.}
\end{eqnarray*}%
Furthermore, 
\begin{equation*}
k_{\beta \gamma }^{1\alpha }\left( \mathbf{Z},\mathbf{Z}_{\ast }\right) \leq
C\left( 1+\left\vert \mathbf{g}\right\vert ^{\chi -1}\right) \exp \left( -%
\dfrac{m_{\beta }}{4k_{B}T}\left\vert \boldsymbol{\xi }\right\vert
\left\vert \mathbf{g}\right\vert \left( 1+\cos \phi \right) \right) \mathbf{%
\delta }_{1}\left( I-\widetilde{E}_{\beta \gamma }^{\prime }\right) \text{.}
\end{equation*}%
It follows that,%
\begin{eqnarray*}
&&\sup \int_{\mathbb{R}^{3}\times \mathbb{R}_{+}}k_{\beta \gamma }^{1\alpha
}\left( \mathbf{Z},\mathbf{Z}_{\ast }\right) -k_{\beta \gamma }^{1\alpha
}\left( \mathbf{Z},\mathbf{Z}_{\ast }\right) \mathbf{1}_{\mathfrak{h}_{N}}d%
\boldsymbol{\xi }_{\ast }dI_{\ast } \\
&\leq &C\int_{0}^{\mathcal{\infty }}\mathbf{\delta }_{1}\left( I-\widetilde{E%
}_{\beta \gamma }^{\prime }\right) dI_{\ast }\sup \left( \int_{\left\vert 
\mathbf{g}\right\vert \leq \frac{1}{N}}1+\left\vert \mathbf{g}\right\vert
^{\chi -1}d\mathbf{g}+\int_{\left\vert \mathbf{g}\right\vert \geq \frac{1}{N}%
\text{, }\left\vert \boldsymbol{\xi }\right\vert \geq N\geq 1}\left(
1+\left\vert \mathbf{g}\right\vert ^{\chi -1}\right) \right.  \\
&&\left. \exp \left( -\frac{m_{\alpha }-m_{\beta }}{8k_{B}T}\left\vert 
\mathbf{g}\right\vert ^{2}\right) \exp \left( -\dfrac{m_{\beta }}{4k_{B}T}%
\left\vert \boldsymbol{\xi }\right\vert \left\vert \mathbf{g}\right\vert
\left( 1+\cos \phi \right) \right) d\boldsymbol{\xi }_{\ast }\right)  \\
&\leq &\frac{C}{N}\rightarrow 0\text{ as }N\rightarrow \infty .
\end{eqnarray*}%
Then by Lemma $\ref{LGD}$, $K_{\beta \gamma }^{1\alpha }=\int_{\mathcal{Z}%
_{\beta }}k_{\beta \gamma }^{1\alpha }\left( \mathbf{Z},\mathbf{Z}_{\ast
}\right) h_{\beta \ast }d\mathbf{Z}_{\ast }$ is compact on $L^{2}\left( d%
\mathbf{Z}\right) $ for any $\left( \alpha ,\beta ,\gamma \right) \in 
\mathcal{C}$ such that $\left( \gamma ,\beta \right) \in \mathcal{I}%
_{mono}\times \mathcal{I}_{poly}$.

Concluding, $K_{\beta \gamma }^{1\alpha }=\int_{\mathcal{Z}_{\beta
}}k_{\beta \gamma }^{1\alpha }\left( \mathbf{Z},\mathbf{Z}_{\ast }\right)
h_{\beta \ast }d\mathbf{Z}_{\ast }$ is compact on $L^{2}\left( d\mathbf{Z}%
\right) $ for any $\left( \alpha ,\beta ,\gamma \right) \in \mathcal{C}$.

\textbf{II. Compactness of }$K_{\alpha \beta }^{2\gamma }=\int_{\mathcal{Z}%
_{\beta }}k_{\alpha \beta }^{2\gamma }\left( \mathbf{Z},\mathbf{Z}_{\ast
}\right) h_{\beta \ast }d\mathbf{Z}_{\ast }$ for $\left( \gamma ,\alpha
,\beta \right) \in \mathcal{C}$.

Under assumption $\left( \ref{est1}\right) $ the\ following bound for the
transition probabilities $W_{\alpha \beta }^{\gamma }=W_{\alpha \beta
}^{\gamma }\left( \mathbf{Z}^{\prime },\mathbf{Z},\mathbf{Z}_{\ast }\right) $%
, see also Remark $\ref{Rem1}$, may be obtained%
\begin{eqnarray*}
&&W_{\alpha \beta }^{\gamma }\left( \mathbf{Z}^{\prime },\mathbf{Z},\mathbf{Z%
}_{\ast }\right)  \\
&=&\frac{m_{\alpha }m_{\beta }}{m_{\gamma }}\varphi _{\alpha }\left(
I\right) \varphi _{\beta }\left( I_{\ast }\right) \sigma _{\alpha \beta
}^{\gamma }\mathbf{\delta }_{1}\left( I^{\prime }-\widetilde{E}_{\alpha
\beta }^{\prime }\right) \mathbf{\delta }_{3}\left( \boldsymbol{\xi }%
^{\prime }-\mathbf{G}_{\alpha \beta }\right) \mathbf{1}_{E_{\gamma }\geq
K_{\alpha \beta }^{\gamma }} \\
&\leq &C\left( 1+\left\vert \mathbf{g}\right\vert ^{\chi -1}\right) \frac{%
\varphi _{\alpha }\left( I\right) \varphi _{\beta }\left( I_{\ast }\right)
\varphi _{\gamma }\left( I^{\prime }\right) }{\mathcal{E}_{\alpha \beta
}^{\delta ^{\left( \alpha \right) }/2}\left( \mathcal{E}_{\alpha \beta
}^{\ast }\right) ^{\delta ^{\left( \beta \right) }/2}}\mathbf{\delta }%
_{1}\left( I^{\prime }-\widetilde{E}_{\alpha \beta }^{\prime }\right)  \\
&&\times \mathbf{\delta }_{3}\left( \boldsymbol{\xi }^{\prime }-\mathbf{G}%
_{\alpha \beta }\right) \mathbf{1}_{E_{\gamma }\geq K_{\alpha \beta
}^{\gamma }} \\
&\leq &C\left( 1+\left\vert \mathbf{g}\right\vert ^{\chi -1}\right) \varphi
_{\alpha }\left( I\right) \varphi _{\beta }\left( I_{\ast }\right) \left( 
\widetilde{E}_{\alpha \beta }^{\prime }\right) ^{1/2}\mathbf{\delta }%
_{1}\left( I^{\prime }-\widetilde{E}_{\alpha \beta }^{\prime }\right)  \\
&&\times \mathbf{\delta }_{3}\left( \boldsymbol{\xi }^{\prime }-\mathbf{G}%
_{\alpha \beta }\right) \mathbf{1}_{E_{\gamma }\geq K_{\alpha \beta
}^{\gamma }}\text{, where }\mathbf{G}_{\alpha \beta }=\frac{m_{\alpha }%
\boldsymbol{\xi }+m_{\beta }\boldsymbol{\xi }_{\ast }}{m_{\alpha }+m_{\beta }%
}\text{.}
\end{eqnarray*}%
Here the first inequality in the following bounds was applied to obtain the
last bound,%
\begin{eqnarray*}
\frac{\left( \widetilde{E}_{\alpha \beta }^{\prime }\right) ^{\delta
^{\left( \gamma \right) }/2-1}}{\mathcal{E}_{\alpha \beta }^{\delta ^{\left(
\alpha \right) }/2}\left( \mathcal{E}_{\alpha \beta }^{\ast }\right)
^{\delta ^{\left( \beta \right) }/2}} &\leq &C\left( \widetilde{E}_{\alpha
\beta }^{\prime }\right) ^{1/2} \\
&\leq &C\left( 1+\left\vert \mathbf{g}\right\vert \right) \left( 1+\mathbf{1}%
_{\alpha \in \mathcal{I}_{poly}}I\right) \left( 1+\mathbf{1}_{\beta \in 
\mathcal{I}_{poly}}I_{\ast }\right) 
\end{eqnarray*}%
while the second bound implies that%
\begin{eqnarray*}
&&k_{\alpha \beta }^{2\gamma }\left( \mathbf{Z},\mathbf{Z}_{\ast }\right)  \\
&\leq &C\left( 1+\left\vert \mathbf{g}\right\vert ^{\chi -1}\right) \left(
1+\left\vert \mathbf{g}\right\vert \right) \left( 1+\mathbf{1}_{\alpha \in 
\mathcal{I}_{poly}}I\right) \left( 1+\mathbf{1}_{\beta \in \mathcal{I}%
_{poly}}I_{\ast }\right) M_{\alpha }^{1/2}M_{\beta \ast }^{1/2} \\
&\leq &C\left( \left\vert \mathbf{g}\right\vert +\left\vert \mathbf{g}%
\right\vert ^{\chi -1}\right) M_{\alpha }^{1/4}M_{\beta \ast }^{1/4}\text{.}
\end{eqnarray*}%
Noting that%
\begin{equation*}
m_{\alpha }\frac{\left\vert \boldsymbol{\xi }\right\vert ^{2}}{2}+m_{\beta }%
\frac{\left\vert \boldsymbol{\xi }_{\ast }\right\vert ^{2}}{2}=\frac{%
m_{\alpha }m_{\beta }}{2\left( m_{\alpha }+m_{\beta }\right) }\left\vert 
\mathbf{g}\right\vert ^{2}+\frac{m_{\alpha }+m_{\beta }}{2}\left\vert 
\mathbf{G}_{\alpha \beta }\right\vert ^{2}\text{,}
\end{equation*}%
we obtain that%
\begin{eqnarray*}
\left( k_{\alpha \beta }^{2\gamma }\left( \mathbf{Z},\mathbf{Z}_{\ast
}\right) \right) ^{2} &\leq &C\left( \left\vert \mathbf{g}\right\vert
+\left\vert \mathbf{g}\right\vert ^{\chi -1}\right) ^{2}M_{\alpha
}^{1/2}M_{\beta \ast }^{1/2} \\
&\leq &C\left( \left\vert \mathbf{g}\right\vert ^{2}+\left\vert \mathbf{g}%
\right\vert ^{2\chi -2}\right) \exp \left( -\frac{m_{\alpha }m_{\beta }}{%
4\left( m_{\alpha }+m_{\beta }\right) k_{B}T}\left\vert \mathbf{g}%
\right\vert ^{2}\right)  \\
&&\times \exp \left( -\frac{m_{\alpha }+m_{\beta }}{4k_{B}T}\left\vert 
\mathbf{G}_{\alpha \beta }\right\vert ^{2}\right) e^{-\mathbf{1}_{\alpha \in 
\mathcal{I}_{poly}}I/\left( 2k_{B}T\right) }e^{-\mathbf{1}_{\beta \in 
\mathcal{I}_{poly}}I_{\ast }/\left( 2k_{B}T\right) }.
\end{eqnarray*}%
Hence, $k_{\alpha \beta }^{2\gamma }\left( \mathbf{Z},\mathbf{Z}_{\ast
}\right) \in L^{2}\left( d\mathbf{Z}d\mathbf{Z}_{\ast }\right) $, since%
\begin{eqnarray*}
&&\int_{\mathcal{Z}_{\alpha }\times \mathcal{Z}_{\beta }}\left( k_{\alpha
\beta }^{2\gamma }\left( \mathbf{Z},\mathbf{Z}_{\ast }\right) \right) ^{2}d%
\mathbf{Z}d\mathbf{Z}_{\ast } \\
&\leq &C\int_{\mathbb{R}^{3}}\left( \left\vert \mathbf{g}\right\vert
^{2}+\left\vert \mathbf{g}\right\vert ^{2\chi -2}\right) \exp \left( -\frac{%
m_{\alpha }m_{\beta }}{4\left( m_{\alpha }+m_{\beta }\right) k_{B}T}%
\left\vert \mathbf{g}\right\vert ^{2}\right) d\mathbf{g} \\
&&\times \int_{\mathbb{R}^{3}}\exp \left( -\frac{m_{\alpha }+m_{\beta }}{%
4k_{B}T}\left\vert \mathbf{G}_{\alpha \beta }\right\vert ^{2}\right) d%
\mathbf{G}_{\alpha \beta }\int_{0}^{\infty }e^{-I/\left( 2k_{B}T\right)
}dI\int_{0}^{\infty }e^{-I_{\ast }/\left( 2k_{B}T\right) }dI_{\ast } \\
&=&C\int_{0}^{\infty }\left( R^{4}+R^{2\chi }\right) \exp \left(
-R^{2}\right) d\mathbf{R}\int_{0}^{\infty }R^{2}\exp \left( -R^{2}\right) d%
\mathbf{R}=C\text{.}
\end{eqnarray*}%
Therefore,%
\begin{equation*}
K_{\alpha \beta }^{2\gamma }=\int_{\mathcal{Z}_{\beta }}k_{\alpha \beta
}^{2\gamma }\left( \mathbf{Z},\mathbf{Z}_{\ast }\right) h_{\beta \ast }d%
\mathbf{Z}_{\ast }
\end{equation*}%
are Hilbert-Schmidt integral operators and as such compact on $L^{2}\left( d%
\mathbf{Z}_{\ast }\right) $, see, e.g., Theorem 7.83 in \cite{RenardyRogers}%
, for $\left( \gamma ,\alpha ,\beta \right) \in \mathcal{C}$.

\textbf{III. Compactness of }$K_{\alpha \gamma }^{3\beta }=\int_{\mathcal{Z}%
_{\beta }}k_{\alpha \gamma }^{3\beta }\left( \mathbf{Z},\mathbf{Z}_{\ast
}\right) h_{\beta \ast }d\mathbf{Z}_{\ast }$ for $\left( \beta ,\alpha
,\gamma \right) \in \mathcal{C}$ follows directly by \textbf{I.}, since $%
k_{\alpha \gamma }^{3\beta }\left( \mathbf{Z},\mathbf{Z}_{\ast }\right)
=k_{\alpha \gamma }^{1\beta }\left( \mathbf{Z}_{\ast },\mathbf{Z}\right) $
for all $\left( \beta ,\alpha ,\gamma \right) \in \mathcal{C}$.

Concluding, the operator 
\begin{eqnarray*}
K_{c} &=&2(K_{c1},...,K_{cs}) \\
&=&2\sum\limits_{\beta .\gamma =1}^{s}(K_{\beta \gamma }^{11}+K_{1\beta
}^{2\gamma }-K_{1\gamma }^{3\beta },K_{\beta \gamma }^{12}+K_{2\beta
}^{2\gamma }-K_{2\gamma }^{3\beta },...,K_{\beta \gamma }^{1s}+K_{s\beta
}^{2\gamma }-K_{s\gamma }^{3\beta })\text{,}
\end{eqnarray*}%
is a compact self-adjoint operator on $\mathcal{\mathfrak{h}}$.
\end{proof}

\section{Bounds on the collision frequency \label{PT2}}

This section concerns the proof of Theorem \ref{Thm2}.

Note that throughout the proof, $C>0$ will denote a generic positive
constant. Moreover, remind that $\varphi _{\alpha }\left( I\right)
=I^{\delta ^{\left( \alpha \right) }/2-1}$ for $\alpha \in \mathcal{I}$
below.

\begin{proof}
Under assumption $\left( \ref{e1a}\right) $ there \cite{Be-24a} exist
positive numbers $\nu _{-}$ and $\widetilde{\nu }_{+}$, where $0<\nu _{-}<%
\widetilde{\nu }_{+}$, such that for any $\alpha \in \mathcal{I}$%
\begin{equation*}
\nu _{-}\left( 1+\left\vert \boldsymbol{\xi }\right\vert +\mathbf{1}_{\alpha
\in \mathcal{I}_{poly}}\sqrt{I}\right) ^{1-\eta }\leq \nu _{m\alpha }\leq 
\widetilde{\nu }_{+}\left( 1+\left\vert \boldsymbol{\xi }\right\vert +%
\mathbf{1}_{\alpha \in \mathcal{I}_{poly}}\sqrt{I}\right) ^{1-\eta }\text{.}
\end{equation*}%
Hence,%
\begin{equation*}
\nu _{\alpha }=\nu _{m\alpha }+\nu _{c\alpha }\geq \nu _{m\alpha }\geq \nu
_{-}\left( 1+\left\vert \boldsymbol{\xi }\right\vert +\mathbf{1}_{\alpha \in 
\mathcal{I}_{poly}}\sqrt{I}\right) ^{1-\eta }
\end{equation*}%
for all $\alpha \in \mathcal{I}$.

Noting that $\left( \ref{df1}\right) $, under assumption $\left( \ref{e1}%
\right) $, the collision frequencies $\nu _{c\alpha }$ $\left( \ref{cf1}%
\right) $ equal%
\begin{equation*}
\nu _{c\alpha }=\frac{1}{\varphi _{\alpha }\left( I\right) }\sum_{\beta
,\gamma }\int\limits_{\mathcal{Z}_{\beta }\mathcal{\times Z}_{\gamma
}}W_{\beta \gamma }^{\alpha }\left( \mathbf{Z},\mathbf{Z}_{\ast },\mathbf{Z}%
^{\prime }\right) +\frac{2M_{\gamma }^{\prime }}{\varphi _{\gamma }\left(
I^{\prime }\right) }W_{\alpha \gamma }^{\beta }\left( \mathbf{Z}_{\ast },%
\mathbf{Z},\mathbf{Z}^{\prime }\right) \,d\mathbf{Z}_{\ast }d\mathbf{Z}%
^{\prime }
\end{equation*}%
for any $\alpha \in \mathcal{I}$. Here%
\begin{eqnarray*}
&&W_{\beta \gamma }^{\alpha }\left( \mathbf{Z},\mathbf{Z}_{\ast },\mathbf{Z}%
^{\prime }\right)  \\
&=&\varphi _{\beta }\left( I_{\ast }\right) \varphi _{\gamma }\left(
I^{\prime }\right) \frac{\sigma _{\beta \gamma }^{\alpha }}{\left\vert 
\boldsymbol{\xi }_{\ast }-\boldsymbol{\xi }^{\prime }\right\vert }\mathbf{1}%
_{\Delta _{\beta \gamma }^{\alpha }\mathcal{E}>0}\mathbf{\delta }_{1}\left(
\left\vert \boldsymbol{\xi }_{\ast }-\boldsymbol{\xi }^{\prime }\right\vert -%
\sqrt{\widetilde{\Delta }_{\beta \gamma }^{\alpha }\mathcal{E}}\right)  \\
&&\times \mathbf{\delta }_{3}\left( \boldsymbol{\xi }-\frac{m_{\beta }%
\boldsymbol{\xi }_{\ast }+m_{\gamma }\boldsymbol{\xi }^{\prime }}{m_{\alpha }%
}\right) \mathbf{1}_{E_{\alpha }\geq K_{\beta \gamma }^{\alpha }} \\
&=&C_{\beta \gamma }^{\alpha }\frac{\varphi _{\alpha }\left( I\right)
\varphi _{\beta }\left( I_{\ast }\right) \varphi _{\gamma }\left( I^{\prime
}\right) }{\mathcal{E}_{\beta \gamma }^{\delta ^{\left( \beta \right)
}/2}\left( \mathcal{E}_{\beta \gamma }^{\ast }\right) ^{\delta ^{\left(
\gamma \right) }/2}}\frac{1}{\left\vert \boldsymbol{\xi }_{\ast }-%
\boldsymbol{\xi }^{\prime }\right\vert E_{\alpha }^{\eta /2}}\mathbf{1}_{%
\widehat{\Delta }_{\beta \gamma }^{\alpha }\mathcal{E}>0} \\
&&\times \mathbf{\delta }_{1}\left( \left\vert \boldsymbol{\xi }_{\ast }-%
\boldsymbol{\xi }^{\prime }\right\vert -\sqrt{\widehat{\Delta }_{\beta
\gamma }^{\alpha }\mathcal{E}}\right) \mathbf{\delta }_{3}\left( \boldsymbol{%
\xi }-\frac{m_{\beta }\boldsymbol{\xi }_{\ast }+m_{\gamma }\boldsymbol{\xi }%
^{\prime }}{m_{\alpha }}\right) \mathbf{1}_{E_{\alpha }\geq K_{\beta \gamma
}^{\alpha }} \\
&&\text{with }\widehat{\Delta }_{\beta \gamma }^{\alpha }\mathcal{E}=\frac{%
2m_{\alpha }}{m_{\beta }m_{\gamma }}\left( I-I_{\ast }\mathbf{1}_{\beta \in 
\mathcal{I}_{poly}}-I^{\prime }\mathbf{1}_{\gamma \in \mathcal{I}%
_{poly}}-\Delta _{\beta \gamma }^{\alpha }\varepsilon _{0}\right) \text{,} \\
&&\text{where }\Delta _{\beta \gamma }^{\alpha }\varepsilon _{0}=\varepsilon
_{\beta 0}+\varepsilon _{\gamma 0}-\varepsilon _{\alpha 0}\text{.}
\end{eqnarray*}%
Noting that%
\begin{equation*}
\frac{\sqrt{\widehat{\Delta }_{\beta \gamma }^{\alpha }\mathcal{E}}}{%
E_{\alpha }^{\eta /2}}\leq CI^{\left( 1-\eta \right) /2}\leq C\left(
1+\left\vert \boldsymbol{\xi }\right\vert +\sqrt{I}\right) ^{1-\eta }\text{,}
\end{equation*}%
we obtain that%
\begin{eqnarray*}
&&\int_{\mathcal{Z}_{\beta }\times \mathcal{Z}_{\gamma }}\frac{W_{\beta
\gamma }^{\alpha }\left( \mathbf{Z},\mathbf{Z}_{\ast },\mathbf{Z}^{\prime
}\right) }{\varphi _{\alpha }\left( I\right) }\,d\mathbf{Z}_{\ast }d\mathbf{Z%
}^{\prime } \\
&=&\int_{\mathbb{R}^{3}\times \left( \mathbb{R}_{+}\right) ^{3}\times 
\mathbb{S}^{2}}\frac{W_{\beta \gamma }^{\alpha }\left( \mathbf{Z},\mathbf{Z}%
_{\ast },\mathbf{Z}^{\prime }\right) }{\varphi _{\alpha }\left( I\right) }%
\left\vert \boldsymbol{\xi }_{\ast }-\boldsymbol{\xi }^{\prime }\right\vert
^{2} \\
&&\times \mathbf{1}_{I_{\ast }\leq \mathcal{E}_{\beta \gamma }}\mathbf{1}%
_{I^{\prime }\leq \mathcal{E}_{\beta \gamma }^{\ast }}d\left\vert 
\boldsymbol{\xi }_{\ast }-\boldsymbol{\xi }^{\prime }\right\vert d\frac{%
m_{\beta }\boldsymbol{\xi }_{\ast }+m_{\gamma }\boldsymbol{\xi }^{\prime }}{%
m_{\alpha }}d\boldsymbol{\sigma }dI_{\ast }dI^{\prime } \\
&=&4\pi \int_{\left( \mathbb{R}_{+}\right) ^{2}}C_{\beta \gamma }^{\alpha }%
\frac{\varphi _{\beta }\left( I_{\ast }\right) \varphi _{\gamma }\left(
I^{\prime }\right) }{\mathcal{E}_{\beta \gamma }^{\delta ^{\left( \beta
\right) }/2}\left( \mathcal{E}_{\beta \gamma }^{\ast }\right) ^{\delta
^{\left( \gamma \right) }/2}}\frac{\sqrt{\widehat{\Delta }_{\beta \gamma
}^{\alpha }\mathcal{E}}}{E_{\alpha }^{\eta /2}}\mathbf{1}_{\widehat{\Delta }%
_{\beta \gamma }^{\alpha }\mathcal{E}>0}\mathbf{1}_{E_{\alpha }\geq K_{\beta
\gamma }^{\alpha }}\mathbf{1}_{I_{\ast }\leq \mathcal{E}_{\beta \gamma }}%
\mathbf{1}_{I^{\prime }\leq \mathcal{E}_{\beta \gamma }^{\ast }}dI_{\ast
}dI^{\prime } \\
&\leq &C\left( 1+\left\vert \boldsymbol{\xi }\right\vert +\sqrt{I}\right)
^{1-\eta }\int_{0}^{\mathcal{E}_{\beta \gamma }}\frac{I_{\ast }^{\delta
^{\left( \beta \right) }/2-1}}{\mathcal{E}_{\beta \gamma }^{\delta ^{\left(
\beta \right) }/2}}dI_{\ast }\int_{0}^{\mathcal{E}_{\beta \gamma }^{\ast }}%
\frac{\left( I^{\prime }\right) ^{\delta ^{\left( \gamma \right) }/2-1}}{%
\left( \mathcal{E}_{\beta \gamma }^{\ast }\right) ^{\delta ^{\left( \gamma
\right) }/2}}dI^{\prime } \\
&=&C\left( 1+\left\vert \boldsymbol{\xi }\right\vert +\sqrt{I}\right)
^{1-\eta }\text{.}
\end{eqnarray*}%
Furthermore,%
\begin{eqnarray*}
&&W_{\alpha \gamma }^{\beta }\left( \mathbf{Z}_{\ast },\mathbf{Z},\mathbf{Z}%
^{\prime }\right)  \\
&=&\frac{m_{\alpha }m_{\gamma }}{m_{\beta }}\varphi _{\alpha }\left(
I\right) \varphi _{\gamma }\left( I^{\prime }\right) \sigma _{\alpha \gamma
}^{\beta }\mathbf{\delta }_{1}\left( I_{\ast }-\widetilde{E}_{\alpha \gamma
}^{\prime }\right) \mathbf{\delta }_{3}\left( \boldsymbol{\xi }_{\ast }-%
\frac{m_{\alpha }\boldsymbol{\xi }+m_{\gamma }\boldsymbol{\xi }^{\prime }}{%
m_{\beta }}\right) \mathbf{1}_{E_{\beta }\geq K_{\alpha \gamma }^{\beta }} \\
&=&C_{\alpha \gamma }\frac{m_{\alpha }m_{\gamma }}{m_{\beta }}\frac{\varphi
_{\alpha }\left( I\right) \varphi _{\beta }\left( I_{\ast }\right) \varphi
_{\gamma }\left( I^{\prime }\right) }{\mathcal{E}_{\alpha \gamma }^{\delta
^{\left( \alpha \right) }/2}\left( \mathcal{E}_{\alpha \gamma }^{\ast
}\right) ^{\delta ^{\left( \gamma \right) }/2}E_{\beta }^{\eta /2}}\mathbf{%
\delta }_{1}\left( I_{\ast }-\widetilde{E}_{\alpha \gamma }^{\prime }\right) 
\mathbf{\delta }_{3}\left( \boldsymbol{\xi }_{\ast }-\frac{m_{\alpha }%
\boldsymbol{\xi }+m_{\gamma }\boldsymbol{\xi }^{\prime }}{m_{\beta }}\right) 
\mathbf{1}_{E_{\beta }\geq K_{\alpha \gamma }^{\beta }}\text{.}
\end{eqnarray*}%
Now noting that%
\begin{eqnarray*}
&&\frac{\left( \widetilde{E}_{\alpha \gamma }^{\prime }\right) ^{\delta
^{\left( \beta \right) }/2-1}}{\mathcal{E}_{\alpha \gamma }^{\delta ^{\left(
\alpha \right) }/2}\left( \mathcal{E}_{\alpha \gamma }^{\ast }\right)
^{\delta ^{\left( \gamma \right) }/2}}\frac{1}{E_{\beta }^{\eta /2}} \\
&\leq &C\left( \widetilde{E}_{\alpha \gamma }^{\prime }\right) ^{\left(
1-\eta \right) /2}\leq C\left( 1+\left\vert \boldsymbol{\xi }\right\vert +%
\mathbf{1}_{\alpha \in \mathcal{I}_{poly}}\sqrt{I}\right) ^{1-\eta }\left( 1+%
\mathbf{1}_{\gamma \in \mathcal{I}_{poly}}I^{\prime }\right) \left(
1+\left\vert \boldsymbol{\xi }^{\prime }\right\vert ^{2}\right) \text{,}
\end{eqnarray*}%
we obtain that%
\begin{eqnarray*}
&&\int_{\mathcal{Z}_{\beta }\times \mathcal{Z}_{\gamma }}\frac{M_{\gamma
}^{\prime }}{\varphi _{\alpha }\left( I\right) \varphi _{\gamma }\left(
I^{\prime }\right) }W_{\alpha \gamma }^{\beta }\left( \mathbf{Z}_{\ast },%
\mathbf{Z},\mathbf{Z}^{\prime }\right) \,d\mathbf{Z}_{\ast }d\mathbf{Z}%
^{\prime } \\
&\leq &\int_{\left( \mathbb{R}_{+}\right) ^{2}\times \mathbb{S}%
^{2}}C_{\alpha \gamma }\frac{m_{\alpha }m_{\gamma }}{m_{\beta }}\frac{%
M_{\gamma }^{\prime }\left( \widetilde{E}_{\alpha \gamma }^{\prime }\right)
^{\delta ^{\left( \beta \right) }/2-1}}{\mathcal{E}_{\alpha \gamma }^{\delta
^{\left( \alpha \right) }/2}\left( \mathcal{E}_{\alpha \gamma }^{\ast
}\right) ^{\delta ^{\left( \gamma \right) }/2}\left( E_{\alpha \gamma
}^{\prime }\right) ^{\eta /2}}\left\vert \boldsymbol{\xi }^{\prime
}\right\vert ^{2}\mathbf{1}_{I^{\prime }\leq \mathcal{E}_{\alpha \gamma
}^{\ast }}\,d\left\vert \boldsymbol{\xi }^{\prime }\right\vert d\boldsymbol{%
\sigma }dI^{\prime } \\
&\leq &C\left( 1+\left\vert \boldsymbol{\xi }\right\vert +\mathbf{1}_{\alpha
\in \mathcal{I}_{poly}}\sqrt{I}\right) ^{1-\eta }\int_{0}^{\infty }\left(
1+\left\vert \boldsymbol{\xi }^{\prime }\right\vert ^{2}\right)
e^{-m_{\gamma }\left\vert \boldsymbol{\xi }^{\prime }\right\vert
^{2}/(2k_{B}T)}\left\vert \boldsymbol{\xi }^{\prime }\right\vert
^{2}d\left\vert \boldsymbol{\xi }^{\prime }\right\vert  \\
&&\times \int_{0}^{\infty }\left( 1+I^{\prime }\right) \left( I^{\prime
}\right) ^{\delta ^{\left( \gamma \right) }/2-1}e^{-I^{\prime
}/(k_{B}T)}dI^{\prime } \\
&=&C\left( 1+\left\vert \boldsymbol{\xi }\right\vert +\mathbf{1}_{\alpha \in 
\mathcal{I}_{poly}}\sqrt{I}\right) ^{1-\eta }\text{.}
\end{eqnarray*}%
Hence, there is a positive constant $\nu _{+}>0$, such that 
\begin{equation*}
\nu _{\alpha }=\nu _{m\alpha }+\nu _{c\alpha }\leq \nu _{+}\left(
1+\left\vert \boldsymbol{\xi }\right\vert +\mathbf{1}_{\alpha \in \mathcal{I}%
_{poly}}\sqrt{I}\right) ^{1-\eta }
\end{equation*}%
for all $\alpha \in \mathcal{I}$.
\end{proof}

\end{document}